\documentclass[reqno,english,11pt]{amsart}
\usepackage[margin=1in]{geometry}
\usepackage[latin9]{inputenc}
\usepackage{textcomp}
\usepackage{amstext}
\usepackage{amsthm}
\usepackage{amssymb}
\usepackage{hyperref}
\usepackage{cite}
\usepackage{url}
\usepackage{xcolor}
\makeatletter

\newcommand{\lyxmathsym}[1]{\ifmmode\begingroup\def\b@ld{bold}
  \text{\ifx\math@version\b@ld\bfseries\fi#1}\endgroup\else#1\fi}

\numberwithin{equation}{section}
\numberwithin{figure}{section}

\theoremstyle{plain}

\newtheorem{thm}{\protect\theoremname}[section]
  \theoremstyle{plain}
  \newtheorem{cor}[thm]{\protect\corollaryname}
  \theoremstyle{remark}
  \newtheorem{rem}[thm]{\protect\remarkname}
  \theoremstyle{plain}
  \newtheorem{lem}[thm]{\protect\lemmaname}
  \theoremstyle{plain}
  \newtheorem{prop}[thm]{\protect\propositionname}
\theoremstyle{definition}
\newtheorem{defn}[thm]{\protect\definitionname}
\theoremstyle{plain}
\makeatother

\def\wtF{\widetilde{{\mathcal{F}}}}

\def\jx{\langle x \rangle}
\def\jy{\langle y \rangle}
\def\jk{\langle k \rangle}

\def\wt{\widetilde}

\def\R{\mathbb{R}}

\def\K{\mathcal{K}}

\usepackage{babel}
\providecommand{\corollaryname}{Corollary}
\providecommand{\definitionname}{Definition}
\providecommand{\lemmaname}{Lemma}
\providecommand{\remarkname}{Remark}
\providecommand{\theoremname}{Theorem}
\providecommand{\propositionname}{Proposition}
\begin{document}
\title[Cubic NLS with soliton]{Long-time dynamics of small solutions to $1d$ cubic nonlinear Schr\"odinger
equations with a trapping potential}
\author[G.\,\,Chen]{Gong Chen}
\address{School of Mathematics, Georgia Institute of Technology, Atlanta, GA 30332-0160}
\email{gc@math.gatech.edu}
\date{June 17, 2021}
\begin{abstract}
In this paper, we analyze the long-time dynamics of small solutions
to the $1d$ cubic nonlinear Schr\"odinger equation (NLS)
with a trapping potential. We show that every small solution will
decompose into a small solitary wave and a radiation term which exhibits the
modified scattering.
Our analysis also establishes the long-time
behavior of solutions to a perturbation of the integrable cubic
NLS with the appearance of solitons.
\end{abstract}

\maketitle
	\setcounter{tocdepth}{1}

\section{Introduction}

In this paper, we study the long-time behavior of small-norm solutions
to the cubic nonlinear Schr\"odinger equation (NLS) in $1$-d,
\begin{equation}
i\partial_{t}u-\partial_{xx}u+Vu=\lambda\left|u\right|^{2}u,\,\,\lambda=\pm1,\qquad u(t=0,x)=u_{0}(x),\label{eq:NLS}
\end{equation}
perturbed by a generic potential $V$ such that the linear Schr\"odinger operator $-\partial_{xx}+V$ has  one negative eigenvalue.
Since we are only interested in small solutions, the sign of $\lambda$
does not make any difference so we set $\lambda=1$.


%
The main goals of this paper are the following two points:
\begin{itemize}
    \item First, we continue to explore space-time resonance analysis with distorted Fourier transforms, while incorporating dispersive analysis to handle the influence of the discrete spectrum of the linear operator.
    \item Second, we study asymptotics of solutions to a perturbation of an integrable system with the appearance of solitons.
\end{itemize}
The key features of the current work are that  we obtain \emph{global} estimates and a \emph{precise} description of the solution with the appearance of solitons under slow dispersion and critical scattering.  

\subsection{Background and previous results}
The study on  the dynamics of solutions to the nonlinear Schr\"odinger equation with a potential has a long history in the mathematical physics.  Without trying to give a complete list here, we refer to the book by Cazenave \cite{Ca},  and papers by Bronski-Jerrard \cite{BJ},  Fr\"ohlich-Gustafson-Jonsson-Sigal \cite{FGJS1}, Holmer-Zworski \cite{HZ}, 
Soffer-Weinstein \cite{SW1,SW2,SW3}, Tsai-Yau \cite{TY} and references therein. 


Focusing on the $1$d model \eqref{eq:NLS} above, we recall that sufficiently regular solutions of \eqref{eq:NLS} conserve the $L^{2}$ norm
\[M(u) := \int\left|u\right|^{2}\,dx \]
and the total energy (Hamiltonian):
\[
H\left(u\right) := \int\big(\frac{1}{2}\left|\partial_xu\right|^{2}+\frac{1}{2}V\left|u\right|^{2}-\frac{1}{4}\left|u\right|^{4}\big) \,dx.
\]
The Cauchy problem for \eqref{eq:NLS} with $V=0$ - we will refer to this as the ``free'' or ``flat'' case -
is globally well-posed in $L^{2}$, see for example
Cazenave-Weissler \cite{CW}. 
Our main interests are the global-in-time bounds and asymptotic behavior as $|t|\rightarrow \infty$. 
The main feature of the cubic nonlinearity is its criticality with respect to scattering:
linear solutions of the Schr\"odinger equation decay at best like $|t|^{-1/2}$ in $L^\infty_x$, 
so that, when evaluating the nonlinearity on linear solutions, one see that $|u|^{2}u \sim |t|^{-1}u$;
the non-integrability of $|t|^{-1}$ 
results in a ``Coulomb''-type contribution of the nonlinear terms.

\subsubsection{Long-time asymptotics}
In the case $V=0$ the problem is well understood. 
Solutions of \eqref{eq:NLS} with $V=0$ and initial data $u|_{t=0}\in H^{1}\cap L^{2}(x^{2}dx)$
(i.e., bounded energy and variance) are known to exhibit the modified scattering as time goes to infinity:
they decay at the same rate of linear solutions
but their asymptotic behavior differs from linear solutions by a logarithmic phase correction. 
Using the complete integrability this was proven in the seminal work of Deift-Zhou \cite{DZ}; see also \cite{DZ2} on nonlinear perturbations of the defocusing cubic NLS. 
Without making use of the complete integrability, and
restricting the analysis to small solutions, proofs of modified scattering were given by Ozawa \cite{O},  Hayashi-Naumkin \cite{HN}, Lindblad-Soffer \cite{LS},
Kato-Pusateri \cite{KP} and Ifrim-Tataru \cite{IT}.


Recently, the results above for small solutions have been extended to the full problem with potential \eqref{eq:NLS}
in the works of Naumkin \cite{N}, Delort \cite{Del} and Germain-Pusateri-Rousset \cite{GPR}.
These works treat 
potentials of sufficient regularity and decay, and establish modified scattering results similar to  those obtained by the above mentioned papers in the flat case.
The  work of Masaki-Murphy-Segata \cite{MMS} treats the special case of a delta potential; 
compared to the works cited above, in \cite{MMS} the potential has no regularity
but, on the other hand, it has the advantage of being explicitly calculable. In the recent work of Chen-Pusateri \cite{CP}, after exploring refined dispersive estimates and bounds for pseudo-differential operators, the condition on the potential is relaxed to be in some weighted $L^1$ spaces. This in particular recovers the $\delta$ potential case by a limiting argument. In all the works mentioned above, the potential is assumed to be generic, or some symmetry conditions are imposed if the potential is not generic. More importantly, in these works, the underlying linear Schr\"odinger operators are assumed to have \emph{no} eigenvalues.  In a recent work by
Lindblad-Luhrmann-Schlag-Soffer \cite{LLSS}, the authors performed the nonlinear analysis in the Klein-Gordon problem that allows the linear operator to have a zero resonance. From a different perspective, with virial type arguments, in the work by Mart\'inez \cite{MM}, certain decay estimates on compact regions are established for  odd solutions with small\footnote{For the  defocusing cubic NLS in the free case or perturbed by a small potential, this smallness condition can be dropped. For details, see \cite[Theorem 1.1] {MM}. } energy norms to the NLS both in the free case and perturbed by an even potential without assumptions on the spectrum.


In this paper, the potential is assumed to have one negative eigenvalue. The appearance of the eigenfunction associated with this negative eigenvalue will produce solitary waves in the nonlinear setting. So in the long-time asymptotics in this setting, one has to take solitary waves into account.   This in particular can be regarded as a special case of the perturbation of the integrable cubic NLS with the appearance of solitons. As mentioned above, the defocusing nonlinear perturbation of the defocusing NLS is analyzed in Deift-Zhou \cite{DZ2}. One basic intuition in that setting is that higher order nonlinearities will enjoy better decay rates. Superficially, one expects that the total deviation from the integrable structure introduced by this higher order perturbation is controllable, whence the global solution will be close to the solution to the cubic NLS.  But the appearance of solitons will destroy the decay of solutions in the focusing problem. This is significantly different from the defocusing problem.  The perturbation problem with the appearance of solitons is also mentioned in the survey by Deift, see Problem 7 in \cite{Dei}. Although the model considered here is a linear perturbation, we believe that this will  shed  light on the nonlinear settings.

\subsubsection{Stability of solitons}

The literature on the stability solitons is extensive and without trying to be exhaustive, we refer
to the survey and the monograph by  Tao \cite{Tao,Tao2}, and references therein.


The main focus in this paper is the asymptotic stability of solitons. For the NLS problem, in higher dimensions, the radiation terms always have better decay rates,  so there are stronger tools can be used, for example, integrable  dispersive decays and Strichartz estimates.  In $3$-d, small solutions to the model \eqref{eq:NLS} was studied in Gustafson-Nakanishi-Tsai\cite{GNT} using Strichartz estimates in the energy space. For more developments on the dynamics of solutions in this model, see the work of Nakanishi \cite{N2} and references therein.  We also refer Soffer-Weinstein \cite{SW2,SW3} for the study of the asymptotic stability of solitons bifurcated from eigenfunctions associated with the linear operator. For the dynamics of solutions to the NLS with a potential which has  multiple negative eigenvalues, see, for example,  works of Soffer-Weinstein  \cite{SW1} Tsai-Yau \cite{TY} and Tsai \cite{Ts}. As for the study of the asymptotic stability of solitons to the purely power-type NLS, we refer to Beceanu \cite{Bec1}, Cuccagna \cite{Cu}, Nakanishi-Schlag \cite{NS},  Schlag \cite{Sch1} and references therein. Finally, we also refer to the surveys by Cuccagna \cite{Cu1}, Cuccagan-Maeda \cite{CuMa}, Schlag \cite{Sch} and Weinstein \cite{Wein1} for more complete pictures on the long-time dynamics and the asymptotic stability of solitons for the NLS.


Focusing on the one-dimension NLS, the asymptotic stability of solitons with supercritical nonlinearties, we refer to the works by Buslaev-Perelman \cite{BP,BP2}, Buslaev-Sulem \cite{BS} Krieger-Schlag \cite{KS}, Mizumachi \cite{Mi} and  Masaki-Murphy-Segata\cite{MMS2} in various different settings. The key point is that higher power nonlinearities always  allow us to bootstrap better dispersive decay rates and smoothing estimates.  Stability problems with low power nonlinearities recently have attracted a lot of attentions. The weak power nonlinearities always result in the orbital stability of solitons, see for example Weinstein \cite{Wein},  but to establish the asymptotic stability is much harder due to the slow decay rate in time of the error term. Using the integrability, the asymptotic stability for solitons of the flat cubic NLS was proved by Cuccagna-Pelinovsky \cite{CuP}. With techniques from the compete integrability, one can actually obtain stronger results on the soliton resolution which implies the  asymptotic stability. We refer to the introductions  of Borghese-Jenkins-McLaughlin \cite{BJM16} and  Chen-Liu-Lu \cite{CLL} for  more comprehensive surveys.  To handle the asymptotic stability under weak nonlinear power settings using PDE techniques is a very challenging problem. 
In the context of Klein-Gordon problems, using virial estimates,  Kowalczyk-Martel-Mu\~noz \cite{KMM,KMM1} successfully shown the asymptotic stability of kinks and solitons under some symmetry assumptions. Also see the recent extension to the moving kinks setting by  Kowalczyk-Martel-Mu\~noz-Van Den Bosch \cite{KMMV}.  Combing virial estimates and the commutator method, Cuccagan-Maeda \cite{CuM} shown the asymptotic stability of small solitons for the cubic NLS with a delta trapping potential. Of course  this list is far from being exhaustive. We also refer to references  from the works cited above for more details.  In these works, the radiation terms are shown to decay in some localized norms or in the integral sense, i.e., the time integrals of the radiation terms measured in some localized norms are finite. In all papers mentioned above, the analysis is carried out in the natural translation invariant energy space, however no explicit decay rates nor descriptions of radiation terms are given.\footnote{For the critical and subcritical scattering problems, in the energy space, one might not expect to derive the precise decay rate.} 
 Our main interest in this paper is establish  the  asymptotic stability with  detailed descriptions of the radiation terms after paying the price of  weights for the initial data. Unlike the supercritical scattering problem, due to the criticality of the nonlinearity in this problem, the long-time asymptotics of the radiation term is nonlinear.  The weights for the initial data are natural in order to capture the modified scattering phenomenon.  To obtain the detailed descriptions, one has to use many refined linear estimates.  For the mKdV equation, the asymptotic stability of the soliton  with a detailed description of the radiation term was established in Germain-Pusateri-Rousset \cite{GPR2}. Unlike the NLS problem, the key feature of the mKdV model is that the dynamics of the soliton and the radiation are basically decoupled. Recently, under symmetry assumptions, the full description of the radiation term from the  asymptotic stability of some kinks whose associated linearized operators have no eigenvalues nor resonances, in the double sine-Gordon model was given by Germain-Pusateri \cite{GP}. We also refer to Delort-Masmoudi \cite{DM} for the results on long-time dispersive estimates of the perturbation around a kink in the $\phi^4$ model.


The goal of this paper is to continue  exploring the space-time resonance analysis, distorted Fourier transforms, normal form analysis (integration by parts in time), modulations and dispersive estimates with the influence of the negative eigenvalue associated with linear operator. In our current setting, the inhomogeneous terms for the equation of the radiation term consist of first order perturbations, quadratic terms and the cubic term.
The cubic term and its structure are analyzed carefully in Chen-Pusateri \cite{CP}. In this paper, to obtain the full detailed description of solutions, the most delicate points  are to analyze the quadratic terms and first order perturbations  given by solitary waves. To deal with the first order perturbations
and quadratic terms, we explore homogeneous and inhomogeneous smoothing estimates which resemble Kato local smoothing estimates. Some of these estimates also appear in Mizumachi \cite{Mi} and Krieger-Nakanishi-Schlag \cite{KNS} for supercritical problems.  In this cubic problem, the decay rate given by the radiation is slow, we need more refined smoothing estimates. Moreover, to overcome the slow decay in lower order perturbations, we also need to employ the Fourier transform in time and several integration by parts in time. 

\subsection{Main result}

In this subsection, we state the main result in this paper. We begin
with our assumption on potentials and then introduce some basic notations.

\subsubsection*{Assumptions on potentials:}

For the Schr\"odinger operator, we assume that  $V$ is smooth,  decays
exponentially such that
\begin{equation}
H=-\partial_{xx}+V\label{eq:schOp}
\end{equation}
has only one negative eigenvalue $-\rho^{2}<0$ with the associated eigenfunction
$\phi$. In other words,
\begin{equation}
\left(-\partial_{xx}+V\right)\phi=-\rho^{2}\phi.\label{eq:negeig}
\end{equation}
In particular,  the potential is generic, i.e., there is no resonance at $0$, see Definition \ref{def:generic}. 
\begin{rem}
We remark that here the regularity and the decay conditions imposed
on the potential are just for the sake of convenience. The potential
in some weighted $L^{1}$ spaces would be enough. As the limiting case,
a delta potential can also work. The essential point here is the generic
condition on the absence of the resonance, and there is only one negative
eigenvalue. The generic condition will ensure all the dispersive estimates,
pointwise decay, smoothing estimates, and the boundedness of wave operators. 
\end{rem}

Throughout this paper, we denote $P_{c}$ and $P_d$ as the projections on the continuous
spectrum  and the discrete spectrum of $H$ respectively.


\subsubsection{Solitary waves}
Under the assumptions on the potential, there exists a family of
small nonlinear bound states $Q=Q\left[z\right]$  parameterized   by the small parameter 
\[
z=\left(\phi,Q\right)\in\mathbb{C}
\]
where we denote the inner product in $L^{2}$ by
\[
\left(a,b\right):=\int\overline{a}b\,dx.
\]
The small soliton $Q[z]$ satisfies
\[
Q\left[z\right]=z\phi+h(z)
\]
such that $\left\Vert h(z) \right\Vert _{H^{2}\bigcap W^{1,1}}\lesssim|z|^2,\,h\left(z\right)\bot\phi$
and solves
\begin{equation}
\left(-\partial_{xx}+V\right)Q\left[z\right]-\left|Q\left[z\right]\right|^{2}Q\left[z\right]=E\left[z\right]Q\left[z\right]\label{eq:NLEj}
\end{equation}
with
\[
E\left[z\right]=-\rho^{2}+o\left(z\right)\in\mathbb{R}.
\]
See Lemma \ref{lem:NLB} for more details. Notice that the gauge co-variance
is inherited by $Q$. More precisely, we have
\begin{equation}\label{eq:covariance}
Q\left[ze^{i\alpha}\right]=Q\left[z\right]e^{i\alpha},\,E\left[z\right]=E\left[\left|z\right|\right].
\end{equation}
Considering the NLS \eqref{eq:NLS}, the nonlinear bound state $Q$ gives raise to a exact solitary wave solution,
\[
u=e^{iE[z]t}Q[z].
\]
Note that $Q\left[z\right]$ is differentiable with respect to $z$
if we regard $z$ as a real vector
\[
z=a+ib\longleftrightarrow\left(a,b\right)\in\mathbb{R}^{2}.
\]
We will denote the $z-$derivatives by
\begin{equation}
\text{D}_{1}Q\left[z\right]:=\frac{\partial}{\partial a}Q\left[z\right],\ \\ \ \text{D}_{2}Q\left[z\right]:=\frac{\partial}{\partial b}Q\left[z\right].\label{eq:Dj}
\end{equation}
Then we use the notation
\[
\text{D}Q\left[z\right]:\,\mathbb{C\rightarrow\mathbb{C}}
\]
to denote the Jacobian matrix regarded as an $\mathbb{R}-$linear
map on $\mathbb{C}$
\begin{equation}
\text{D}Q\left[z\right]w:=\text{D}_{1}Q\left[z\right]\Re w+i\text{D}_{2}Q\left[z\right]\Im w.\label{eq:Dj2}
\end{equation}
The gauge co-variance of $Q\left[z\right]$ implies
\begin{equation}
\text{D}Q\left[z\right]iz=iQ\left[z\right].\label{eq:Dj2I}
\end{equation}
Indeed, one can differentiate the first relation in \eqref{eq:covariance} with respect to $\alpha$ and evaluate at $\alpha=0$ then the desired relation follows.
The relations \eqref{eq:Dj2} and \eqref{eq:Dj2I} will be used when we
analyze the modulation equations.

\subsubsection{Linearization and spectrum}

Consider a small-norm solution to the NLS
\[
iu_{t}-\partial_{xx}u+Vu-\left|u\right|^{2}u=0.
\]
With some time-dependent soliton parameter $z(t)$, it is natural to decompose the solution as
\begin{equation}
u=Q\left[z\left(t\right)\right]+\eta\left(t\right)\label{eq:1decomp}
\end{equation}
and obtain the equation for $\eta$:
\[
i\eta_{t}+\mathrm{H}\left[z\right]\eta-E\left[z\right]Q\left[z\right]-i\text{D}Q\left[z\right]\dot{z}+F_{2}\left(z,\eta\right)=0
\]
where
\begin{equation}\label{eq:mathrmH}
\mathrm{H}\left[z\right]\eta=\left(-\partial_{xx}+V\right)\eta-2\left|Q\left[z\right]\right|^{2}\eta-\left(Q\left[z\right]\right)^{2}\bar{\eta}
\end{equation}
and $F_{2}\left(z,\eta\right)$ is the collection of terms of quadratic
or higher in $\eta$. 
\begin{rem}
In this paper, we choose the scalar formulation of our equation since
we are working on small solitons and all the spectral analysis are
closely tied to the scalar Schr\"odinger operator $-\partial_{xx}+V$.
\end{rem}

Now we can define the continuous spectral subspace with respect to
$\mathrm{H}\left[z\right]$.
\begin{defn}[Continuous spectral subspace]
\label{def:Conti}The continuous spectral
subspace $\mathcal{H}_{c}\left[z\right]$ with respect to $\mathrm{H}\left[z\right]$
is defined as
\[
\mathcal{H}_{c}\left[z\right]:=\left\{ \eta\in L^{2}:\,\left\langle i\eta,\text{D}_{1}Q\left[z\right]\right\rangle =\left\langle i\eta,\text{D}_{2}Q\left[z\right]\right\rangle =0\right\}
\]
where the inner product is given by
\begin{equation}
\left\langle a,b\right\rangle :=\Re\left\langle a,b\right\rangle =\int\Re a\Re b+\Im a\Im b\,dx.\label{eq:inner}
\end{equation}
\end{defn}

Notice that $\mathcal{H}_{c}\left[z\right]$ in invariant with respect
to $i\left(\mathrm{H}\left[z\right]-E\left[z\right]\right)$ by the
relations $\text{D}Q\left[z\right]iz=iQ\left[z\right]$ and 
\[
\left(\mathrm{H}\left[z\right]-E\left[z\right]\right)\text{D}Q\left[z\right]=\left(\text{D}E\left[z\right]\right)Q\left[z\right]
\]which is obtained by differentiating the nonlinear eigenvalue problem
\eqref{eq:NLEj}. 
Therefore, restricting onto $\mathcal{H}_{c}\left[z\right]$ removes
the non-decaying  solution to the linear equation
\[
i\eta_{t}+\left(\mathrm{H}\left[z\right]\eta-E\left[z\right]\right)\eta=0
\]
for fixed $z$.

At the first glance, the decomposition \eqref{eq:1decomp} is not unique.
We will choose the unique decomposition to ensure the following orthogonal
conditions
\[
\left\langle i\eta,\text{D}_{1}Q\left[z\right]\right\rangle =\left\langle i\eta,\text{D}_{2}Q\left[z\right]\right\rangle =0
\]
such that $\eta\left(t\right)$ indeed is dispersive. Note that the
linearization destroys gauge invariance, the linearized operator $\mathrm{H}\left[z\right]$
is not complex-linear. It is, however, symmetric if we regard $\mathbb{C}$
as $\mathbb{R}^{2}$ and use the reduced inner product \eqref{eq:inner}.  We remark that in many literature, matrix forms of the linearized
operators are commonly used, like for example Beceanu \cite{Bec1}
and Schlag \cite{Sch1}.

\subsubsection{Main theorem}

With preparations above, we can state the main theorem on the long-time
behavior the small-norm solutions to \eqref{eq:NLS}. 
\begin{thm}
\label{thm:AC}
There exists $0<\epsilon_{0}\ll 1$ such that  every solution $u$ to the equation
\begin{equation}
i\partial_{t}u-\partial_{xx}u+Vu=\lambda\left|u\right|^{2}u\label{eq:eqthm1}
\end{equation}
 with sufficiently small initial data
\begin{equation}\label{eq:smalldata}
    \left\Vert u_{0}\right\Vert _{H^{1}}+\left\Vert xu_{0}\right\Vert _{L^{2}}\lesssim\epsilon\ll1,\,\epsilon\leq \epsilon_{0}
\end{equation}
can be uniquely decomposed as
\begin{equation}
u=Q\left(x,z\left(t\right)\right)+\eta\left(t\right)\label{eq:decomp}
\end{equation}
with differentiable $z\left(t\right)\in\mathbb{C}$ such that
\[
\left\langle i\eta(t),\mathrm{D}_{1}Q\left[z(t)\right]\right\rangle =\left\langle i\eta(t),\mathrm{D}_{2}Q\left[z\right(t)]\right\rangle =0.
\]
The radiation term $\eta$ satisfies the sharp decay rate
\begin{align}\label{main1fdecay}
\left\Vert \eta(t)\right\Vert_{L^\infty_x}
  \lesssim \frac{\epsilon}{\left(1+\left|t\right|\right)^{\frac{1}{2}}}. 
\end{align}
Moreover, if we define the profile of  $\eta$ as
\begin{align}
\label{main1prof}
f\left(t,x\right):=e^{-it\left(-\partial_{xx}+V\right)}\eta\left(t,x\right),
  \qquad \tilde{f}\left(t,k\right):=e^{-itk^{2}}\tilde{\eta}\left(t,k\right),
\end{align}
where $\tilde{g}=\widetilde{\mathcal{F}}g$ denotes the distorted Fourier transform, 
then 
\begin{align}\label{main1fbounds}
{\big\| \tilde{f}(t) \big\|}_{L_k^\infty}
  + (1+|t|)^{-\alpha} {\| \partial_{k}\tilde{f}(t) \|}_{L_k^2} \lesssim\epsilon
\end{align}
for some $\alpha=\alpha(V)>0$ small enough.

We  also have the following asymptotics:
there exists $W_{+\infty}\in L^{\infty}$ such that
\begin{align}\label{mainasy}
\left|\tilde{f}\left(t,k\right)\exp\left(\frac{i}{2}\int_{0}^{t}\left|\tilde{f}\left(s,k\right)\right|^{2}
  \frac{ds}{s+1}\right)-W_{+\infty}(k)\right| \lesssim \epsilon \, t^{-\beta}
\end{align}
for some $\beta\in(0,\alpha)$ as $t\rightarrow\infty$.
Combining \eqref{mainasy} above, and the linear asymptotic formula
\begin{align}\label{linearasy}
\eta\left(t,x\right)=\frac{e^{i\frac{x^{2}}{4t}}}{\sqrt{-2it}}
  \tilde{f}\left(t,-\frac{x}{2t}\right)+\mathcal{O}\left(t^{-\frac{1}{2}-\alpha}\right),\ \ t\gg1,
\end{align}
one can also derive the following asymptotic formula for $\eta$ in the physical space:
\begin{align}\label{nonlinearasy}
\eta\left(t,x\right)=\frac{e^{i\frac{x^{2}}{4t}}}{\sqrt{-2it}}
  \exp\left(-\frac{i}{2}\left|W_{+\infty}\left(-\frac{x}{2t}\right)\right|^{2}\log t\right)
  W_{+\infty}\left(-\frac{x}{2t}\right)+\mathcal{O}\left(t^{-\frac{1}{2}-\alpha}\right),\ \ t\gg1.
\end{align}
The soliton parameter $z(t)$ satisfies
\begin{equation}
\left\Vert z\right\Vert _{L_{t}^{\infty}}\lesssim\left\Vert u_{0}\right\Vert _{H^{1}},\qquad
\left|\dot{z}-iE\left[z\right]z\right|\left(t\right)\lesssim\epsilon^{2}t^{-2+2\alpha}.\label{eq:upper2}
\end{equation}
Furthermore, one can also obtain that there exists $z\left(\infty\right)\in\mathbb{C}$
such that 
\begin{equation}
z\left(t\right)\exp\left\{ i\int_{0}^{t}E\left[z\left(s\right)\right]\,ds\right\} \rightarrow z\left(\infty\right)\label{eq:paralim}.
\end{equation}
In particular, we also have
\begin{equation}
\left|\left(\left|z\left(\infty\right)\right|-\left|z\left(t\right)\right|\right)\right|\lesssim\epsilon^{2}t^{-1+2\alpha}.\label{eq:param}
\end{equation}
and if $z\left(\infty\right)\neq0$, then
\begin{equation}
\arg z\left(t\right)+\int_{0}^{t}E\left[z\left(s\right)\right]\,ds-\arg z\left(\infty\right)\rightarrow0\mod2\pi.\label{eq:paralim-2}
\end{equation}
\end{thm}

Before discussing the main ideas of this paper, let us give a few
comments on this result
\begin{rem}
We choose the complex variable $z\left(t\right)$ to trace the dynamics
of solitons since for the small solution, potentially, the solitary
part will go to $0$. In this situation,  if we try to use the real variables as modulation parameters
to describe the dynamics of the solitary wave, the argument will be degenerate since  the phase of a zero-size soliton is not well-defined.
\end{rem}
\begin{rem}
As mentioned in the introductory part, to capture the modified scattering phenomenon of the radiation term, \eqref{nonlinearasy}, the requirement of the weights here is natural since the $W_{+\infty}$ function is obtained via the pointwise behavior of the (distorted) Fourier transform of the solution. From the view of the Fourier duality, the weights in the physical space imply the regularity of the Fourier transform which gives the pointwise meaning of the Fourier transform via the Sobolev embedding. 
\end{rem}
\begin{rem}
The theorem above in particular implies the asymptotic completeness
of the nonlinear equation \eqref{eq:eqthm1} restricted onto the small-norm
solutions. It also implies the asymptotic stability of small solitons.
Note that in the asymptotic decomposition \eqref{eq:decomp}, both the
pieces are nonlinear.
\end{rem}

\begin{rem}
The model \eqref{eq:eqthm1} considered in the theorem above can be
regarded as a perturbation of the integrable cubic NLS in one dimension.
The method used in this paper can also be applied to other settings. For example, more
generally, we can treat
\[
i\partial_{t}u-\partial_{xx}u+Vu=\lambda a\left(x\right)\left|u\right|^{2}u+\left|u\right|^{p}u
\]
with $a\left(x\right)\rightarrow1$ sufficiently fast as $\left|x\right|\rightarrow\infty$ and $p>2$. 
\end{rem}

\begin{rem}
In this paper, in order to achieve the desired result, one of the key points is to study Jost functions in detail. To study the stability of solitary waves in higher dimensions for critical nonlinear models, one has to understand and study the refined structures the distorted Fourier transforms. For the asymptotics of small solutions of quadratic NLS in $3$d, we refer to Germain-Hani-Walsh \cite{GHS} and Pusateri-Soffer \cite{PS}. Smoothing estimates are also expected to be useful. For example, we refer to Mizumachi \cite{Miz2} for applications of smoothing estimates in $2$d supercritical problems.
\end{rem}

\subsection{Sketch of the ideas}

To finish the introduction, we sketch some difficulties, key steps
and main ideas in this paper.

\subsubsection{Modulation}

Consider a small-norm solution $u$ to the equation \eqref{eq:NLS}.
With the appearance of solitary waves, it is natural to decompose
the solution as
\begin{equation}
u=Q\left(x,z\left(t\right)\right)+\eta\left(t\right).\label{eq:decomp1}
\end{equation}
Here we make the parameter of the soliton time-dependent so that the
radiation term $\eta$ will be orthogonal to the generalized kernel
of the linearized operator around $Q\left(x,z(t)\right)$. 

For the modulation analysis, the implicit function theorem will give
us the initial decomposition. Then we differentiate the orthogonality
conditions to find the equations for $z\left(t\right)$. From the
modulation equations, one has
\begin{equation}
|\dot{z}(t)-iz(t)E\left[z(t)\right]|\lesssim\left\langle \phi,\eta^{2}\right\rangle  \lesssim t^{-2+2\alpha}\label{eq:mod1}
\end{equation}
with the bootstrap assumption on $\eta$ using the local improved decay, see \eqref{eq:introimpro}.

Typically, in the analysis of the asymptotic stability with modulation
parameters, the linearized operator is time-dependent. To invoke dispersive estimates, one has to find a reference operator which typically
is given by the linearized operator with parameters evaluated
at $t=\infty$, see for example Buslaev-Perelman \cite{BP,BP2}, Buslaev-Sulem
\cite{BS} and Krieger-Schlag \cite{KS}. To find this reference operator,
one always needs to integrate the modulation equation \eqref{eq:mod1}
twice. Unlike the problem with higher power nonlinearities in the
references mentioned above, in the cubic problem the best decay rate
one can get is $t^{-2+2\alpha}$ which is not sufficient to be integrated
twice. In our current setting, we can use the smallness of the solitons
to employ the scalar Schr\"odinger operator as the reference operator
to get rid of this difficulty.

\subsubsection{Spectrum and projections}

Plugging  the decomposition \eqref{eq:decomp1} into the equation \eqref{eq:NLS}
and using the equation for $Q$, \eqref{eq:NLEj}, we get the equation
for $\eta$:
\begin{align}
i\partial_{t}\eta-\partial_{xx}\eta+V\left(x\right)\eta & =2\left|Q\left[z\right]\right|^{2}\eta+\left(Q\left[z\right]\right)^{2}\bar{\eta}\label{eq:etaeqintro}\\
 & +E\left[z\right]Q\left[z\right]-i\text{D}Q\left[z\right]\dot{z}\nonumber \\
 & +\overline{Q\left[z\right]}\eta^{2}+2Q\left[z\right]\left|\eta\right|^{2}+\left|\eta\right|^{2}\eta.\nonumber 
\end{align}
From the modulation equation, $\eta$ is orthogonal to the time-dependent
linear operator
\[
\mathrm{H}\left[z\right]\eta=\left(-\partial_{xx}+V\right)\eta-2\left|Q\left[z\right]\right|^{2}\eta-\left(Q\left[z\right]\right)^{2}\bar{\eta}.
\]
We will use the scalar operator $H$, \eqref{eq:schOp}, as the reference
operator to perform distorted Fourier transforms and dispersive analysis.
Since $z\left(t\right)$ are small, the orthogonality conditions ensured
by the modulation equation will give $P_{c}\eta\sim\eta$ where $P_{c}$
is the projection onto the continuous spectrum of $H$, in the spaces
we are interested in. Note that in general for large solitons, we will
need dispersive estimates for time-dependent potentials and the matrix
formalism, see Beceanu \cite{Bec1}, Krieger-Schlag \cite{KS} and Schlag
\cite{Sch}.

More precisely, let $\phi$ be the eigenfunction of $H$ associated
with the negative eigenvalue $-\rho^{2}$. In the complex setting,
$\phi$ spans $1-D$ subspace $\left\{ \beta\phi,\,\beta\in\mathbb{C}\right\} $
which is a $2-D$ subspace if we use the real variables and the induced
inner product \eqref{eq:inner}. 
Since $\left|z\right|$ is small, then $\left\{ \beta\phi,\,\beta\in\mathbb{C}\right\} $
is a good approximation to the complement of  $\mathcal{H}_{c}\left[z\right]$.
Because intuitively, by Lemma \ref{lem:NLB}, we know
\[
\left|\text{D}_{1}Q\left[z\right]-\phi\right|+\left|\text{D}_{2}Q\left[z\right]-i\phi\right|
\lesssim 
\left|E[z]+\rho^2\right|\lesssim\left|z\right|.
\]
So roughly if we project away from the complex span of $\phi$, we
will roughly stay away from the generalized kernel of the linearized
operator $H[z]$ and vice versa. For more precise analysis see Lemma \ref{lem:Diff}.  But we point out that this comparison is time-dependent and there is no appropriate estimate for the time derivative of comparison. To understand the refined structure of this comparison for $\eta$, we find  some refined decomposition for $P_c \eta$.

\subsubsection{Analysis of the radiation term}
From the equation \eqref{eq:etaeqintro}, by the fact $Q\left[z\right]=iz\text{D}Q$,
after setting
\begin{equation}
\mathcal{Q}\left[z\right]=e^{-i\int_{0}^{t}E\left[z(s)\right]\,ds}Q\left[z\right]\label{eq:modQ-1}
\end{equation}
one has the following equation
\begin{align}
i\partial_{t}\eta-\partial_{xx}\eta+V\eta & =2\left|\mathcal{Q}\left[z\right]\right|^{2}\eta+\left(\mathcal{Q}\left[z\right]\right)^{2}e^{2i\int_{0}^{t}E\left[z(\sigma)\right]\,d\sigma}\bar{\eta}\label{eq:eta-1-1}\\
 & +\overline{Q\left[z\right]}\eta^{2}+2Q\left[z\right]\left|\eta\right|^{2}\nonumber \\
 & +\left|\eta\right|^{2}\eta\nonumber \\
 & +i\text{D}Q\left(\dot{z}-izE\left[z\right]\right)\nonumber \\
 & =:N_{1,1}+N_{1,2}+N_{2}+N_{3}+M=:F.\label{eq:inhomo-1-1}
\end{align}
The key point for the change \eqref{eq:modQ-1} is that
\begin{equation}
\partial_{t}\mathcal{Q}\left[z\left(t\right)\right]=e^{-i\int_{0}^{t}E\left[z(s)\right]\,ds}\text{D}Q\left[z\right]\left(\dot{z}\left(t\right)-iz\left(t\right)E\left[z\left(t\right)\right]\right)\label{eq:mathcalQ-1}
\end{equation}
for which we can apply the modulation equation.

By Duhamel formula, solving the equation for $\eta$ for $t=1$, one
has
\begin{equation}
\eta\left(t,x\right)=e^{iH\left(t-1\right)}\eta_{1}+\int_{1}^{t}e^{iH\left(t-s\right)}\left(F(s)\right)\,ds.\label{eq:duhameta-1}
\end{equation}
Define the profile for $\eta$ as
\begin{equation}
f\left(t\right):=e^{-iHt}P_c \eta\left(t\right).\label{eq:profief-1}
\end{equation}
Our first step  is to use the Fourier Transform adapted to $H$ -
the so-called ``Distorted Fourier Transform'' - to rewrite \eqref{eq:duhameta-1} in (distorted) Fourier space.
For the sake of this brief introduction, it suffices
to admit for the moment the existence of ``generalized plane waves'' $\mathcal{K}(x,\lambda)$ such that
one can define an $L^2$ unitary transformation $\wtF$ by
\begin{align}\label{tildeFint}
\wtF[f](\lambda) := \widetilde{f}(\lambda) := \int \overline{\mathcal{K}(x,\lambda)}f(x)\,dx,
  \quad \mbox{with} \quad \wtF^{-1}\left[\phi\right](x)=\int\mathcal{K}(x,\lambda)\phi(\lambda)\,d\lambda.
\end{align}
See \eqref{matK}, \eqref{psipm} and \eqref{defT},
for the precise definition of $\mathcal{K}(x,\lambda)$ and its relation
with the generalized eigenfunctions of $H$.
The distorted transform $\wtF$ diagonalizes the Schr\"odinger operator restricted on the continuous spectrum:
$(-\partial_{xx}+V)P_c=\wtF^{-1}\lambda^{2}\wtF$.

Given a solution $\eta$ of \eqref{eq:eta-1-1}, one can prove the basic linear estimate\footnote{Technically, this estimate holds after projecting onto the continuous spectrum $P_c \eta$. But as we discussed above, in our setting, one has $\eta\sim P_c \eta$.} 
\begin{align}\label{intlinest}
{\| \eta(t,\cdot) \|}_{L^\infty_x} \lesssim \frac{1}{|t|^{1/2}} {\| \wt{f}(t) \|}_{L^\infty_k}
  + \frac{1}{|t|^{3/4}} {\| \partial_k \wt{f}(t) \|}_{L^2_k}
\end{align}
which is the analogue of the standard linear estimate for the case $V=0$
(where one can replace ${\| \partial_k \wt{f}(t) \|}_{L^2}$ by 
a standard weighted norm ${\| xf(t) \|}_{L^2} = {\| Ju(t) \|}_{L^2}$, with $J = x+2it\partial_x$).
To obtain the sharp pointwise decay of $|t|^{-1/2}$ it then suffices to control 
$\wt{f}$ uniformly in $k$ and $t$ and the $L^2$-norm of $\partial_k \wt{f}(t)$
with a small growth in $t$. 
Both of these bounds are achieved by studying the equation in the distorted Fourier space.

Taking the distorted transform and the in terms of profile, we have
\begin{equation}\label{eq:introprof}
\tilde{f}\left(t,k\right)=\tilde{f}\left(1,k\right)+\int_{1}^{t}e^{-ik^{2}s}\tilde{F}\left(s,k\right)\,ds.
\end{equation}
To deal the cubic feature of the problem, we perform the space-time resonance
analysis as in Kato-Pusateri \cite{KP} in the free case, and Germain-Pusateri-Rousset
\cite{GPR} and Chen-Pusateri \cite{CP} in the perturbed settings.
The key point is to establish the following two estimates
\begin{equation}
\left\Vert \partial_{k}\tilde{f}\left(t,\cdot\right)\right\Vert _{L_{k}^{2}}\lesssim \epsilon t^{\alpha},\,\,\,\left\Vert \tilde{f}\left(t,\cdot\right)\right\Vert _{L_{k}^{\infty}}\lesssim\epsilon\label{eq:keyest}
\end{equation}for some small $\alpha>0$. 
We will run a bootstrap argument to prove the estimates above. Given the bootstrap assumption above, immediately, from the localized pointwise decay one has
\begin{align}\label{eq:introimpro}
\left\Vert \jx^{-2} \eta \right\Vert_{L^{\infty}_x}
  \lesssim |t|^{-1}{\big\| \tilde{f} \big\|}_{H^1_k}\lesssim \epsilon t^{-1+\alpha}.
\end{align}
To analyze the weighted estimate, i.e., the first estimate of \eqref{eq:keyest},
taking $\partial_{k}$ of the inhomogeneous term, we get two pieces
\begin{align}
\partial_{k}\left(\int_{1}^{t}e^{-ik^{2}s}\tilde{F}\left(s,k\right)\,ds\right) & =\int_{1}^{t}ikse^{-ik^{2}s}\tilde{F}\left(s,k\right)\,ds+\int_{1}^{t}e^{-ik^{2}s}\partial_{k}\tilde{F}\left(s,k\right)\,ds.\label{eq:partialkF}
\end{align}
The analysis for the cubic term $N_3$ here will be the same as Chen-Pusateri \cite{CP}.
The key point is to explore the structure of the nonlinear spectral
measure, the decomposition of the Jost functions and various dispersive
decay.

The crucial parts in this paper are to estimate the corresponding inhomogeneous
terms associated with the quadratic term $N_{2}$ and first order
perturbation $N_{1,1}+N_{1,2}$ from equation \eqref{eq:inhomo-1-1}.
With the localized coefficients, the second term on the RHS of \eqref{eq:partialkF}
is relatively easier to handle compared with the first term which
has an additional growth in $s$. So in this introduction, we focus on the first term.

\noindent\textbf{Analysis of quadratic terms.} Using Plancherel's theorem to switch back to the physical space, one
of key estimates to bound the first term on the RHS of \eqref{eq:partialkF}
is the inhomogeneous local smoothing estimate
\[
\left\Vert \int e^{-isH}P_{c}F\left(s\right)\,ds\right\Vert _{L_{x}^{2}}\lesssim\left\Vert \left\langle x\right\rangle F\right\Vert _{L_{x}^{1}L_{t}^{2}}
\]
which can be regarded as the dual version of the Kato smoothing estimate
or certain resolvent estimates. One also needs to use some pseudo-differential
operator bounds here. For the quadratic terms, the inhomogenous smoothing
estimate closes the bootstrapping for  the weighted estimate since the decay rate from the
decay of $\eta$, \eqref{eq:introimpro}, and the growth $s$ together give $s^{-2+2\alpha}s$
which is $L_{s}^{2}$ integrable via the smoothing estimate. Here
to deal with the weights, we use the boundedness of wave operators.

\smallskip

\noindent\textbf{Analysis of first order perturbations.} 
The most difficult part is the analysis of the first order perturbations.
Superficially, taking the additional growth $s$ into account, with
the improved local decay rate given by $\eta$, the first term from
\eqref{eq:partialkF} associated with first order perturbations will
result in some mild growth $s^{\alpha}$ even without the time integration.
To overcome these difficulties, we perform integration by parts in time.
Unlike more standard situations, for example, Shatah \cite{Sh} and
Germain-Pusateri \cite{GP},  integration
by parts in time and the equation will not give extra decay in time
since first order perturbations will only bring in extra smallness
of the coefficient instead of the extra time decay in other settings. Moreover, due to the influence of the additional phase shift in $N_{1,2}$ from \eqref{eq:eta-1-1},  in our problem, we need to perform integration by parts in time \emph{twice}, see \S \ref{subsubsec:firstorderweight} for full details.  To use the time derivatives from integration by parts above, we explore more on
smoothing estimates and the Fourier transform in time.  
This part is much more involved. The basic reason is that the comparison $\eta\sim P_c \eta$ is time-dependent which could not commute with the Fourier transform in time. And the orthogonality conditions for $\eta$ could not imply any orthogonality conditions for $\partial_t \eta$ and $\partial_{tt}\eta$. 

Here we sketch some ideas of the analysis of first order perturbations. From the modulation equation and its estimates \eqref{eq:mod1}, we know that $\mathcal{Q}[z(t)]$ converges to some $\mathcal{Q}[z(\infty)]$. We can rewrite the linear part of the equation \eqref{eq:eta-1-1} as
\begin{align*}
    i\partial_{t}\eta-\partial_{xx}\eta+V\eta & =2\left|\mathcal{Q}\left[z(\infty)\right]\right|^{2}\eta+\left(\mathcal{Q}\left[z(\infty)\right]\right)^{2}e^{2i\int_{0}^{t}E\left[z(\sigma)\right]\,d\sigma}\bar{\eta}\\
 & +2\left(\left|\mathcal{Q}\left[z\right]\right|^{2}-\left|\mathcal{Q}\left[z(\infty)\right]\right|\right)\eta
 +\left(\left(\mathcal{Q}\left[z\right]\right)^{2}-\left(\mathcal{Q}\left[z(\infty)\right]\right)^{2}\right)e^{2i\int_{0}^{t}E\left[z(\sigma)\right]\,d\sigma}\bar{\eta}.
\end{align*}
From the modulation equation and the decay of $\eta$, the last  line of the RHS above can be treated as quadratic terms.

Therefore  at the linear level, the equation of $\eta$ is given
by
\begin{equation}\label{eq:introlineareta}
    i\eta_{t}=-H\eta+A\eta+Be^{i2\int_{0}^{t}E\left[z(\sigma)\right]\,d\sigma}\overline{\eta}
\end{equation}
where we denoted $A=2\left|\mathcal{Q}\left[z(\infty)\right]\right|^{2}$
and $B=\mathcal{Q}^{2}\left[z(\infty)\right]$. 
Here
we remark that we used $t=\infty$ for the sake of convenience. One can also
use $t=T$ to define $A$ and $B$. 

We decompose $\eta$ as
\begin{equation}\label{eq:introdecomg}
\eta=g+\mathsf{a}\left(t\right)\phi
\end{equation}
where $\mathsf{a}\left(t\right)\phi=\left(\eta,\phi\right)\phi=P_{d}\eta$
and $g=P_{c}\eta$. 

\smallskip

\noindent\textbf{Refined decomposition.}
Projecting \eqref{eq:introlineareta} onto the continuous spectrum with respect to $H$, one
has
\begin{equation}
ig_{t}=-Hg+P_{c}\left(A\left(g+\mathsf{a}(t)\phi\right)\right)+e^{i2\int_{0}^{t}E\left[z(\sigma)\right]\,d\sigma}P_{c}\left(B\left(\overline{g}+\overline{\mathsf{a}}(t)\phi\right)\right)\label{eq:linearg1}
\end{equation}
The key point  here is that  on the RHS of the equation above, there are terms
with $\mathsf{a}(t)$ and $e^{i2\int_{0}^{t}E\left[z(\sigma)\right]\,d\sigma}{\overline{\mathsf{a}}(t)}$
involved.  
It is also related to the fact that the projection $P_{c}$
does not commute with the linear operator $\mathrm{H}\left[z(\infty)\right]$
defined by \eqref{eq:mathrmH}.  These terms, although enjoy the pointwise decay rates due to the comparison of the continuous spaces do not satisfy refined smoothing estimates.  So we could not work on smoothing estimates for $g$ directly.




To get a better understanding of $g$, we introduce a refined decomposition:
\begin{equation}\label{eq:refdecomp1}
g=r+\mathsf{a}(t)\mathfrak{A}(x)+e^{i2\int_{0}^{t}E\left[z(\sigma)\right]\,d\sigma}\overline{\mathsf{a}(t)}\mathfrak{B}(x)
\end{equation}
for some $\mathfrak{A}(x)\in P_{c}L^{2}$ and $\mathfrak{B}(x)\in P_{c}L^{2}$ such that plugging the decomposition \eqref{eq:refdecomp1} above into
the linear equation \eqref{eq:linearg1}, it results in an equation for
$r$ with the property that,  on the RHS, approximately only $r$ is involved. At the linear level this equation is given by
\begin{align}
ir_{t} & =-Hr+P_{c}\left(Ar\right)+e^{i2\int_{0}^{t}E\left[z(\sigma)\right]\,d\sigma}P_{c}\left(B\overline{r}\right)\nonumber \\
 & -\left(Ar+e^{i2\int_{0}^{t}E\left[z(\sigma)\right]\,d\sigma}B\overline{r},\phi\right)\mathfrak{A}-\left(Br+Ae^{i2\int_{0}^{t}E\left[z(\sigma)\right]\,d\sigma}\overline{r},\phi\right)\mathfrak{B}. 
\end{align}
Unlike the RHS of the earlier equation \eqref{eq:linearg1}, there are no terms on  the RHS of the 
equation above without smoothing estimates. Therefore, when we apply the smoothing estimates to $r$ using the corresponding Duhamel formula, one can absorb the corresponding first order terms on the RHS to the LHS by the fact that the coefficients $A$ and $B$ are small. Moreover, it has a clean structure to perform the Fourier transform in $t$ on both sides to establish smoothing estimates.
Here to find $\mathfrak{A}$ and $\mathfrak{B}$, see \eqref{eq:ell1} and \eqref{eq:ell2},  is reminiscent the Poincar\'e normal form in Buslaev-Sulem \cite{BS}, Komech-Kopylova \cite{KK1} and Soffer-Weinstein \cite{SW} at the linear level.  

After obtaining the refined decomposition, we impose the auxiliary bootstrap estimates
{\small\begin{equation}
\left\Vert \left\langle x\right\rangle ^{-2}t\left(-2iE\left[z(t)\right]\partial_{t}+\partial_{t}^{2}\right)\left(r\right)_{L}\right\Vert _{L_{x}^{\infty}L_{t}^{2}[0,T]}+\left\Vert \left\langle x\right\rangle ^{-2}t\partial_{x}^{j}\left(r\right)_{H}\right\Vert _{L_{x}^{\infty}L_{t}^{2}[0,T]}\lesssim T^\alpha \left\Vert \tilde{\eta}(1,\cdot)\right\Vert _{H^{1}},\ j=0,1\label{eq:impsmL-1}
\end{equation}}where the subscripts $L$ and $H$  denote the low and high frequency
parts with respect to the Fourier transform in $t$ respectively.\footnote{Strictly speaking, when we perform the Fourier transform with respect to $t$ for the solution, we need to a global solution.
One can treat our estimates here as a priori estimates and then use the Picard iteration to obtain the final result, see Remark \ref{rem:FTt} for details.} With these refined
smoothing estimates as auxiliary bootstrap conditions, we are able
to absorb the extra growth in $s$ form the first term on the RHS
of \eqref{eq:partialkF} after performing integration by parts in time.

\smallskip
\noindent\textbf{Integration by parts in time.} When performing the integration by parts in time, we need to check several resonance
conditions caused by the potential, the phase shift and the bound states.  To illustrate the idea, consider the inhomogeneous term in the Duhamel expansion \eqref{eq:introprof}
for the profile $f$ corresponding to $N_{1,2}$. Using the decomposition \eqref{eq:introdecomg} and refined one \eqref{eq:refdecomp1}, we write$$N_{1,2}=N_{1,2,r}+N_{1,2,\mathsf{a}}+N_{1,2,\overline{\mathsf{a}}}$$ where the three terms on the RHS collect terms involving $r$, $\mathsf{a}(t)$ and $\overline{\mathsf{a}}(t)$ respectively. 

For the part with $r$, taking $\partial_k$, see \eqref{eq:partialkF}, 
the most difficult term is
\begin{equation}\label{eq:example}
\int_{1}^{t}2ike^{-ik^{2}s}\left(\widetilde{{N}_{1,2,r}}\left(s,k\right)\right)\,ds=
\int_{1}^{t}2ike^{-ik^{2}s+2i\int_{0}^{s}E\left[z(\sigma)\right]\,d\sigma}\left(\widetilde{\mathcal{N}_{1,2,r}}\left(s,k\right)\right)\,ds
\end{equation}
where $\mathcal{N}_{1,2,r}=\left(\mathcal{Q}\left[z(s)\right]\right)^{2}\bar{r}\left(s\right)$. Here we have two ways which are both necessary for our goal, to perform integration by parts in time. On one hand, on the LHS of \eqref{eq:example}, we can use $ke^{-isk^2}=\frac{1}{-ik}\partial_se^{-isk^2}$. The singularity at $k=0$ can be compensated by the generic conditions
of the potential and the localized coefficients which implies that $\widetilde{N_{1,2,r}}(s,k)\sim k,\,k\sim 0$.  On the other hand, from the RHS of \eqref{eq:example}, we notice that the oscillatory phase can be written as
\begin{equation}\label{eq:phaseintro}
    e^{-ik^{2}s}e^{2i\int_{0}^{s}E\left[z(\sigma)\right]\,d\sigma}=\frac{1}{-ik^{2}+2iE\left[z(s)\right]}\partial_{s}\left(e^{-ik^{2}s}e^{2i\int_{0}^{s}E\left[z(\sigma)\right]\,d\sigma}\right).
\end{equation}
Note that $E\left[z\right]$ is strictly negative and $k^{2}-2E\left[z(s)\right]>\rho^{2}$, 
whence we can freely perform integration by parts in $s$.  After performing these two ways of integration by parts in $s$, the most difficult bulk terms will resemble the first estimate in \eqref{eq:impsmL-1}. 

To handle weighted estimates for inhomogeneous terms with $N_{1,2,\mathsf{a}}$ and $N_{1,2,\overline{\mathsf{a}}}$, we define
\[
\mathsf{b}\left(t\right):=e^{i\rho^{2}t}\mathsf{a}\left(t\right).
\]
Then the most difficult parts are given in the following forms
\begin{equation*}
    \int_{1}^{t}e^{-isk^{2}}2isk\int e^{-i\rho^{2}s}\mathsf{b}\left(s\right)\tilde{\mathcal{U}}\left(\ell\right)\nu\left(k,\ell\right)\,d\ell ds.
\end{equation*}
and
\begin{equation*}
\int_{1}^{t}e^{-isk^{2}}2isk\int e^{2i\int_{0}^{s}E\left[z(\sigma)\right]\,d\sigma}e^{i\rho^{2}s}\mathrm{\overline{\mathsf{b}}}\left(s\right)\tilde{\mathcal{U}}\left(\ell\right)\nu\left(k,\ell\right)\,d\ell ds
\end{equation*}with some smooth localized function $\mathcal{U}$. 
For the first integral, we note that $k^{2}+\rho^{2}\ge\rho^{2}>0$. Using the identity
\begin{equation}\label{eq:introb1}
    e^{-isk^{2}}e^{-i\rho^{2}s}=-\frac{1}{i\left(k^{2}+\rho^{2}\right)}\frac{d}{ds}\left(e^{-isk^{2}}e^{-i\rho^{2}s}\right)
\end{equation}
we can perform integration by parts in $s$ and use the equation for $\mathsf{b}$.  For the second integral,  we notice that $k^{2}-2E\left[z(t)\right]-\rho^{2}\gtrsim\rho^{2}$.
Now we perform integration by parts using the identity
\begin{equation}\label{eq:introb2}
    e^{-isk^{2}}e^{i\rho^{2}s}e^{2i\int_{0}^{s}E\left[z(\sigma)\right]\,d\sigma}=-\frac{1}{i\left(k^{2}-\rho^{2}-2E\left[s\right]\right)}\frac{d}{ds}\left(e^{-isk^{2}}e^{i\rho^{2}s}e^{2i\int_{0}^{s}E\left[z(\sigma)\right]\,d\sigma}\right)
\end{equation}
As in the \eqref{eq:example} setting, we need to perform integration by parts twice  using both \eqref{eq:introb1}  and \eqref{eq:introb2} to obtain the desired estimates. Actually, here one can iterate this process and remove all the bulk terms with $\mathsf{a}$ and $\overline{\mathsf{a}}$ involved, see Remark \ref{rem:odenormal}.

Notice that the non-resonant observations \eqref{eq:phaseintro}, \eqref{eq:introb1}  and \eqref{eq:introb2} are crucial and  very different from setting with oscillatory modes
in the NLS setting, see Buslaev-Sulem \cite{BS} and the Ginzburg-Landau
setting with internal modes, see Komech-Kopylova \cite{KK1} and Soffer-Weinstein
\cite{SW}. In those settings with internal modes, additional resonances
will make the weighted estimates have more growth and then make the
decay rate of the radiation much worse.  





Finally,  similar integrations by parts in time are also required to obtain the pointwise bound for the profile,
the second estimate in \eqref{eq:keyest}. This makes the ODE analysis
for the profile much more involved  than the ODE analysis in Kato-Pusateri
\cite{KP}, Germain-Pusateri-Rousset \cite{GPR} and Chen-Pusateri
\cite{CP}. We also need some refined smoothing estimates here.

\subsection{Outline of the paper}
This paper is organized as follows: In Section \ref{sec:prelim}, we collect some basic results on the spectral theory, distorted transforms and linear estimates for the linear Schr\"odinger flow. Then we present the analysis of modulation parameters in Section \ref{sec:Mod}. In Section \ref{sec:reduction}, we set up the analysis for the equation of the radiation term and reduce the analysis to a model problem. In section \ref{sec:weight}, we analyze weighted estimates for the profile. Finally, we show the pointwise bounds for the profile and conclude the bootstrap argument for the model problem. 
In Appendix \ref{sec:NBS}, for the sake of completeness, we record results on the existence of small nonlinear bound states and the comparison of continuous spectral spaces. 
\subsection*{Notations}

As usual, \textquotedblleft $A:=B\lyxmathsym{\textquotedblright}$
or $\lyxmathsym{\textquotedblleft}B=:A\lyxmathsym{\textquotedblright}$
is the definition of $A$ by means of the expression $B$. We use
the notation $\langle x\rangle=\left(1+|x|^{2}\right)^{\frac{1}{2}}$.
For positive quantities $a$ and $b$, we write
$a\lesssim b$ for $a\leq Cb$ where $C$ is some prescribed constant.
Also $a\simeq b$ for $a\lesssim b$ and $b\lesssim a$. Throughout,
we use $u_{t}:=\frac{\partial}{\partial t}u$, $u_{xx}:=\frac{\partial^{2}}{\partial x^{2}}u$.

\subsection{Acknowledgment}
G.C. would like to thank Fabio Pusateri for several stimulating discussions.
G.C. would also like to thank Jiaqi Liu for introducing the topic on
the perturbation of integrable systems to him. G.C. is grateful for Claudio Mu\~noz  and Catherine Sulem for useful suggestions and comments.  Part of this work was
done when G.C. was supported by Fields Institute for Research in Mathematical
Sciences via the Fields-Onatrio postdoc program and Thematic Program
on Mathematical Hydrodynamics. G.C.  would like to thank the Department
of Mathematics, University of Toronto for their hospitality and the
financial support.  G.C. is also grateful to the Department of Mathematics, University of Kentucky for the financial support. G.C.  would also like to
thank the anonymous referees for their helpful comments.

\section{Preliminaries}\label{sec:prelim}

\subsection{Basic properties of general Jost functions}
In this subsection, we recall some basic properties of Jost functions.
The facts provided here hold for all potentials, see 
\cite{DT} and 
\cite{GPR}.

The Jost functions $\psi_+(x,k)$ and $\psi_-(x,k)$
are defined as solutions to 
\begin{align}\label{psipm}
H\psi_{\pm}(x,k)=\left(-\partial_{xx}+V\right)\psi_{\pm}(x,k)=k^{2}\psi_{\pm}(x,k)
\end{align}
such that
\begin{align}\label{psipmlim}
\lim_{x\rightarrow\infty}\left|e^{-ikx}\psi_+(x,k)-1\right|=0,
\qquad  \lim_{x\rightarrow-\infty}\left|e^{ikx}\psi_-(x,k)-1\right|=0. 
\end{align}
We let 
\begin{align}\label{mpm}
m_{\pm}(x,k)=e^{\mp ikx}\psi_{\pm}(x,k).
\end{align}
Then for fixed $x$, $m_{\pm}$ is analytic in $k$ for $\Im k>0$
and continuous up to $\Im k\geq0$.

We define
\begin{align}\label{defW_+-}
\mathcal{W}_{+}^{s}(x)=\int_x^\infty \jy^{s}\left|V(y)\right|\,dy,
  \qquad \mathcal{W}_{-}^{s}(x)=\int_{-\infty}^{x}\jy^{s}\left|V(y)\right|\,dy.
\end{align}
Note that if $V(y)$ decays fast enough, then $\mathcal{W}_{\pm}^{s}(x)$
also decay as $x\rightarrow\pm\infty$ respectively.

\begin{lem}\label{lem:Mestimates}
For every $s\ge0$, we have the estimates: 
\begin{align}\label{Mestimates1}
\begin{split}
& \left|\partial_{k}^{s}\left(m_{\pm}(x,k)-1\right)\right|\lesssim\frac{1}{\jk}\mathcal{W}_{\pm}^{s+1}(x),\qquad \pm x\geq-1,
\\
& \left|\partial_{k}^{s}\left(m_{\pm}(x,k)-1\right)\right|\lesssim\frac{1}{\jk}\jx^{s+1}, \qquad \pm x\leq1.
\end{split}
\end{align}
Moreover
\begin{align}\label{Mestimates2}
\begin{split}
& \left|\partial_{k}^{s} \partial_x m_{\pm}(x,k) \right| \lesssim \mathcal{W}_{\pm}^{s}(x), \qquad \pm x\geq-1,
\\
& \left|\partial_{k}^{s} \partial_x m_{\pm}(x,k) 
  \right|\lesssim\jx^{s}, \qquad \pm x\leq1.
\end{split}
\end{align}
\end{lem}

\begin{proof}
The proofs of these estimates follow from analyzing the Volterra equation satisfied by $m_\pm$,
that is,
\begin{align}\label{minteq}
m_\pm(x,\lambda)=1\pm\int_x^{\pm\infty} D_{\lambda}(\pm(y-x))V(y)m_\pm(y,\lambda)\,dy,
  \qquad D_{\lambda}(x)=\frac{e^{2i\lambda x}-1}{2i\lambda};
\end{align} 
as in Deift-Trubowitz \cite{DT}, Weder \cite{Wed} or \cite[Appendix A]{GPR}. 
\end{proof}

Denote $T(k)$ and $R_{\pm}(k)$ the transmission
and reflection coefficients associated to the potential $V$ respectively.
For more details, see Deift-Trubowitz \cite{DT}. With these coefficients,
one can write
\begin{equation}\label{eq:f_+-1}
\psi_{+}(x,k)=\frac{R_{-}(k)}{T(k)}\psi_{-}(x,k)+\frac{1}{T(k)}\psi_{-}\left(x,-k\right)
\end{equation}
\begin{equation}\label{eq:f_--1}
\psi_{-}(x,k)=\frac{R_{+}(k)}{T(k)}\psi_{+}(x,k)+\frac{1}{T(k)}\psi_{+}\left(x,-k\right).
\end{equation}
Moreover, these coefficients are given explicitly by
\begin{align}\label{defT}
\frac{1}{T(k)}=1-\frac{1}{2ik}\int V(x)m_{\pm}(x,k)\,dx, 
\end{align}
\begin{align}\label{defR}
\frac{R_{\pm}(k)}{T(k)}=\frac{1}{2ik}\int e^{\mp2ikx}V(x)m_{\mp}(x,k)\,dx. 
\end{align}
\begin{defn}\label{def:generic}
$V$ is defined to be a ``generic'' potential if
\[
\int V(x)m_{\pm}\left(x,0\right)\,dx\neq0.
\]
\end{defn}

If a potential is generic, then by the relation given above, we know that
\[
T\left(0\right)=0,\ \ \ R_{\pm}\left(0\right)=-1.
\]
We have the following lemma on the coefficients.

\begin{lem}\label{estiTR}
Assuming that $\jx^{2}V\in L^1$, we have the uniform estimates for $k\in\mathbb{R}$:
\[
\left|\partial_{k}T(k)\right|+\left|\partial_{k}R_{\pm}(k)\right|\lesssim\frac{1}{\jk}.
\]
\end{lem}

Moreover, for a generic potential, the associated transmission and
reflection coefficients have the following Taylor expansions near
$k\sim0$. For a detailed proof, see page 144 in Deift-Trubowitz \cite{DT}.

\begin{lem}\label{lem:estiTRTaylor}
	Assuming that $\jx^2 V\in L^1$ and $V$
	is generic, then
	\[
	T(k)=\alpha k+\mathcal{O}(k^2),\ \alpha\neq0,\ \text{as}\ k\rightarrow0,
	\]
	and
	\[
	1+R_{\pm}(k)=\alpha_{\pm}k+\mathcal{O}(k^2),\ \text{as}\ k\rightarrow0.
	\]
\end{lem}

\subsection{Distorted Fourier transform}\label{ssecDFT}
We recall some basic properties of the distorted Fourier transform
with respect to the perturbed Schr\"odinger operator.
First, recall that the standard Fourier transform is defined, for $\varphi\in L^2$, as
\[
\mathcal{F}\left[\varphi\right](\lambda):=\hat{\varphi}(\lambda)=\frac{1}{\sqrt{2\pi}}\int e^{-i\lambda x}\varphi(x)\,dx
\]
with its inverse as
\[
\mathcal{F}^{-1}\left[\varphi\right](x):=\frac{1}{\sqrt{2\pi}}\int e^{i\lambda x}\varphi(\lambda)\,d\lambda.
\]
Given the Jost functions $\psi_{\pm}$ from \eqref{psipm}, we set
\begin{align}\label{matK}
\mathcal{K}(x,\lambda):=\frac{1}{\sqrt{2\pi}}
\begin{cases}
T(\lambda) \psi_+(x,\lambda) & \lambda\geq0
\\
T(-\lambda) \psi_-(x,-\lambda) & \lambda<0
\end{cases},
\end{align}
and define the ``distorted Fourier transform'' for $f\in \mathcal{S}$ by
\begin{align}\label{tildeF}
\wtF\left[\varphi\right](\lambda) = \widetilde{\varphi}(\lambda) := \int \overline{\mathcal{K}(x,\lambda)} \varphi(x)\,dx.
\end{align}

\begin{lem}\label{lemtildeF}
In our setting, one has
\[
{\big\| \wtF\left[\varphi\right] \big\|}_{L^{2}}=\left\Vert P_c\varphi\right\Vert _{L^{2}},\,\,\forall f\in L^{2}
\]
and
\begin{align}\label{tildeF-1}
\wtF^{-1}\left[\varphi\right](x)=\int\mathcal{K}(x,\lambda)\phi(\lambda)\,d\lambda.
\end{align}
Also, if $D:=\sqrt{-\partial_{xx}+V}$, 
\begin{align}
m(D)P_c=\wtF^{-1}m(\lambda)\wtF.
\end{align}
so that in particular
$\left(-\partial_{xx}+V\right)P_c=\wtF^{-1}\lambda^{2}\wtF$.

\smallskip
We also have the following properties:
	
\setlength{\leftmargini}{2em}
\begin{itemize}

\item[(i)] If $\varphi\in L^1$, then $\widetilde{\varphi}$ is a continuous, bounded function. 

\smallskip
\item[(ii)] If the potential $V$ is generic $\tilde{\phi}\left(0\right)=0$.
  

\smallskip
\item[(iii)] There exists $C>0$ such that one has
\begin{equation*}
\left\Vert \lambda\widetilde{u}\right\Vert _{L^{2}}
  \leq C\left(1+\left(\left\Vert V\right\Vert_{L^1}\right)^{\frac{1}{2}}\right)\left\Vert u\right\Vert _{H^{1}},
\end{equation*}
and
\begin{align}\label{eq:weiF}
\left\Vert \partial_{\lambda}\widetilde{u}\right\Vert _{L^{2}}\leq C\left\Vert \jx u\right\Vert _{L^{2}}.
\end{align}
\end{itemize}

\end{lem}

\begin{proof}
See for example Section 6 in \cite{Agm}, 
\cite{DS,Yaf}, 
and \cite[Lemma 2.4]{GPR}.

\end{proof}

\subsection{Linear estimates}

In this subsection, we recall and collect some linear estimates associated with the linear Schr\"odinger flow.

\subsubsection{Pointwise decay and local $L^2$ decay}

First of all, let us recall the linear dispersive estimate for the free flow:

\begin{lem}\label{lem:linearest}
The linear free Schr\"odinger flow has the following
dispersive estimate: for $t\geq0$
\begin{equation}
\left(e^{-it\partial_{xx}}h\right)(x)
  = \frac{1}{\left(it\right)^{\frac{1}{2}}}e^{i\frac{\left|x\right|^{2}}{2t}}\hat{h}\left(\frac{x}{t}\right)
  +\frac{1}{t^{\frac{1}{2}+b}}\mathcal{O}\left(\left\Vert h\right\Vert _{H^{0,c}}\right)\label{eq:linearasy-1}
\end{equation}
for $x\in\mathbb{R}$ and $c \geq \frac{1}{2}+2b$. 
As a consequence, for $t\geq0$
\begin{equation}\label{eq:linearpoinwise0}
{\big\| e^{-it\partial_{xx}}h \big\|}_{L^\infty_x}
  \lesssim \frac{1}{\sqrt{t}} {\big\| \hat{h} \big\|}_{L^\infty_k}
  +\frac{1}{t^{\frac{3}{4}}} {\big\| \partial_{k}\hat{h} \big\|}_{L^2_k}.
\end{equation}
\end{lem}

The above free dispersive estimate can be extended to the perturbed
flow after projecting onto the continuous spectrum, see Goldberg-Schlag
\cite{GSch} and Germain-Pusateri-Rousset \cite{GPR}. 

\begin{lem}\label{lem:pointwiseH}
Suppose $\jx^\gamma V(x)\in L^1$ with $\gamma \geq 1$. The perturbed Schr\"odinger flow has the following dispersive estimate:
for $t\geq0$
\begin{equation}\label{eq:linearpoinwiseH}
{\big\| e^{itH}P_c h \big\|}_{L^\infty_x} \lesssim\frac{1}{\sqrt{t}}{\big\| \tilde{h} \big\|}_{L^\infty_k}
  +\frac{1}{t^{\frac{3}{4}}}{\big\| \partial_k\tilde{h} \big\|} _{L^2_k}.
\end{equation}
\end{lem}
Then we recall an improved decay estimate after localizing in the space. For the proof, see Lemma A.1 in Chen-Pusateri \cite{CP}.
\begin{lem}[Improved $L^\infty$ local decay]\label{lemlocdecinfty}
Suppose $\jx^\gamma V(x)\in L^1$ with $\gamma \geq 3$ and $V$ is generic. 
Then, for the perturbed flow one has
\begin{align}\label{locdecinfty0}
\left\Vert \jx^{-2} e^{iHt}P_c h \right\Vert_{L^{\infty}_x}
  \lesssim |t|^{-1}{\big\| \tilde{h} \big\|}_{H^1_k}.
\end{align}
\end{lem}
Finally, we recall the local $L^2$ decay estimate of the derivative from Lemma 3.12 in \cite{CP}.
\begin{lem}[Improved $L^2$ local decay]\label{lem:localEn}
Assuming $\jx^\gamma V(x)\in L^1$, $\gamma\geq2$, is generic, 
we have
\begin{align}
\label{locdecL2}
{\big\| \jx^{-1} \partial_x
  \big(e^{itH}P_c h \big) \big\|}_{L^2_x}\lesssim |t|^{-1} {\big\| \tilde{h} \big\|}_{H^1_k}.
\end{align}
\end{lem}

\subsubsection{Local smoothing estimates}

Finally, we collect some local decay or smoothing estimates which
can also be regarded as Kato smoothing estimates. Although the
first two estimates are shown in Mizumachi \cite{Mi}. Here we provide
a proof based on distorted transforms.
\begin{lem}
\label{lem:smoothing}Suppose $\jx^\gamma V(x)\in L^1$ with $\gamma \geq 1$ is generic, then
we have the following estimate
\begin{equation}
\left\Vert \left\langle x\right\rangle ^{-1}e^{iHt}P_{c}h\right\Vert _{L_{x}^{\infty}L_{t}^{2}}\lesssim\left\Vert h\right\Vert _{L^{2}}.\label{eq:smoothing1}
\end{equation}
For inhomogeneous estimate, we have
\begin{equation}
\left\Vert \int e^{isH}P_{c}F\left(s\right)\,ds\right\Vert _{L_{x}^{2}}\lesssim\left\Vert \left\langle x\right\rangle F\right\Vert _{L_{x}^{1}L_{t}^{2}}\label{eq:smoothing2}
\end{equation}
\end{lem}

\begin{proof}
We prove \eqref{eq:smoothing1} first and then \eqref{eq:smoothing2}
will follow by duality. Using the distorted basis, we write
\[
e^{itH}P_{c}h=\int\mathcal{K}\left(x,k\right)e^{itk^{2}}\tilde{h}\left(k\right)\,dk.
\]
We only discuss $k\geq0$ since the other piece would be similar.
For $k\geq0$, we have
\[
\int_{k\geq0}\mathcal{K}\left(x,k\right)e^{itk^{2}}\tilde{h}\left(k\right)\,dk=\int T\left(k\right)e^{ikx}m_{+}\left(x,k\right)e^{itk^{2}}\tilde{h}\left(k\right)\,dk.
\]
Making a change of variable, $k^{2}=\lambda$, $2kdk=d\lambda$, one
has
\begin{align}
\int T\left(k\right)e^{ikx}m_{+}\left(x,k\right)e^{itk^{2}}\tilde{h}\left(k\right)\,dk=\label{eq:smcom0}\\
\int T\left(\sqrt{\lambda}\right)e^{i\sqrt{\lambda}x}m_{+}\left(x,\sqrt{\lambda}\right)e^{it\lambda}\tilde{h}\left(\sqrt{\lambda}\right)\frac{1}{2\sqrt{\lambda}}\,d\lambda.\nonumber 
\end{align}
Note that when $k$ is small, $T\left(k\right)\sim k$ since $V$
is generic.

Applying Plancherel's theorem in $t$, to obtain the $L^{2}$ estimate
in $t$, it suffices to estimate the $L^{2}$ norm of
\[
T\left(\sqrt{\lambda}\right)e^{i\sqrt{\lambda}x}m_{+}\left(x,\sqrt{\lambda}\right)\tilde{h}\left(\sqrt{\lambda}\right)\frac{1}{2\sqrt{\lambda}}
\]
in $\lambda$. Note that
\begin{equation}
\left\Vert \left\langle x\right\rangle ^{-1}\frac{1}{2\sqrt{\lambda}}T\left(\sqrt{\lambda}\right)e^{i\sqrt{\lambda}x}m_{+}\left(x,\sqrt{\lambda}\right)\right\Vert _{L_{x,\lambda}^{\infty}}<\infty\label{eq:smcom1}
\end{equation}
by the estimate for Jost functions, Lemma \ref{lem:Mestimates}.  Therefore, after localizing $x$,
it suffices to compute
\begin{align}
\int_{0}^{\infty}\left(\frac{1}{2\sqrt{\lambda}}T\left(\sqrt{\lambda}\right)\right)^{2}\left|\tilde{h}\left(\sqrt{\lambda}\right)\right|^{2}\,d\lambda & \sim\int_{0}^{\infty}\left|\tilde{h}\left(k\right)\right|^{2}\,dk \sim\left\Vert \tilde{h}\right\Vert _{L^{2}}\label{eq:smcomp2}
\end{align}
where we used $T\left(\sqrt{\lambda}\right)\sim\sqrt{\lambda}$ for
$\lambda$ small and $T\left(\sqrt{\lambda}\right)\sim1$ for $\lambda$
large.

Putting \eqref{eq:smcom0}, \eqref{eq:smcom1}, \eqref{eq:smcomp2} and
similar computations for $k\leq0$, we conclude that
\[
\left\Vert \left\langle x\right\rangle ^{-1}e^{iHt}P_{c}h\right\Vert _{L_{x}^{\infty}L_{t}^{2}}\lesssim\left\Vert h\right\Vert _{L^{2}}
\]
and
\begin{equation}
\left\Vert \int e^{isH}P_{c}F\left(s\right)\,ds\right\Vert _{L_{x}^{2}}\lesssim\left\Vert \left\langle x\right\rangle F\right\Vert _{L_{x}^{1}L_{t}^{2}}\label{eq:smoothing2-1}
\end{equation}
follows by duality.
\end{proof}
\begin{rem}
One can restrict estimates above onto the finite interval. In other
words, for any $T$, we have
\begin{equation}
\left\Vert \left\langle x\right\rangle ^{-1}e^{iHt}P_{c}h\right\Vert _{L_{x}^{\infty}L_{t}^{2}\left[0,T\right]}\lesssim\left\Vert h\right\Vert _{L^{2}}.\label{eq:smoothing1-1}
\end{equation}
and the inhomogeneous estimate version
\begin{equation}
\left\Vert \int_{0}^{T}e^{isH}P_{c}F\left(s\right)\,ds\right\Vert _{L_{x}^{2}}\lesssim\left\Vert \left\langle x\right\rangle F\right\Vert _{L_{x}^{1}L_{t}^{2}}.\label{eq:smoothing2-2}
\end{equation}
\end{rem}

\begin{rem}
In Mizumachi \cite{Mi}, a retarted inhomogenous smoothing estimate
is also established:
\[
\left\Vert \left\langle x\right\rangle ^{-1}\int_{0}^{t}e^{i\left(t-s\right)H}P_{c}F\left(s\right)\,ds\right\Vert _{L_{x}^{2}}\lesssim\left\Vert \left\langle x\right\rangle F\right\Vert _{L_{x}^{1}L_{t}^{2}}.
\]
This can be also shown by the Laplace transform of resolvents and
using the limiting absorption principle or the boundedness of the
resolvent in weighted space. We refer to our later analysis Lemma
\ref{lem:limitingab} and its application.
\end{rem}

Next we provide the smoothing estimates which can absorb an additional
growth in $t$.
\begin{lem}\label{lem:imprsmoothing}
 Suppose $\jx^\gamma V(x)\in L^1$ with $\gamma \geq 2$ is generic. Let $\phi_{1}$ and $\phi_{2}$ be smooth nonegative functions such
that $\phi_{1}+\phi_{2}=1$, $\phi_{1}\left(\lambda\right)=1$ for
$\left|\lambda\right|\leq1$ and $\phi_{1}\left(\lambda\right)=0$
for $\left|\lambda\right|\ge2$. Then we have the following estimates
\begin{equation}
\left\Vert \left\langle x\right\rangle ^{-2}t\partial_{t}^j\left(e^{itH}\phi_{1}\left(H\right)P_ch\right)\right\Vert _{L_{x}^{\infty}L_{t}^{2}}\lesssim\left\Vert \tilde{h}\right\Vert _{H^{1}},\,j=1,2\label{eq:impsmL}
\end{equation}
and
\begin{equation}
\left\Vert \left\langle x\right\rangle ^{-2}t\partial_{x}^{j}\left(e^{itH}\phi_{2}\left(H\right)P_c h\right)\right\Vert _{L_{x}^{\infty}L_{t}^{2}}\lesssim\left\Vert \tilde{h}\right\Vert _{H^{1}},\ j=0,1.\label{eq:impsmH}
\end{equation}
\end{lem}

\begin{rem}
Note that in the homogeneous case, the decomposition of $\phi_{1}\left(H\right)$
and $\phi_{2}\left(H\right)$ actually is equivalent to decompose
the time Fourier transform of the solution
\[
\mathcal{F}_{t}\left[e^{itH}h\right]\left(\tau\right)=\int e^{-it\tau}\left(e^{itH}h\right)\,dt
\]
in to low and high frequency parts. We also notice that given $\tilde{h}\in H^1$, by Lemma \ref{lem:smoothing}, the Fourier transform with respect to time is well-defined. 
\end{rem}

\begin{proof}
We only prove \eqref{eq:impsmL} with $j=1$ since in the region we are interested in the analysis for $j=2$ is basically the same. Using the distorted Fourier transform, one can write
\[
t\partial_{t}\left(e^{itH}\phi_{1}\left(H\right)h\right)=t\int\mathcal{K}\left(x,k\right)ik^{2}\phi_{1}\left(k^{2}\right)e^{itk^{2}}\tilde{h}\left(k\right)\,dk.
\]
We only discuss $k\geq0$ since the other piece would be similar.
For $k\geq0$, we have
\[
\int\mathcal{K}\left(x,k\right)ik^{2}\phi_{1}\left(k^{2}\right)e^{itk^{2}}\tilde{h}\left(k\right)\,dk=\int T\left(k\right)e^{ikx}ik^{2}m_{+}\left(x,k\right)\phi_{1}\left(k^{2}\right)e^{itk^{2}}\tilde{h}\left(k\right)\,dk.
\]
Making a change of variable, $k^{2}=\lambda$, $2kdk=d\lambda$, one
has{\footnotesize`
\[
\int T\left(k\right)e^{ikx}ik^{2}m_{+}\left(x,k\right)\phi_{1}\left(k^{2}\right)e^{itk^{2}}\tilde{h}\left(k\right)\,dk\sim\int\frac{T\left(\sqrt{\lambda}\right)}{2\sqrt{\lambda}}e^{i\sqrt{\lambda}x}i\lambda\phi_{1}\left(\lambda\right)m_{+}\left(x,\sqrt{\lambda}\right)e^{it\lambda}\tilde{h}\left(\sqrt{\lambda}\right)\,d\lambda.
\]}Note that when $k$ is small, $T\left(k\right)\sim k$. Multiplying
the above expression by $t$ and then performing integration by parts
in $\lambda$, we obtain{\footnotesize
\begin{align}
t\int\lambda\frac{T\left(\sqrt{\lambda}\right)}{2\sqrt{\lambda}}e^{i\sqrt{\lambda}x}i\phi_{1}\left(\lambda\right)m_{+}\left(x,\sqrt{\lambda}\right)e^{it\lambda}\tilde{h}\left(\sqrt{\lambda}\right)\,d\lambda\label{eq:IBPL1}\\
\sim i\int\partial_{\lambda}\left(\lambda\frac{T\left(\sqrt{\lambda}\right)}{2\sqrt{\lambda}}i\phi_{1}\left(\lambda\right)\right)e^{i\sqrt{\lambda}x}m_{+}\left(x,\sqrt{\lambda}\right)e^{it\lambda}\tilde{h}\left(\sqrt{\lambda}\right)\,d\lambda\nonumber \\
-\int\left(\lambda\frac{T\left(\sqrt{\lambda}\right)}{2\sqrt{\lambda}}i\phi_{1}\left(\lambda\right)\right)\frac{1}{\sqrt{\lambda}}xe^{i\sqrt{\lambda}x}m_{+}\left(x,\sqrt{\lambda}\right)e^{it\lambda}\tilde{h}\left(\sqrt{\lambda}\right)\,d\lambda\nonumber \\
-i\int\left(\lambda\frac{T\left(\sqrt{\lambda}\right)}{2\sqrt{\lambda}}i\phi_{1}\left(\lambda\right)\right)\frac{1}{\sqrt{\lambda}}e^{i\sqrt{\lambda}x}m'_{+}\left(x,\sqrt{\lambda}\right)e^{it\lambda}\tilde{h}\left(\sqrt{\lambda}\right)\,d\lambda\nonumber \\
-i\int\left(\lambda\frac{T\left(\sqrt{\lambda}\right)}{2\sqrt{\lambda}}i\phi_{1}\left(\lambda\right)\right)\frac{1}{\sqrt{\lambda}}e^{i\sqrt{\lambda}x}m_{+}\left(x,\sqrt{\lambda}\right)e^{it\lambda}\tilde{h}'\left(\sqrt{\lambda}\right)\,d\lambda & .\nonumber 
\end{align}}
Here we use $m'$ to denote the differentiation with respect to the
second variable.

To estimate the $L^{2}$ norm with respect to $t$, it suffices to
estimate the $L_{\lambda}^{2}$ norm of the following expressions
from the RHS of \eqref{eq:IBPL1}:{\footnotesize
\[
\partial_{\lambda}\left(\lambda\frac{T\left(\sqrt{\lambda}\right)}{2\sqrt{\lambda}}i\phi_{1}\left(\lambda\right)\right)e^{i\sqrt{\lambda}x}m_{+}\left(x,\sqrt{\lambda}\right)e^{it\lambda}\tilde{h}\left(\sqrt{\lambda}\right)
\]
\[
\left(\lambda\frac{T\left(\sqrt{\lambda}\right)}{2\sqrt{\lambda}}i\phi_{1}\left(\lambda\right)\right)\frac{1}{\sqrt{\lambda}}xe^{i\sqrt{\lambda}x}m_{+}\left(x,\sqrt{\lambda}\right)e^{it\lambda}\tilde{h}\left(\sqrt{\lambda}\right)
\]
\[
\left(\lambda\frac{T\left(\sqrt{\lambda}\right)}{2\sqrt{\lambda}}i\phi_{1}\left(\lambda\right)\right)\frac{1}{\sqrt{\lambda}}e^{i\sqrt{\lambda}x}m'_{+}\left(x,\sqrt{\lambda}\right)e^{it\lambda}\tilde{h}\left(\sqrt{\lambda}\right)
\]}
and{\footnotesize
\[
\left(\lambda\frac{T\left(\sqrt{\lambda}\right)}{2\sqrt{\lambda}}i\phi_{1}\left(\lambda\right)\right)\frac{1}{\sqrt{\lambda}}e^{i\sqrt{\lambda}x}m_{+}\left(x,\sqrt{\lambda}\right)e^{it\lambda}\tilde{h}'\left(\sqrt{\lambda}\right)
\]}
 by Plancherel's theorem.

To estimate the first piece, we just notice that
\begin{align*}
\int\left|\partial_{\lambda}\left(\lambda\frac{T\left(\sqrt{\lambda}\right)}{2\sqrt{\lambda}}i\phi_{1}\left(\lambda\right)\right)e^{i\sqrt{\lambda}x}m_{+}\left(x,\sqrt{\lambda}\right)e^{it\lambda}\tilde{h}\left(\sqrt{\lambda}\right)\right|^{2}\,d\lambda & \lesssim\left\langle x\right\rangle \int\frac{1}{k}\left|\tilde{h}\left(k\right)\right|^{2}\,dk\\
 & \lesssim\left\langle x\right\rangle \left\Vert \partial_{k}\tilde{h}\right\Vert _{H^{1}}^{2}
\end{align*}
where in the last inequality, we applied Hardy's inequality or the fundamental
theorem of calculus with the fact $\tilde{h}$$\left(0\right)=0$. 

The second one can be estimated similarly. We should have
\begin{align*}
\int\left|\left(\lambda\frac{T\left(\sqrt{\lambda}\right)}{2\sqrt{\lambda}}i\phi_{1}\left(\lambda\right)\right)\frac{1}{\sqrt{\lambda}}xe^{i\sqrt{\lambda}x}m_{+}\left(x,\sqrt{\lambda}\right)e^{it\lambda}\tilde{h}\left(\sqrt{\lambda}\right)\right|^{2}\,d\lambda\\
\lesssim\left\langle x\right\rangle ^{4}\int\lambda\phi_{1}\left(\lambda\right)\left|\tilde{h}\left(\sqrt{\lambda}\right)\right|^{2}\,d\lambda\lesssim\left\langle x\right\rangle ^{4}\left\Vert \tilde{h}\right\Vert _{L^{2}}^{2} & .
\end{align*}
The third one can be treated identically as the second piece. One
has
\begin{align*}
\int\left|\left(\lambda\frac{T\left(\sqrt{\lambda}\right)}{2\sqrt{\lambda}}i\phi_{1}\left(\lambda\right)\right)\frac{1}{\sqrt{\lambda}}e^{i\sqrt{\lambda}x}m'_{+}\left(x,\sqrt{\lambda}\right)e^{it\lambda}\tilde{h}\left(\sqrt{\lambda}\right)\right|^{2}\,d\lambda\\
\lesssim\left\langle x\right\rangle ^{4}\int\lambda\phi_{1}\left(\lambda\right)\left|\tilde{h}\left(\sqrt{\lambda}\right)\right|^{2}\,d\lambda\lesssim\left\langle x\right\rangle ^{4}\left\Vert \tilde{h}\right\Vert _{L^{2}}^{2} & .
\end{align*}
It remains to check the last piece:
\begin{align*}
\int\left|\left(\lambda\frac{T\left(\sqrt{\lambda}\right)}{2\sqrt{\lambda}}i\phi_{1}\left(\lambda\right)\right)\frac{1}{\sqrt{\lambda}}e^{i\sqrt{\lambda}x}m_{+}\left(x,\sqrt{\lambda}\right)e^{it\lambda}\tilde{h}'\left(\sqrt{\lambda}\right)\right|^{2}\,d\lambda\\
\lesssim\left\langle x\right\rangle ^{2}\int\phi_{1}^{2}\left(\lambda\right)\left|\tilde{h}'\left(\sqrt{\lambda}\right)\right|^{2}\,d\lambda
\lesssim\left\langle x\right\rangle ^{2}\left\Vert \partial_{k}\tilde{h}\right\Vert _{L^{2}}^{2} & .
\end{align*}
Adding everything together, we conclude that
\[
\left\Vert \left\langle x\right\rangle ^{-2}t\partial_{t}\left(e^{itH}\phi_{1}\left(H\right)P_c h\right)\right\Vert _{L_{x}^{\infty}L_{t}^{2}}\lesssim\left\Vert \tilde{h}\right\Vert _{H^{1}}.
\]
For the high frequency part, the same analysis can be applied. In
the support of $\phi_{2}$, there is no singularities introduced by
the change of variable $2kdk=d\lambda$ and the differentiation $\partial_{\lambda}$.
Directly, one has
\[
\left\Vert \left\langle x\right\rangle ^{-2}te^{itH}\phi_{2}\left(H\right)P_c h\right\Vert _{L_{x}^{\infty}L_{t}^{2}}\lesssim\left\Vert \tilde{h}\right\Vert _{H^{1}}.
\]
We check the estimate \eqref{eq:impsmH} for $j=1$. Again using the
distorted Fourier transform, one can write
\[
t\partial_{x}\left(e^{itH}\phi_{2}\left(H\right)h\right)=t\partial_{x}\int\mathcal{K}\left(x,k\right)\phi_{2}\left(k^{2}\right)e^{itk^{2}}\tilde{h}\left(k\right)\,dk.
\]
We only discuss $k\geq0$ since the other piece would be similar.
For $k\geq0$, we have
\begin{align}
\partial_{x}\int\mathcal{K}\left(x,k\right)\phi_{2}\left(k^{2}\right)e^{itk^{2}}\tilde{h}\left(k\right)\,dk & =\int T\left(k\right)ike^{ikx}m_{+}\left(x,k\right)\phi_{2}\left(k^{2}\right)e^{itk^{2}}\tilde{h}\left(k\right)\,dk\label{eq:partialxH}\\
 & +\int T\left(k\right)e^{ikx}\partial_{k}m_{+}\left(x,k\right)\phi_{2}\left(k^{2}\right)e^{itk^{2}}\tilde{h}\left(k\right)\,dk\nonumber 
\end{align}
The second term on the RHS of \eqref{eq:partialxH} can be estimated
by the same argument as $j=0$. The reason we need to pay some additional
attention to the first term is that there is an extra $k$ in the
integrand. We need to make sure that it will not introduce extra weights
in the final estimate.

As before, making a change of variable, $k^{2}=\lambda$, $2kdk=d\lambda$,
one has
\begin{align*}
\int T\left(k\right)ike^{ikx}m_{+}\left(x,k\right)\phi_{2}\left(k^{2}\right)e^{itk^{2}}\tilde{h}\left(k\right)\,dk\\
\sim\int\frac{T\left(\sqrt{\lambda}\right)}{2\sqrt{\lambda}}e^{i\sqrt{\lambda}x}i\sqrt{\lambda}\phi_{2}\left(\lambda\right)m_{+}\left(x,\sqrt{\lambda}\right)e^{it\lambda}\tilde{h}\left(\sqrt{\lambda}\right)\,d\lambda\\
\sim\int T\left(\sqrt{\lambda}\right)e^{i\sqrt{\lambda}x}\phi_{2}\left(\lambda\right)m_{+}\left(x,\sqrt{\lambda}\right)e^{it\lambda}\tilde{h}\left(\sqrt{\lambda}\right)\,d\lambda
\end{align*}
Multiplying the above expression by $t$ and then performing integration
by parts in $\lambda$, we obtain{\small
\begin{align}
t\int T\left(\sqrt{\lambda}\right)e^{i\sqrt{\lambda}x}\phi_{2}\left(\lambda\right)m_{+}\left(x,\sqrt{\lambda}\right)e^{it\lambda}\tilde{h}\left(\sqrt{\lambda}\right)\,d\lambda\label{eq:IBPL1-1}\\
\sim i\int\partial_{\lambda}\left(T\left(\sqrt{\lambda}\right)i\phi_{2}\left(\lambda\right)\right)e^{i\sqrt{\lambda}x}m_{+}\left(x,\sqrt{\lambda}\right)e^{it\lambda}\tilde{h}\left(\sqrt{\lambda}\right)\,d\lambda\nonumber \\
-\int\left(T\left(\sqrt{\lambda}\right)i\phi_{2}\left(\lambda\right)\right)\frac{1}{\sqrt{\lambda}}xe^{i\sqrt{\lambda}x}m_{+}\left(x,\sqrt{\lambda}\right)e^{it\lambda}\tilde{h}\left(\sqrt{\lambda}\right)\,d\lambda\nonumber \\
-i\int\left(T\left(\sqrt{\lambda}\right)i\phi_{2}\left(\lambda\right)\right)\frac{1}{\sqrt{\lambda}}e^{i\sqrt{\lambda}x}m'_{+}\left(x,\sqrt{\lambda}\right)e^{it\lambda}\tilde{h}\left(\sqrt{\lambda}\right)\,d\lambda\nonumber \\
-i\int\left(T\left(\sqrt{\lambda}\right)i\phi_{2}\left(\lambda\right)\right)\frac{1}{\sqrt{\lambda}}e^{i\sqrt{\lambda}x}m_{+}\left(x,\sqrt{\lambda}\right)e^{it\lambda}\tilde{h}'\left(\sqrt{\lambda}\right)\,d\lambda & .\nonumber 
\end{align}}To estimate the $L^{2}$ norm of \eqref{eq:partialxH} with respect to $t$, it suffices to
estimate the $L_{\lambda}^{2}$ norm of the following expressions
from the RHS of \eqref{eq:IBPL1-1}:
\[
\partial_{\lambda}\left(T\left(\sqrt{\lambda}\right)i\phi_{2}\left(\lambda\right)\right)e^{i\sqrt{\lambda}x}m_{+}\left(x,\sqrt{\lambda}\right)e^{it\lambda}\tilde{h}\left(\sqrt{\lambda}\right)
\]
\[
\left(T\left(\sqrt{\lambda}\right)i\phi_{2}\left(\lambda\right)\right)\frac{1}{\sqrt{\lambda}}xe^{i\sqrt{\lambda}x}m_{+}\left(x,\sqrt{\lambda}\right)e^{it\lambda}\tilde{h}\left(\sqrt{\lambda}\right)
\]
\[
\left(T\left(\sqrt{\lambda}\right)i\phi_{2}\left(\lambda\right)\right)\frac{1}{\sqrt{\lambda}}e^{i\sqrt{\lambda}x}m'_{+}\left(x,\sqrt{\lambda}\right)e^{it\lambda}\tilde{h}\left(\sqrt{\lambda}\right)
\]
and
\[
\left(T\left(\sqrt{\lambda}\right)i\phi_{2}\left(\lambda\right)\right)\frac{1}{\sqrt{\lambda}}e^{i\sqrt{\lambda}x}m_{+}\left(x,\sqrt{\lambda}\right)e^{it\lambda}\tilde{h}'\left(\sqrt{\lambda}\right)
\]
 by Plancherel's theorem. All the computations are similar to those
for the low frequency part. We only check some pieces.

To estimate the first piece, we just notice that{\small
\begin{align*}
\int\left|\partial_{\lambda}\left(T\left(\sqrt{\lambda}\right)i\phi_{2}\left(\lambda\right)\right)e^{i\sqrt{\lambda}x}m_{+}\left(x,\sqrt{\lambda}\right)e^{it\lambda}\tilde{h}\left(\sqrt{\lambda}\right)\right|^{2}\,d\lambda & \lesssim\left\langle x\right\rangle \int\frac{1}{k}\left|\tilde{h}\left(k\right)\right|^{2}\,dk\\
 & \lesssim\left\langle x\right\rangle \left\Vert \partial_{k}\tilde{h}\right\Vert _{H^{1}}^{2}
\end{align*}}
where in the last inequality, we applied Hardy's inequality.

The second one can be estimate similarly. We should have
\begin{align*}
\int\left|\left(T\left(\sqrt{\lambda}\right)i\phi_{2}\left(\lambda\right)\right)\frac{1}{\sqrt{\lambda}}xe^{i\sqrt{\lambda}x}m_{+}\left(x,\sqrt{\lambda}\right)e^{it\lambda}\tilde{h}\left(\sqrt{\lambda}\right)\right|^{2}\,d\lambda\\
\lesssim\left\langle x\right\rangle ^{2}\int\frac{1}{\lambda}\phi_{2}\left(\lambda\right)^{2}\left|\tilde{h}\left(\sqrt{\lambda}\right)\right|^{2}\,d\lambda\lesssim\left\langle x\right\rangle ^{2}\int\frac{1}{k}\phi_{2}\left(k^{2}\right)^{2}\left|\tilde{h}\left(k\right)\right|^{2}\,dk & .
\end{align*}
For the last piece:
\begin{align*}
\int\left|\left(T\left(\sqrt{\lambda}\right)i\phi_{2}\left(\lambda\right)\right)\frac{1}{\sqrt{\lambda}}e^{i\sqrt{\lambda}x}m_{+}\left(x,\sqrt{\lambda}\right)e^{it\lambda}\tilde{h}'\left(\sqrt{\lambda}\right)\right|^{2}\,d\lambda\\
\lesssim\int\frac{1}{\lambda}\phi_{2}\left(\lambda\right)^{2}\left|\tilde{h}'\left(\sqrt{\lambda}\right)\right|^{2}\,d\lambda
\lesssim\int\frac{1}{k}\phi_{2}\left(k^{2}\right)^{2}\left|\tilde{h}'\left(k\right)\right|^{2}\,dk\lesssim\left\Vert \tilde{h}\right\Vert _{H^{1}}^{2} & .
\end{align*}
Adding everything together, we conclude that
\[
\left\Vert \left\langle x\right\rangle ^{-2}t\partial_{x}e^{itH}\phi_{2}\left(H\right)P_c h\right\Vert _{L_{x}^{\infty}L_{t}^{2}}\lesssim\left\Vert \tilde{h}\right\Vert _{H^{1}}
\]
as desired.
\end{proof}
\section{Analysis of modulation parameters }\label{sec:Mod}

In this section, we show that for every small solution to our original NLS,
\[
i\partial_{t}u-\partial_{xx}u+Vu=\left|u\right|^{2}u
\]
we can choose the modulation parameter for the soliton such that the radiation term
is orthogonal to the non-decaying solutions to the time-dependent Hamiltonian given by the soliton.

We recall that the solitary is constructed via the solution to the nonlinear elliptic equation
\[
\left(-\partial_{xx}+V\right)Q[z]-\left|Q[z]\right|^{2}Q[z]=E[z]Q[z].
\]
As the discussion in the introduction, linearzing the NLS around the soliton given by $Q[z]$, we will have the $z$ dependent
spectral problem
\[
\mathrm{H}\left[z\right]\eta=\left(-\partial_{xx}+V\right)\eta-2\left|Q\left[z\right]\right|^{2}\eta-\left(Q\left[z\right]\right)^{2}\bar{\eta}
\]
Recall that from Definition \ref{def:Conti}, we have the continuous spectral subspace with respect
to $\mathrm{H}\left[z\right]$:
\[
\mathcal{H}_{c}\left[z\right]:=\left\{ \eta\in L^{2}:\,\left\langle i\eta,\text{D}_{1}Q\left[z\right]\right\rangle =\left\langle i\eta,\text{D}_{2}Q\left[z\right]\right\rangle =0\right\} .
\]
For $u$ as a solution to the NLS, we decompose
\[
u=Q\left(x,z\left(t\right)\right)+\eta\left(t\right).
\]
We will use the time-dependent modulation parameters $z\left(t\right)$ to ensure
the orthogonal conditions for $t\geq0$\begin{equation}\label{eq:orthcond}
    \left\langle i\eta,\text{D}_{1}Q\left[z\right]\right\rangle =\left\langle i\eta,\text{D}_{2}Q\left[z\right]\right\rangle =0.
\end{equation}
The initial decomposition at $t=0$ can be obtained by the means of the
implicit value theorem (or inverse function theorem). Then we evolve
$z\left(t\right)$ according to orthogonality conditions, \eqref{eq:orthcond}, above.
\begin{lem}[Existence of the initial decomposition] There exists $\delta>0$
such that $\forall u\in H^{1}$ with $\left\Vert u\right\Vert _{H^{1}}\leq\delta$
can be uniquely written as 
\[
u\left(T\right)=Q[z]+\eta
\]
where $z\in\mathbb{C}$, such that
\[
\left\langle i\eta,\text{D}_{1}Q\left[z\right]\right\rangle =\left\langle i\eta,\text{D}_{2}Q\left[z\right]\right\rangle =0.
\]
 and
\[
\left|z\right|+\left\Vert \eta\right\Vert _{H^{1}}\lesssim\left\Vert u\right\Vert _{H^{1}}.
\]
\end{lem}

\begin{proof}
We define the map $K:\,\left\{ z\in\mathbb{C},\left|z\right|<\delta\right\} \times H^{1}\longmapsto\mathbb{R}^{2}$
as
\[
\left(z,u\right)\rightarrow K\left(z,u\right)=\left(K_{1}\left(z,u\right),K_{2}\left(z,u\right)\right)
\]
where
\[
K_{j}\left(z,u\right):=\left\langle i\left(u-Q\left[z\right]\right),\text{D}_{j}Q\left[z\right]\right\rangle .
\]
We first take the partial derivatives of $K\left(z,u\right)$ with respect
to $z$. Note that by Lemma \ref{lem:NLB},
\begin{align*}
\text{D}_{\ell}K_{j}\left(z\right) & =\left\langle -i\text{D}_{\ell}Q\left[z\right],\text{D}_{j}Q\left[z\right]\right\rangle +\left\langle i\left(u-Q\left[z\right]\right),\text{D}_{\ell}\text{D}_{j}Q\left[z\right]\right\rangle \\
 & =j-\ell+o\left(\left\Vert u\right\Vert _{H^{1}}+\left|z\right|\right).
\end{align*}
Therefore the partial derivatives of $K$ with respect to $z$ are
not singular provided that $\left\Vert u\right\Vert _{H^{1}}$ and $\left|z\right|$
are small, i.e., $\delta$ is sufficiently small.

Then by the implicit function theorem, we can find a map 
\[
\tilde{K}:\left\{ u\in H^{1}.\left\Vert u\right\Vert _{H^{1}}<\delta\right\} \longmapsto\mathbb{C}
\]
 such that
\[
K\left(z,u\right)=\left(0,0\right)
\]
is uniquely given as 
 $
\left(\tilde{K}\left(u\right),u\right).
$
Clearly, one has
\[
\left|\tilde{K}\left(u\right)\right|\lesssim\left\Vert u\right\Vert _{H^{1}}.
\]
We define $z=\tilde{K}\left(u\right)$ and decompose $u$ as
\[
u\left(T\right)=Q[z]+\eta
\]
then by construction $K\left(z,u\right)=\left(0,0\right)$ which is
equivalent to
\[
\left\langle i\eta,\text{D}_{1}Q\left[z\right]\right\rangle =\left\langle i\eta,\text{D}_{2}Q\left[z\right]\right\rangle =0.
\]
We are done.
\end{proof}

\subsection{Analysis of modulation equations\label{subsec:AnalMod}}

In this subsection, we  analyze the equations for the modulation parameters. As we want
to analyze the long-time behavior, we might assume now for $t$ large,
we can decompose the solution as
\[
u=Q\left(x,z\left(t\right)\right)+\eta\left(t\right)
\]
From the orthogonality conditions, we need
\[
\left\langle i\eta,\text{D}_{1}Q\left[z\right]\right\rangle =\left\langle i\eta,\text{D}_{2}Q\left[z\right]\right\rangle =0.
\]
In order to ensure these conditions, we need to make $z$ time-dependent.
The equations for $z$ will be obtained by differentiating the above
conditions.

Differentiating the orthogonality conditions with respect to $t$,
from the first condition, we obtain
\[
0=\left\langle i\dot{\eta},\text{D}_{1}Q\left[z\right]\right\rangle +\left\langle i\eta,\text{D}_{1}\text{D}Q\left[z\right]\dot{z}\right\rangle .
\]
From the equation for $\eta$, we have
\begin{align*}
i\partial_{t}\eta-\partial_{xx}\eta+V\left(x\right)\eta & =2\left|Q\left[z\right]\right|^{2}\eta+\left(Q\left[z\right]\right)^{2}\bar{\eta}
  +E\left[z\right]Q\left[z\right]-i\text{D}Q\left[z\right]\dot{z}
  +N\left(\eta,Q\left(x,z\left(t\right)\right)\right).
\end{align*}
Note that by the gauge covariance, one has
\[
\left\langle E\left[z\right]Q\left[z\right]-i\text{D}Q\left[z\right]\dot{z},\text{D}_{j}Q\left[z\right]\right\rangle =\left\langle \text{D}Q\left[z\right]\left(izE\left[z\right]-\dot{z}\right),i\text{D}Q\left[z\right]\right\rangle ,
\]
here we used
\[
\text{D}Q\left[z\right]iz=iQ\left[z\right].
\]
The orthogonality conditions give us
\[
\left\langle i\eta,\text{D}Q\left[z\right]\right\rangle =0.
\]
Hence
\[
\left\langle \mathrm{H}[z]\eta,\text{D}_{j}Q\right\rangle =\left\langle \eta,\mathrm{H}[z]\text{D}_{j}Q\right\rangle =\left\langle \eta,\text{D}_{j}\left(EQ\right)\right\rangle =\left\langle \eta,E\text{D}_{j}Q\right\rangle =\left\langle i\eta,E\text{D}_{j}\mathrm{D}Qiz\right\rangle .
\]
So we can conclude that
\[
\sum_{k=1,2}\left(\left\langle i\text{D}_{j}Q,\text{D}_{k}Q\right\rangle +\left\langle i\eta,\text{D}_{j}\text{D}_{k}Q\right\rangle \right)\left(\dot{z}-iE\left[z\right]z\right)_{k}=-\left\langle N,\text{D}_{j}Q\right\rangle .
\]
Therefore
\[
\left|\dot{z}-iE\left[z\right]z\right|\lesssim\left|\left\langle N,\text{D}Q\left[z\right]\right\rangle \right|\lesssim\left\Vert N\right\Vert _{L^{1}+L^{\infty}}.
\]
This in particular implies estimates on parameters in Theorem \ref{thm:AC} as the following:
\begin{cor}\label{cor:modulation}
Given the conditions in Theorem \ref{thm:AC}, then the solution to \eqref{eq:eqthm1} can be uniquely decomposed as
\begin{equation}
u=Q\left(x,z\left(t\right)\right)+\eta\left(t\right)\label{eq:decompcor}
\end{equation}
with differentiable $z\left(t\right)\in\mathbb{C}$ such that
\[
\left\langle i\eta(t),\mathrm{D}_{1}Q\left[z(t)\right]\right\rangle =\left\langle i\eta(t),\mathrm{D}_{2}Q\left[z\right(t)]\right\rangle =0.
\]
The soliton parameter $z(t)$ satisfies
\begin{equation}
\left|\dot{z}-iE\left[z\right]z\right|\left(t\right)\lesssim \left|\left\langle \overline{Q\left[z\right]}\eta^{2}+2Q\left[z\right]\left|\eta\right|^{2}+\left|\eta\right|^{2}\eta,\text{D}Q\left[z\right]\right\rangle \right|.\label{eq:upper2cor}
\end{equation}
\end{cor}
With the computations above, we have the following result concerning the behavior of the phase $E[z(t)]$.
\begin{cor}\label{cor:decayphase}
Under the decay assumption that  $\left\Vert \jx^{-2}\eta(t)\right\Vert_{L^\infty_x}\lesssim t^{-1+\alpha}$, the phase $E[z(t)]$ has a limit $E[z(\infty)]$ and it satisfies the following estimates:
\[
\left|E\left[z(\infty)\right]-E\left[z(t)\right]\right|\lesssim t^{-1+2\alpha}.
\]
\end{cor}
\begin{proof}
From the decay assumption, the modulation equation \eqref{eq:upper2cor}, one has
\[
\left|\dot{z}\left(t\right)-iE[z\left(t\right)]z(t)\right|=\frac{d}{dt}\left|z\left(t\right)e^{-i\int_{0}^{s}E\left[z(\sigma)\right]\,d\sigma}\right|=\frac{d}{dt}\left|z(t)\right|\lesssim\epsilon^{2}t^{-2+2\alpha}.
\]
Therefore $|z(t)|$ has a limit $|z(\infty)|$ such that\[
\left||z(\infty)|-|z(t)|\right|\lesssim \epsilon^2 t^{-1+2\alpha}.\]
Recall that $E[z(t)]$ smoothly depends on $|z(t)|$ only. The desired result follows from the estimate above.
\end{proof}

\section{Reduction and the basic set-up}\label{sec:reduction}

In this section, we prepare the analysis the weighted estimates for $\eta$. First
of all, we reduce the problem to a model equation. 

Consider the small norm solution with the decomposition
\[
u=Q\left(x,z\left(t\right)\right)+\eta\left(t\right).
\]
Plunging the decomposition above into the NLS \eqref{eq:NLS}, then
the equation for $\eta$ is given by
\begin{align*}
i\partial_{t}\eta-\partial_{xx}\eta+V\left(x\right)\eta & =2\left|Q\left[z\right]\right|^{2}\eta+\left(Q\left[z\right]\right)^{2}\bar{\eta}\\
 & +E\left[z\right]Q\left[z\right]-i\text{D}Q\left[z\right]\dot{z}\\
 & +\overline{Q\left[z\right]}\eta^{2}+2Q\left[z\right]\left|\eta\right|^{2}+\left|\eta\right|^{2}\eta.
\end{align*}
By the fact $Q\left[z\right]=iz\text{D}Q$, it follows that
\[
E\left[z\right]Q\left[z\right]-i\text{D}Q\left[z\right]\dot{z}=i\text{D}Q\left(\dot{z}-izE\left[z\right]\right).
\]
Note that from the modulation equation, we know
\begin{equation}
\left|\dot{z}\left(t\right)-iz\left(t\right)E\left[z\left(t\right)\right]\right|\lesssim\left|\left\langle \phi,\eta^{2}\right\rangle \right|.\label{eq:moduz}
\end{equation}
In order to make use of the modulation equation when we perform the integration by parts in time later on, we rewrite
\begin{equation}
\mathcal{Q}\left[z\right]=e^{-i\int_{0}^{t}E\left[z(s)\right]\,ds}Q\left[z\right]\label{eq:modQ}
\end{equation}
 and the equation for $\eta$ becomes
\begin{align}
i\partial_{t}\eta-\partial_{xx}\eta+V\eta & =2\left|\mathcal{Q}\left[z\right]\right|^{2}\eta+\left(\mathcal{Q}\left[z\right]\right)^{2}e^{2i\int_{0}^{t}E\left[z(\sigma)\right]\,d\sigma}\bar{\eta}\label{eq:eta}\\
 & +\overline{Q\left[z\right]}\eta^{2}+2Q\left[z\right]\left|\eta\right|^{2}\nonumber \\
 & +\left|\eta\right|^{2}\eta\nonumber \\
 & +i\text{D}Q\left(\dot{z}-izE\left[z\right]\right)\nonumber \\
 & =:N_{1,1}+N_{1,2}+N_{2}+N_{3}+M=:F.\label{eq:inhomo}
\end{align}
The key difference is that
\begin{equation}
\partial_{t}\mathcal{Q}\left[z\left(t\right)\right]=e^{-i\int_{0}^{t}E\left[z(s)\right]\,ds}\text{D}Q\left[z\right]\left(\dot{z}\left(t\right)-iz\left(t\right)E\left[z\left(t\right)\right]\right)\label{eq:mathcalQ}
\end{equation}
meanwhile
\[
\partial_{t}Q\left[z(t)\right]=\dot{z}\left(t\right)\text{D}Q.
\]
Then in the expression for \eqref{eq:mathcalQ}, the time differentiation
results in the expression governed by the modulation equation.

By Duhamel's formula, solving the equation for $\eta$ for $t=1$, one
has
\begin{equation}
\eta\left(t,x\right)=e^{iH\left(t-1\right)}\eta_{1}+\int_{1}^{t}e^{iH\left(t-s\right)}\left(F(s)\right)\,ds.\label{eq:duhameta}
\end{equation}
Define the profile for $\eta$ as
\begin{equation}
f\left(t\right):=e^{-iHt}P_c\eta\left(t\right).\label{eq:profief}
\end{equation}
Then in terms of the distorted Fourier transform, the equation for
the profile is given by
\begin{equation}
\tilde{f}\left(t,k\right)=\tilde{f}\left(1,k\right)+\int_{1}^{t}e^{-ik^{2}s}\tilde{F}\left(s\right)\,ds.\label{eq:duhamelf}
\end{equation}
We will estimate the solution in the following bootstrap space
\begin{equation}
X_{T}:=\left\{ \eta|\left\Vert \eta\right\Vert _{L_{t}^{\infty}\left(\left[0,T\right];H^{1}\right)}+\left\Vert \tilde{f}\right\Vert _{L_{t}^{\infty}\left(\left[0,T\right];L_{k}^{\infty}\right)}+\left\Vert t^{-\alpha}\tilde{f}\right\Vert _{L_{t}^{\infty}\left(\left[0,T\right];H_{k}^{1}\right)}\right\} .\label{eq:bootstrap1}
\end{equation}
The energy estimate estimate will be standard so in this paper we
focus on the weighted estimate and the pointwise bound for the profile.

We begin with some observation and reduction.
\subsection{Bound states}\label{subsec:bound}
By the modulation equation, $\eta\left(t\right)$ is orthogonal to
the kernel of $\mathrm{H}\left[z(t)\right]$ associated with the linearizaiton
around $Q\left(x,z\left(t\right)\right)$. By the difference of the
continuous spectrum of $\mathcal{H}\left(t\right)$ and $H=-\partial_{xx}+V$,
see Lemma \ref{lem:Diff}, it follows that for any space $Y$ such
that $H^{2}\bigcap W^{1,1}\subset Y\subset H^{-2}+L^{\infty}$, one
has
\begin{equation}
P_{c}\eta(t)\sim\eta(t)=\mathcal{K}(z(t))P_c \eta(t),\label{eq:etaXt}
\end{equation}
where the operator $\mathcal{K}(z(t))$ is from Lemma \ref{lem:Diff} which is bounded in the space which we perform the bootstrap argument. Hence $\eta$  enjoys the same decay estimates as $P_c \eta$. We summarize them as the following corollary.
\begin{cor}\label{cor:eta}
Given the bootstrap assumption \eqref{eq:bootstrap1}, we have the following:
pointwise decay
\begin{equation}
    \left\Vert\eta(t)\right\Vert\lesssim \epsilon t^{-1/2},
\end{equation}
improved local decay
\begin{equation}
    \left\Vert\jx^{-2}(\eta(t))\right\Vert\lesssim \epsilon t^{-1+\alpha}
\end{equation}
and local $L^2$ decay
\begin{equation}
    \left\Vert\jx^{-1}(\partial_x \eta(t))\right\Vert\lesssim \epsilon t^{-1+\alpha}.
\end{equation}
Moreover, from the construction of $K(z(t))$, one has that \begin{equation}
    \left\Vert\jx^m(\eta(t)-P_c\eta(t))\right\Vert\lesssim \epsilon^2 t^{-1+\alpha}
\end{equation}for any fixed integer $m$ due to the exponential decay of the eigenfuction $\phi$.

\end{cor}
\begin{proof}
These are direct consequences of the bootstrap assumption \eqref{eq:bootstrap1} and Lemmata \ref{lem:pointwiseH},  \ref{lemlocdecinfty}, \ref{lem:localEn} and \ref{lem:Diff}.
\end{proof}  
\subsection{Modulation terms}
First of all, from the localized improved decay estimate, from \eqref{eq:moduz}
the bootstrap assumption \eqref{eq:bootstrap1} and the improved local decay \eqref{locdecinfty0}, it follows
\[
\left|\dot{z}\left(t\right)-iz\left(t\right)E\left[z\left(t\right)\right]\right|\lesssim\epsilon^{2}t^{-2+2\alpha}.
\]
Given the strong localization by $\text{D}Q$, the term $M(t,x)=i\text{D}Q\left(\dot{z}-izE\left[z\right]\right)$
coming from the modulation parameters in the equation for $\eta$, \eqref{eq:eta} and \eqref{eq:inhomo}, 
can be treated by the smoothing estimates in the analysis of weighted
estimate since $t^{-1+2\alpha}$ is $L_{t}^{2}$ integrable. Actually
this term will resemble the quadratic terms in the equation for $\eta$
due to the modulation equation. We will give  a more detailed
remark on this in the quadratic term analysis section.





\subsection{Reduction to model equations}

From the analysis above, in the equation for $\eta$, \eqref{eq:eta},
the term $M=i\text{D}Q\left(\dot{z}-izE\left[z\right]\right)$ given
by the modulation parameter behaves like the quadratic term with a
localized coefficient.  We also note that
if we perform integration by parts in $s$, as we computed in \eqref{eq:mathcalQ},
when the time derivative hits the coefficients of those first order
perturbations, $\left|\mathcal{Q}\left[z\right]\right|^{2}$ and $\left(\mathcal{Q}\left[z\right]\right)^{2}$,  it will result in localized terms with fast decay rates.

Based on these observations, and what we will see in all estimates
below, the structures of the quadratic terms will not be important,
it suffices to consider the following model problem
\begin{equation}\label{eq:modelu}
i\partial_{t}u-\partial_{xx}u+Vu=a_{1}\left(x\right)u+a_{2}(x)e^{2i\int_{0}^{t}E\left[z(\sigma)\right]\,d\sigma}\bar{u}+b\left(x\right)u^{2}+\left|u\right|^{2}u
\end{equation}
with $a_1(x)$, $a_2(x)$ and $b(x)$ being smooth functions which decay exponentially such that $|a_j(x)|\lesssim\epsilon ^2 $ and $|b(x)|\lesssim \epsilon $ under the assumption that
\[
\left|\frac{d}{dt}E\left[z\left(t\right)\right]\right|\lesssim\epsilon^{2}t^{-2+2\alpha}
\]
as in the full problem given by the modulation equation.


Due to the comparison of the continuous spectrum as we analyzed the bound state above,  we can restrict the analysis of $\eta$
or the $u$ in the model problem \eqref{eq:modelu} above completely
onto the continuous spectrum. Therefore in the model problem, we can
furthermore assume that $-\partial_{xx}+V$ has no bound states. 

Note that the analysis first order perturbations in the full problem will be more involved, see \S \ref{subsec:apptofull}. But after a suitable refined decomposition, the analysis will recast the argument for the model problem.

For this model problem, we have the following result:
\begin{thm}\label{thm:mainu}
Consider the nonlinear Schr\"odinger equation with a potential \eqref{eq:modelu} under the additional assumption that 
$V$   has no eigenvalues.
Then we have the following:

There exists $0<\epsilon_{0}\ll 1$ 
such that for all $\epsilon\leq\epsilon_{0}$ and 
\begin{align}\label{datasmall1}
\left\Vert u_{0}\right\Vert _{H^{1,1}}=
  \left\Vert u_{0}\right\Vert _{H^{1}}+\left\Vert xu_{0}\right\Vert _{L^{2}}=\epsilon
\end{align}
the equation \eqref{eq:modelu} has a unique
global solution $u\in C(\R,H^1(\R))$, with $u(0,x)=u_{0}(x)$, and satisfying the sharp decay rate
\begin{align}\label{main1fdecay1}
\left\Vert u(t)\right\Vert_{L^\infty_x}
  \lesssim \frac{\epsilon}{\left(1+\left|t\right|\right)^{\frac{1}{2}}}. 
\end{align}
Moreover, if we define the profile of the solution $u$ as
\begin{align}
\label{main1prof1}
f\left(t,x\right):=e^{-it\left(-\partial_{xx}+V\right)}u\left(t,x\right),
  \qquad \tilde{f}\left(t,k\right):=e^{-itk^{2}}\tilde{u}\left(t,k\right),
\end{align}
then one has 
\begin{align}\label{main1fbounds1}
{\big\| \tilde{f}(t) \big\|}_{L_k^\infty}
  + (1+|t|)^{-\alpha} {\| \partial_{k}\tilde{f}(t) \|}_{L_k^2} \lesssim\epsilon
\end{align}
for some $\alpha=\alpha(\gamma)>0$ small enough.

Finally, we have the following asymptotics:
there exists $W_{+\infty}\in L^{\infty}$ such that
\begin{align}\label{mainasy1}
\left|\tilde{f}\left(t,k\right)\exp\left(\frac{i}{2}\int_{0}^{t}\left|\tilde{f}\left(s,k\right)\right|^{2}
  \frac{ds}{s+1}\right)-W_{+\infty}(k)\right| \lesssim \epsilon \, t^{-\beta}
\end{align}
for some $\beta\in(0,\alpha)$ as $t\rightarrow\infty$.
\end{thm}

As in Theorem \ref{thm:AC}, one can also derive the following asymptotic formula for $u$ in physical space:
\begin{align}\label{nonlinearasyu}
u\left(t,x\right)=\frac{e^{i\frac{x^{2}}{4t}}}{\sqrt{-2it}}
  \exp\left(-\frac{i}{2}\left|W_{+\infty}\left(-\frac{x}{2t}\right)\right|^{2}\log t\right)
  W_{+\infty}\left(-\frac{x}{2t}\right)+\mathcal{O}\left(t^{-\frac{1}{2}-\alpha}\right),\ \ t\gg1.
\end{align}
The direct application of the model problem is the following:
\begin{proof}[Proof of Theorem \ref{thm:AC}]
The theorem above together with Corollary \ref{cor:modulation}   imply results  in Theorem \ref{thm:AC}.
\end{proof}
\section{Weighted estimates}\label{sec:weight}

In this section, we perform the bootstrap analysis for the weighted
estimates for the model problem \eqref{eq:modelu}
\begin{equation}\label{eq:modelu5}
    i\partial_{t}u-\partial_{xx}u+Vu=a_{1}\left(x\right)u+a_{2}(x)e^{2i\int_{0}^{t}E\left[z(\sigma)\right]\,d\sigma}\bar{u}+b\left(x\right)u^{2}+\left|u\right|^{2}u
\end{equation}
where $-\partial_{xx}+V$ is generic without any bound states. The coefficients $a_j(x)$ and $b(x)$ are assumed to be smooth functions that decay exponentially such that $|a(x)|\lesssim\epsilon ^2 $ and $|b(x)|\lesssim \epsilon $.
We
again denote the profile for $u$ as
\[
f:=e^{-iHt}u.
\]
Then we separate the first-order perturbations, quadratic, and cubic
terms and deal with them by different methods. In terms of profiles,
one can write
\begin{equation}
\tilde{f}\left(t,k\right)=\tilde{f}\left(1,k\right)+\int_{1}^{t}e^{-ik^{2}s}\left(\tilde{N}_{1}\left(s\right)+\tilde{N}_{2}\left(s\right)+\tilde{N}_{3}\left(s\right)\right)\,ds\label{eq:duhamelf-1}
\end{equation}
where $N_{1}=a_{1}\left(x\right)u+a_{2}(x)e^{2i\int_{0}^{t}E\left[z(\sigma)\right]\,d\sigma}\bar{u}$, $N_{2}=b\left(x\right)u^{2}$ and $N_{3}=\left|u\right|^{2}u$.

We will bootstrap estimates for $u$ and its profile $f$ in the following
space:
\begin{equation}
X_{T}:=\left\{ u|\left\Vert u\right\Vert _{L_{t}^{\infty}\left(\left[0,T\right];H^{1}\right)}+\left\Vert \tilde{f}\right\Vert _{L_{t}^{\infty}\left(\left[0,T\right];L_{k}^{\infty}\right)}+\left\Vert t^{-\alpha}\tilde{f}\right\Vert _{L_{t}^{\infty}\left(\left[0,T\right];H_{k}^{1}\right)}\right\} \label{eq:bootstrap1-1}
\end{equation}
for some sufficiently small $\alpha>0$.
Applying the localized pointwise decay, Lemma \ref{lemlocdecinfty}, and the improved local $L^2$ decay, Lemma \ref{lem:localEn}, we have the following estimates for the solution $u$ in the bootstrap space above:
\begin{cor}\label{cor:directXT}
Suppose $u\in X_T$, then we have
\begin{align}
\left\Vert \jx^{-2}u \right\Vert_{L^{\infty}_x}+{\big\| \jx^{-1} \partial_x u  \big\|}_{L^2_x}
  \lesssim |t|^{-1+\alpha}{\big\| u \big\|}_{X_T}.
\end{align}
\end{cor}

The following proposition closes the bootstrap argument for the weighted estimates.
\begin{prop}\label{pro:weightmain1}
For $1\leq t\leq T$, one has that, for some $C>0$,
\begin{align}\label{weightmainconc1}
{\big\| \partial_{k}\tilde{f}(t) \big\|}_{L_{k}^{2}}
  \leq {\big\| \partial_{k}\tilde{f}(1) \big\|}_{L_k^2}
  + C t^{\alpha} {\| u \|}_{X_{T}}^{3}.
\end{align}
\end{prop}

\begin{proof}
This is a combination of Proposition \ref{pro:weightmainfirst}, Proposition \ref{pro:weightmain2} and Proposition \ref{pro:weightmain}.
\end{proof}
\subsection{The cubic term}

The analysis of the weighted estimates for the cubic term is now well-studied.
We refer to Chen-Pusateri \cite{CP} for the detailed analysis.
The poof is based on three main ingredients:
the Fourier transform adapted to the Schr\"odinger operator $H=-\partial_{xx}+V$,
local decay  estimates, and $L^2$-bounds on pseudo-differential operators
whose symbols are given by the Jost functions for $H$. We sketch some basic logics from Chen-Pusateri \cite{CP} here.

Consider the inhomogeneous term from \eqref{eq:duhamelf-1} associated with the cubic term:
{\footnotesize\begin{align}\label{introD}
\begin{split}
\int_{1}^{t}e^{-ik^{2}s}\left(\tilde{N}_{3}\left(s\right)\right)\,ds&:=  i \, \mathcal{N}_\mu[f,f,f](t,k) 
\\ 
\mathcal{N}_\mu[f,f,f](t,k) & := \int_{1}^{t} \iiint e^{is(-k^2+\ell^2-m^2+n^2)}
  \tilde{f}(s,\ell)\overline{\tilde{f}(s,m)}\tilde{f}(s,n)\mu(k,\ell,m,n) \,dndmd\ell ds
\end{split}
\end{align}}
where we have defined
the {\it nonlinear spectral distribution}
\begin{align}\label{intromu}
\mu(k,\ell,m,n):=\int\overline{\mathcal{K}(x,k)}\mathcal{K}(x,\ell)\overline{\mathcal{K}(x,m)}\mathcal{K}(x,n)\,dx.
\end{align}
To obtain the desired bounds on $\mathcal{N}_\mu$ we need to understand the structure of $\mu$.
To do this we first decompose 
$\K = \K_S + \K_R$ where:
$\K_S$ is linear combination of exponentials $e^{\pm i xk}$ whose coefficients depend on the sign of $k$ and $x$,
and therefore resembles a (flat) plane wave;
$\K_R$ is the component arising from the interaction with the potential
and has strong localization in $x$ and is uniformly regular in $k$.


According to this basic decomposition,  in Chen-Pusateri \cite{CP}, 
we proposed a splitting of $\mu$ into two pieces: $\mu = \mu_S + \mu_R$,
where $\mu_S$ only contains the interaction of the four $\K_S$ functions having argument $x$ of the same sign,
and $\mu_R$ is all the rest. 
We call $\mu_S$ the ``singular' part of $\mu$
and $\mu_R$, the``regular'' part of $\mu$.
We then define $\mathcal{N}_S = \mathcal{N}_{\mu_S}$, respectively $\mathcal{N}_R := \mathcal{N}_{\mu_R}$, 
to be the singular, respectively, the regular, part of the nonlinear terms $\mathcal{N}_\mu$ in \eqref{introD}.

Two components $\mathcal{N}_S$ and $\mathcal{N}_R$ were analyzed separately by relying on two main observations:
a {\it commutation} property with $\partial_k $ for the singular part, 
and the {\it localization} property of the regular part.
More precisely, with a simple explicit calculation, we shown that the multilinear commutator
between $\partial_k$ and $\mathcal{N}_S$ satisfies
\begin{align}\label{introcomm}
[\partial_k,\mathcal{N}_S] = \mathcal{N}_S^{'}
\end{align}
where $\mathcal{N}_S^{'}$ 
is a localized term of the form $a(x) |u|^2 u$, for a Schwartz function $a$. 
This last term is then very easy to handle using the localized decay estimates which are also used to estimate $\mathcal{N}_R$ as well.

The regular part $\mathcal{N}_R$ can be thought of as 
the (flat) transform of a nonlinear term of the form  $\jx^{-\rho}|u|^2u$, 
for some $\rho >0$ related to the decay of $V$. 
More precisely, we can view it as (the transform of) a localized trilinear term
whose inputs are pseudo-differential operators applied to the solution $u$ that satisfy
$L^2$ and $L^\infty$ type estimates similar to those satisfied by $u$ itself.

Then we observed that  applying $\partial_k$ to $\mathcal{N}_R$ essentially amounts to multiplying it by 
a factor of $tk$.
Then, we are reduced to estimating the $L^2_x$ norm of an expression of the form 
\begin{align*}
\int_0^t s \, \langle \partial_x \rangle \jx^{-\rho} |u(s)|^2 u(s)\, ds.
\end{align*}
This can be estimated by  {\it local decay estimates} for 
$u=e^{itH}f$.
Importantly, we need to do this under the sole assumptions that we can control (up to some small growth in time)
$\partial_k \wt{f}$ in $L^2_k$, see Lemma \ref{lem:localEn} and Lemma \ref{lemlocdecinfty}.

The following result is the main result from Section 4 in Chen-Pusateri \cite{CP}.
\begin{prop}\label{pro:weightmain}
For $1\leq t\leq T$, one has that, for some $C>0$,
\begin{align}\label{weightmainconc}
{\big\| \partial_{k}\mathcal{N}_\mu[f,f,f](t,k) \big\|}_{L_{k}^{2}}
  \leq 
  C t^{\alpha} {\| u \|}_{X_{T}}^{3}.
\end{align}
\end{prop}
See Proposition 4.1 in Chen-Pusateri \cite{CP}.





\subsection{The quadratic term}

In this subsection, we analyze the quadratic inhomogeneous term
\begin{equation}
\int_{1}^{t}e^{-ik^{2}s}\tilde{N}_{2}\left(s\right)\,ds=\int_{1}^{t}e^{-isk^{2}}\int\overline{\mathcal{K}}\left(x,k\right)\left(b\left(x\right)u^{2}\right)\,dxds.\label{eq:quadraticdu}
\end{equation}
In this analysis, sometime it is more convenient to use the distorted
Fourier transform to go back to the physical space. 

Denote
\[
h=\mathcal{\tilde{F}}^{-1}\left[\int_{1}^{t}e^{-ik^{2}s}\tilde{N}_{2}\left(s\right)\,ds\right].
\]
The weighted estimate analysis is reduced analyze the $J$ operator
acting on the nonlinearity where
\begin{equation}
Jh:=\mathcal{\tilde{F}}^{-1}\left[\partial_{k}\mathcal{\tilde{F}}\left[h\right]\left(k\right)\right].\label{eq:Jop}
\end{equation}
One can write
\[
b\left(x\right)u^{2}=\left(\int\mathcal{K}\left(x,\ell\right)\tilde{b}\left(\ell\right)\,d\ell\right)\left(\int\mathcal{K}\left(x,n\right)\tilde{u}\left(n\right)\,dn\right)\left(\int\mathcal{K}\left(x,m\right)\tilde{u}\left(m\right)\,dm\right).
\]
Therefore,
\begin{align*}
\int_{1}^{t}e^{-isk^{2}}\iiiint\tilde{b}\left(\ell\right)\tilde{u}\left(n\right)\tilde{u}\left(m\right)\overline{\mathcal{K}}\left(x,k\right)\mathcal{K}\left(x,\ell\right)\mathcal{K}\left(x,n\right)\mathcal{K}\left(x,m\right)\,dxd\ell dmdn\\
=\int_{1}^{t}e^{-isk^{2}}\iiint\tilde{b}\left(\ell\right)\tilde{u}\left(n\right)\tilde{u}\left(m\right)\mu\left(k,\ell,n,m\right)\,d\ell dmdn.
\end{align*}
where
\[
\mu\left(k,\ell,n,m\right)=\int\overline{\mathcal{K}}\left(x,k\right)\mathcal{K}\left(x,\ell\right)\mathcal{K}\left(x,n\right)\mathcal{K}\left(x,m\right)\,dx
\]
as in the cubic term analysis.

By direct computations, we have
\begin{align*}
\partial_{k}\tilde{h}\left(t,k\right) & =-\int_{1}^{t}e^{-isk^{2}}2isk\iiint\tilde{b}\left(\ell\right)\tilde{u}\left(n\right)\tilde{u}\left(m\right)\mu\left(k,\ell,n,m\right)\,d\ell dmdn\\
 & +\int_{1}^{t}e^{-isk^{2}}\iiint\tilde{b}\left(\ell\right)\tilde{u}\left(n\right)\tilde{u}\left(m\right)\partial_{k}\mu\left(k,\ell,n,m\right)\,d\ell dmdn.
\end{align*}
In the physical space, one has
\begin{align*}
Jh & =\int_{0}^{t}e^{-isH}\mathcal{\tilde{F}}^{-1}\left[-2isk\iiint\tilde{b}\left(\ell\right)\tilde{u}\left(n\right)\tilde{u}\left(m\right)\mu\left(k,\ell,n,m\right)\,d\ell dmdn\right]\\
 & +\int_{0}^{t}e^{-isH}\mathcal{\tilde{F}}^{-1}\left[\iiint\tilde{b}\left(\ell\right)\tilde{u}\left(n\right)\tilde{u}\left(m\right)\partial_{k}\mu\left(k,\ell,n,m\right)\,d\ell dmdn\right].
\end{align*}
Applying the inhomogeneous smoothing estimate, \eqref{eq:smoothing2-2}, one can estimate
\begin{align}
\left\Vert Jh\right\Vert _{L^{2}} & \lesssim\left\Vert \left\langle x\right\rangle ^{\frac{5}{2}}\mathcal{\tilde{F}}^{-1}\left[2isk\iiint\tilde{b}\left(\ell\right)\tilde{u}\left(n\right)\tilde{u}\left(m\right)\mu\left(k,\ell,n,m\right)\,d\ell dmdn\right]\right\Vert _{L_{t}^{2}L_{x}^{2}}\label{eq:JHexp}\\
 & +\left\Vert \left\langle x\right\rangle ^{\frac{5}{2}}\mathcal{\tilde{F}}^{-1}\left[\iiint\tilde{b}\left(\ell\right)\tilde{u}\left(n\right)\tilde{u}\left(m\right)\partial_{k}\mu\left(k,\ell,n,m\right)\,d\ell dmdn\right]\right\Vert _{L_{t}^{2}L_{x}^{2}.}\nonumber 
\end{align}
First of all, we analyze the term
\begin{equation}
\left\Vert \mathcal{\tilde{F}}^{-1}\left[2isk\iiint\tilde{b}\left(\ell\right)\tilde{u}\left(n\right)\tilde{u}\left(m\right)\mu\left(k,\ell,n,m\right)\,d\ell dmdn\right]\right\Vert _{L^{2}}.\label{eq:RHSq1}
\end{equation}
Focusing on the measure part and integrating by parts, for $k\geq0$,
(the analysis for $k\leq0$ will be the same), explicitly, one has{\small
\begin{align}
k\mu\left(k,\ell,n,m\right) & =\int k\overline{\mathcal{K}}\left(x,k\right)\mathcal{K}\left(x,\ell\right)\mathcal{K}\left(x,n\right)\mathcal{K}\left(x,m\right)\,dx\nonumber \\
 & =\int ke^{-ikx}T\left(k\right)\overline{m}_{+}\left(x,k\right)\mathcal{K}\left(x,\ell\right)\mathcal{K}\left(x,n\right)\mathcal{K}\left(x,m\right)\,dx\nonumber \\
 & =-i\int e^{-ikx}T\left(k\right)\partial_{x}\overline{m}_{+}\left(x,k\right)\mathcal{K}\left(x,\ell\right)\mathcal{K}\left(x,n\right)\mathcal{K}\left(x,m\right)\,dx\label{eq:qm0}\\
 & -i\int\overline{\mathcal{K}}\left(x,k\right)\partial_{x}\mathcal{K}\left(x,\ell\right)\mathcal{K}\left(x,n\right)\mathcal{K}\left(x,m\right)\,dx\label{eq:qm1}\\
 & -i\int\overline{\mathcal{K}}\left(x,k\right)\mathcal{K}\left(x,\ell\right)\partial_{x}\mathcal{K}\left(x,n\right)\mathcal{K}\left(x,m\right)\,dx\label{eq:qm2}\\
 & -i\int\overline{\mathcal{K}}\left(x,k\right)\mathcal{K}\left(x,\ell\right)\mathcal{K}\left(x,n\right)\partial_{x}\mathcal{K}\left(x,m\right)\,dx.\label{eq:qm3}
\end{align}}Note that by explicit computations
\[
\left(\int\partial_{x}\mathcal{K}\left(x,\ell\right)\tilde{b}\left(\ell\right)\,d\ell\right)=\partial_{x}b\left(x\right),\,
\left(\int\partial_{x}\mathcal{K}\left(x,n\right)\tilde{u}\left(n\right)\,dn\right)=\partial_{x}u\left(t,x\right).
\]
Therefore it follows that{\small
\begin{align*}
\iiint\tilde{b}\left(\ell\right)\tilde{u}\left(n\right)\tilde{u}\left(m\right)\left(\int\overline{\mathcal{K}}\left(x,k\right)\partial_{x}\mathcal{K}\left(x,\ell\right)\mathcal{K}\left(x,n\right)\mathcal{K}\left(x,m\right)\,dx\right)\,d\ell dmdn\\
=\int\overline{\mathcal{K}}\left(x,k\right)\partial_{x}b\left(x\right)u^{2}\,dx\\
=\tilde{\mathcal{F}}\left[\partial_{x}b\left(x\right)u^{2}\right]
\end{align*}}
and similarly{\small
\begin{align*}
\iiint\tilde{b}\left(\ell\right)\tilde{u}\left(n\right)\tilde{u}\left(m\right)\left(\int\overline{\mathcal{K}}\left(x,k\right)\mathcal{K}\left(x,\ell\right)\partial_{x}\mathcal{K}\left(x,n\right)\mathcal{K}\left(x,m\right)\,dx\right)\,d\ell dmdn\\
=\int\overline{\mathcal{K}}\left(x,k\right)a\left(x\right)u\partial_{x}u\,dx\\
=\tilde{\mathcal{F}}\left[b\left(x\right)u\partial_{x}u\right].
\end{align*}}Therefore, to bound \eqref{eq:RHSq1}, we first estimate the terms given
by \eqref{eq:qm1}, \eqref{eq:qm2} and \eqref{eq:qm3}. First of all,  explicitly,
we get
\begin{align*}
\left\Vert \left\langle x\right\rangle ^{\frac{5}{2}}\tilde{\mathcal{F}}^{-1}\left[2is\tilde{\mathcal{F}}\left[\partial_{x}b\left(x\right)u^{2}\right]\right]\right\Vert _{L^{2}} & \lesssim\left\Vert \left\langle x\right\rangle ^{\frac{5}{2}}s\partial_{x}b\left(x\right)u^{2}\right\Vert _{L^{2}}
 \lesssim\epsilon\left\Vert s\left\langle x\right\rangle ^{-4}u\right\Vert _{L_{x}^{\infty}}^{2}
  \lesssim\epsilon s^{-1+2\alpha}\left\Vert u\right\Vert _{X_{T}}^{2}
\end{align*} where in the last line, we applied Corollary \ref{cor:directXT}.

Similarly, one has
\begin{align*}
\left\Vert \left\langle x\right\rangle ^{\frac{5}{2}}\tilde{\mathcal{F}}^{-1}\left[2is\tilde{\mathcal{F}}\left[b\left(x\right)u\partial_{x}u\right]\right]\right\Vert _{L^{2}} & \lesssim\left\Vert \left\langle x\right\rangle ^{\frac{5}{2}}b\left(x\right)u\partial_{x}u\right\Vert _{L^{2}}\\
 & \lesssim\epsilon\left\Vert \left\langle x\right\rangle ^{-3}s\partial_{x}u\right\Vert _{L_{x}^{2}}\left\Vert \left\langle x\right\rangle ^{-2}u\right\Vert _{L_{x}^{\infty}}\\
 & \lesssim\epsilon s^{-1+2\alpha}\left\Vert u\right\Vert _{X_{T}}^{2}.
\end{align*}
It remains to bound the term given by \eqref{eq:qm0}. We need to study
the boundedness given by
\[
\int e^{-ikx}\overline{T}\left(k\right)\partial_{x}\overline{m}_{+}\left(x,k\right)\mathcal{K}\left(x,\ell\right)\mathcal{K}\left(x,n\right)\mathcal{K}\left(x,m\right)\,dx.
\]
Taking $\tilde{b}\left(\ell\right)\tilde{u}\left(n\right)\tilde{u}\left(m\right)$
into account, we consider{\small
\begin{align*}
\int e^{-ixk}\overline{T}\left(k\right)\partial_{x}\overline{m}_{+}\left(x,k\right)\left(\int\mathcal{K}\left(x,\ell\right)\tilde{b}\left(\ell\right)\,d\ell\right)\left(\int\mathcal{K}\left(x,n\right)\tilde{u}\left(n\right)\,dn\right)\left(\int\mathcal{K}\left(x,m\right)\tilde{u}\left(m\right)\,dm\right)\,dx\\
=\int e^{ixk}\overline{T}\left(k\right)\partial_{x}\overline{m}_{+}\left(x,k\right)b\left(x\right)u\left(x\right)u\left(x\right)\,dx.
\end{align*}}This term has a nice $L^{2}$ bound using pseudo-differential operators
by Lemma \ref{lem:pesudo1}. Therefore without weights, we get
\begin{align*}
\left\Vert \int e^{ixk}\overline{T}\left(k\right)\partial_{x}\overline{m}_{+}\left(x,k\right)b\left(x\right)u\left(x\right)u\left(x\right)\,dx\right\Vert _{L^{2}} & \lesssim\left\Vert \left\langle x\right\rangle ^{2}b\left(x\right)u^{2}\right\Vert _{L^{2}} \lesssim\epsilon s^{-2+2\alpha}\left\Vert u\right\Vert _{X_{T}}^{2}
\end{align*}where again in the last line, we applied the decay from Corollary \ref{cor:directXT}.

Next we consider the weighted version. We only need to consider the homogeneous
weight since the inhomogeneous one can be obtained by the standard
interpolation.

Consider
\begin{equation}
\left\Vert x^{\frac{5}{2}}\tilde{\mathcal{F}}^{-1}\left[s\int e^{iyk}\partial_{x}\overline{m}_{+}\left(y,k\right)b\left(y\right)u\left(y\right)u\left(y\right)\,dy\right]\left(x\right)\right\Vert _{L_{x}^{2}}.\label{eq:weightedWO}
\end{equation}
Taking the flat Fourier transform, by the Fourier duality, it remains to bound{\small
\begin{align*}
\left\Vert \partial_{\ell}^{\frac{5}{2}}\hat{\mathcal{F}}\tilde{\mathcal{F}}^{-1}\left[s\int e^{iyk}\partial_{x}\overline{m}_{+}\left(y,k\right)b\left(y\right)u\left(y\right)u\left(y\right)\,dy\right]\left(\ell\right)\right\Vert _{L_{\ell}^{2}}\\
=\left\Vert \partial_{\ell}^{\frac{5}{2}}W_{+}^{\ast}\left[\int se^{iyk}\partial_{x}\overline{m}_{+}\left(y,k\right)b\left(y\right)u\left(y\right)u\left(y\right)\,dy\right]\left(\ell\right)\right\Vert _{L_{\ell}^{2}}\\
\lesssim\left\Vert \partial_{k}^{\frac{5}{2}}\int e^{iyk}\partial_{x}\overline{m}_{+}\left(y,k\right)sb\left(y\right)u\left(y\right)u\left(y\right)\,dy\right\Vert _{L_{k}^{2}}\\
\lesssim\left\Vert \int e^{iyk}\left|y\right|^{\frac{5}{2}}\partial_{x}\overline{m}_{+}\left(y,k\right)sb\left(y\right)u\left(y\right)u\left(y\right)\,dy\right\Vert _{L_{k}^{2}}\\
+\left\Vert \int e^{iyk}\partial_{k}^{\frac{5}{2}}\partial_{x}\overline{m}_{+}\left(y,k\right)sb\left(y\right)u\left(y\right)u\left(y\right)\,dy\right\Vert _{L_{k}^{2}}\\
\lesssim\left\Vert s\left|b\right|^{\frac{1}{2}}\left(y\right)u\left(y\right)u\left(y\right)\right\Vert _{L_{y}^{2}}
\lesssim s\epsilon\left\Vert \left\langle x\right\rangle ^{-2}u\right\Vert _{L^{\infty}}^{2}
\lesssim s^{-1+2\alpha}\epsilon\left\Vert u\right\Vert _{X}^{2}
\end{align*}}where in the third line, we applied the boundedness of the wave operator
$W_{+}^{\ast}=\hat{\mathcal{F}}\tilde{\mathcal{F}}^{-1}$ in $W^{k,p}$, see
Weder \cite{Wed}. 
\begin{rem}
It is possible that the bounds above can be obtained directly via integration
by parts instead of using the boundedness of wave operators.
\end{rem}

We still need to estimate the second term on the RHS of \eqref{eq:JHexp},
the case that the differentiation hits the measure
\begin{equation}
\left\Vert \left\langle x\right\rangle ^{\frac{5}{2}}\tilde{\mathcal{F}}^{-1}\left[\iiint\tilde{b}\left(\ell\right)\tilde{u}\left(n\right)\tilde{u}\left(m\right)\partial_{k}\mu\left(k,\ell,n,m\right)\,d\ell dmdn\right]\right\Vert _{L_{t}^{2}L_{x}^{2}.}.\label{eq:inhomq2}
\end{equation}
The same as before, we first consider the estimates without weights:
\[
\left\Vert \tilde{\mathcal{F}}^{-1}\left[\iiint\tilde{b}\left(\ell\right)\tilde{u}\left(n\right)\tilde{u}\left(m\right)\partial_{k}\mu\left(k,\ell,n,m\right)\,d\ell dmdn\right]\right\Vert _{L^{2}}.
\]
Applying Plancherel's theorem, we estimate the $L^{2}$ norm of
\[
\iiint\tilde{b}\left(\ell\right)\tilde{u}\left(n\right)\tilde{u}\left(m\right)\partial_{k}\mu\left(k,\ell,n,m\right)\,d\ell dmdn.
\]
Again, we focus on the case that for $k\geq0$ and perform integrate
by parts{\small
\begin{align}
\partial_{k}\mu\left(k,\ell,n,m\right) & =\int\partial_{k}\overline{\mathcal{K}}\left(x,k\right)\mathcal{K}\left(x,\ell\right)\mathcal{K}\left(x,n\right)\mathcal{K}\left(x,m\right)\,dx\nonumber \\
 & =\int-ixe^{-ikx}T\left(k\right)\overline{m}_{+}\left(x,k\right)\mathcal{K}\left(x,\ell\right)\mathcal{K}\left(x,n\right)\mathcal{K}\left(x,m\right)\,dx\nonumber \\
 & +\int e^{-ikx}\partial_{k}T\left(k\right)\overline{m}_{+}\left(x,k\right)\mathcal{K}\left(x,\ell\right)\mathcal{K}\left(x,n\right)\mathcal{K}\left(x,m\right)\,dx\nonumber \\
 & +\int e^{-ikx}T\left(k\right)\partial_{k}\overline{m}_{+}\left(x,k\right)\mathcal{K}\left(x,\ell\right)\mathcal{K}\left(x,n\right)\mathcal{K}\left(x,m\right)\,dx.\label{eq:partialkmu}
\end{align}}
Note that again by explicit computations, one has
\begin{align*}
&\iiint\int\tilde{b}\left(\ell\right)\tilde{u}\left(n\right)\tilde{u}\left(m\right)ixe^{-ikx}T\left(k\right)\overline{m}_{+}\left(x,k\right)\mathcal{K}\left(x,\ell\right)\mathcal{K}\left(x,n\right)\mathcal{K}\left(x,m\right)\,dxd\ell dmdn\\
&=\int\overline{\mathcal{K}}\left(x,k\right)ixb\left(x\right)u^{2}\,dx
\end{align*}
and
\begin{align*}
\left\Vert \int\overline{\mathcal{K}}\left(x,k\right)ixb\left(x\right)u^{2}ixe^{-ikx}\,dx\right\Vert _{L^{2}} & \lesssim\epsilon\left\Vert \left\langle x\right\rangle ^{-2}u\right\Vert _{L^{\infty}}^{2} \lesssim \epsilon s^{-2+2\alpha}\left\Vert u\right\Vert _{X_{T}}^{2}
\end{align*}
where again, we used Corollary \ref{cor:directXT}.

Similarly, we also have
\begin{align*}
\iiiint\tilde{b}\left(\ell\right)\tilde{u}\left(n\right)\tilde{u}\left(m\right)e^{-ikx}\partial_{k}T\left(k\right)\overline{m}_{+}\left(x,k\right)\mathcal{K}\left(x,\ell\right)\mathcal{K}\left(x,n\right)\mathcal{K}\left(x,m\right)\,dxd\ell dmdn\\
=\int e^{-ikx}\partial_{k}T\left(k\right)\overline{m}_{+}\left(x,k\right)b\left(x\right)u^{2}\,dx,
\end{align*}
whence it follows
\begin{align*}
\left\Vert \int e^{-ikx}\partial_{k}T\left(k\right)\overline{m}_{+}\left(x,k\right)b\left(x\right)u^{2}\,dx\right\Vert _{L^{2}} & \lesssim\epsilon\left\Vert \left\langle x\right\rangle ^{-2}u\right\Vert _{L^{\infty}}^{2} \lesssim\epsilon s^{-2+2\alpha}\left\Vert u\right\Vert _{X_{T}}^{2}.
\end{align*}
For the last piece from \eqref{eq:partialkmu}, one has
\begin{align*}
\iiiint\tilde{b}\left(\ell\right)\tilde{u}\left(n\right)\tilde{u}\left(m\right)e^{-ikx}T\left(k\right)\partial_{k}\overline{m}_{+}\left(x,k\right)\mathcal{K}\left(x,\ell\right)\mathcal{K}\left(x,n\right)\mathcal{K}\left(x,m\right)\,dxd\ell dmdn\\
=\int e^{-ikx}T\left(k\right)\partial_{k}\overline{m}_{+}\left(x,k\right)b\left(x\right)u^{2}\,dx.
\end{align*}
Then applying the boundedness of the pseudo-differential operator, Lemma \ref{lem:pesudo1}, we
conclude that
\begin{align*}
\left\Vert \int e^{-ikx}T\left(k\right)\partial_{k}\overline{m}_{+}\left(x,k\right)b\left(x\right)u^{2}ixe^{-ikx}\,dx\right\Vert _{L^{2}}\\
\lesssim\left\Vert \left\langle x\right\rangle ^{2}b\left(x\right)u^{2}\right\Vert _{L^{2}}
\lesssim\epsilon s^{-2+2\alpha}\left\Vert u\right\Vert _{X_{T}}^{2}.
\end{align*}
For the weighted estimates, we again take the flat Fourier transform
and use the boundedness of the wave operator as \eqref{eq:weightedWO}. It follows that
\[
\left\Vert \left\langle x\right\rangle ^{\frac{5}{2}}\tilde{\mathcal{F}}^{-1}\left[\iiint\tilde{b}\left(\ell\right)\tilde{u}\left(n\right)\tilde{u}\left(m\right)\partial_{k}\mu\left(k,\ell,n,m\right)\,d\ell dmdn\right]\right\Vert _{L_{x}^{2}.}\lesssim\epsilon s^{-2+2\alpha}\left\Vert u\right\Vert _{X_{T}}^{2}.
\]
Performing the $L_{t}^{2}$ integration and summing up the above pieces,
one has
\begin{align}
\left\Vert Jh(t)\right\Vert _{L^{2}} & \lesssim\epsilon\left(\int_{1}^{t}\left|s^{-1+2\alpha}\left\Vert u\right\Vert _{X_{T}}^{2}+s^{-2+2\alpha}\left\Vert u\right\Vert _{X_{T}}^{2}\right|^{2}\,ds\right)^{\frac{1}{2}}\nonumber \\
 & \lesssim\epsilon^{3}\left(\left\langle t\right\rangle ^{-\frac{1}{2}+2\alpha}+1\right)\label{eq:quadbdF}
\end{align}
which is actually globally bounded. Therefore, for the quadratic term,
we recover the bootstrap condition for the weighted norm.
\begin{prop}\label{pro:weightmain2}
For $1\leq t\leq T$, one has that, for some $C>0$,
\begin{align}\label{weightmainconc2}
{\big\| \partial_{k}\int_{1}^{t}e^{-ik^{2}s}\tilde{N}_{2}\left(s\right)\,ds \big\|}_{L_{k}^{2}}
  \leq 
  C \epsilon^3.
\end{align}
\end{prop}
Finally, we record a pseudo-differential operator bound  used above.
\begin{lem}
\label{lem:pesudo1}Consider the integral
\[
\int e^{ixk}T\left(k\right)\partial_{k}m_{+}\left(x,k\right)h\left(x\right)\,dx.
\]
Then one has
\[
\left\Vert \int e^{ixk}T\left(k\right)\partial_{k}m_{+}\left(x,k\right)h\left(x\right)\,dx\right\Vert _{L^{2}}\lesssim\left\Vert \left\langle x\right\rangle ^{2}h\right\Vert _{L^{2}}.
\]The same estimate holds for $\partial_k m_+$ replaced by $\partial_x m_+$. 
\end{lem}

\begin{proof}
We only prove the bound for $\partial_k m_+(x,k)$ since $\partial_x m_+(x,k)$ actually enjoys a better estimate, see Lemma \ref{lem:Mestimates}.

We write
\[
\int e^{ixk}T\left(k\right)\partial_{k}m_{+}\left(x,k\right)h\left(x\right)\,dx=\int e^{ixk}T\left(k\right)\frac{\partial_{k}m_{+}\left(x,k\right)}{\left\langle x\right\rangle ^{2}}\left(\left\langle x\right\rangle ^{2}h\left(x\right)\right)\,dx.
\]
Then one again looks at the the pseudo-differential operator 
\[
Op_{a}\left(\psi\right)\left(k\right)=\int e^{ikx}a\left(x,k\right)\psi\,dx,
\]
where
\[
a\left(x,k\right)=T\left(k\right)\frac{\partial_{k}m_{+}\left(x,k\right)}{\left\langle x\right\rangle ^{2}},\,\text{and}\,\psi\left(x\right)=\left\langle x\right\rangle ^{2}h\left(x\right).
\]
By estimates for Jost functions, Lemma \ref{lem:Mestimates}, we know that $\left\Vert \partial^j_k \partial^{\ell}_x a(x, k)\right\Vert _{L^{\infty}}<\infty, j,\ell=0,1.$ Then the $L^{2}$ bound follows from the standard bounded of pseudo-differential
operators, see Hwang \cite{Hwang}, Chen-Pusateri \cite{CP}, and hence
\[
\left\Vert \int e^{ixk}T\left(k\right)\partial_{k}m_{+}\left(x,k\right)h\left(x\right)\,dx\right\Vert _{L^{2}}\lesssim\left\Vert \left\langle x\right\rangle ^{2}h\right\Vert _{L^{2}}.
\]
We are done.
\end{proof}

\subsubsection{Modulation term in the equation for $\eta$.}

Recall that in the Duhamel expansion for the profile of $\eta$, the
modulation term produces an inhomogeneous term
\[
\int e^{-iHs}i\text{D}Q\left[z(s)\right]\left(\dot{z}(s)-iz(s)E\left[z(s)\right]\right)\,ds.
\]
By the same argument above, with the bootstrap assumption which implies
\[
\left|\dot{z}\left(t\right)-iz\left(t\right)E\left[z\left(t\right)\right]\right|\lesssim\epsilon^{2}t^{-2+2\alpha},
\]
we obtain{\small
\begin{align}
\left\Vert J
\int e^{-iHs}i\text{D}Q\left[z(s)\right]\left(\dot{z}(s)-iz(s)E\left[z(s)\right]\right)\,ds\right\Vert _{L^{2}} & \lesssim\epsilon^{3}\left(\int_{1}^{t}\left|s^{-1+2\alpha}+s^{-2+2\alpha}\right|^{2}\,ds\right)^{\frac{1}{2}}\nonumber \\
 & \lesssim\epsilon^{3}\left(\left\langle t\right\rangle ^{-\frac{1}{2}+2\alpha}+1\right)\label{eq:quadbdF-1}
\end{align}}
which also recasts the bootstrap condition.

.

\subsection{First order perturbation}

In this subsection, we analyze the estimate for
\[
\partial_{k}\int_{1}^{t}e^{-ik^{2}s}\left(\tilde{N}_{1}\left(s\right)\right)\,ds
\]
where $N_{1}=a_{1}\left(x\right)u+a_{2}(x)e^{2i\int_{0}^{t}E\left[z(\sigma)\right]\,d\sigma}\bar{u}$. To handle this term, it requires much more complicated
and refined arguments. Morally, a direct $k$ differentiation will
result in a growth like
\[
\int_{1}^{t}e^{-ik^{2}s}2sik\left(\tilde{N}_{1}\left(s\right)\right)\,ds.
\]
The decay given by the inhomogeneous term is  $a_{1}\left(x\right)u+a_{2}(x)e^{2i\int_{0}^{s}E\left[z(\sigma)\right]\,d\sigma}\bar{u}\sim s^{-1+\alpha}$.
Directly applying smoothing estimates or other dispersive estimates,
the integral above is far from being integrable. If we  integrate by parts in $s$, the RHS of the equation for $u$
will not give extra decay in $s$ unlike other nonlinear problems, for example, see Shatah \cite{Sh} and Germain-Pusateri \cite{GP}.
This is one crucial part in this paper. In the stability analysis
of other problems, for example, see  Krieger-Schlag \cite{KS}
and Schlag \cite{Sch1}, the coefficient of the first order perturbation has some
decays in time. In the current setting, the first order perturbation,
it will only introduce $\epsilon^{2}$ by the size of small solitons
but there is no decay at all.

To handle the current delicate setting, we introduce some refined
smoothing estimate to absorb the growth in time in the inhomogeneous
estimate. We also need to use the Fourier transform with respect to
$t$. We will need some auxiliary estimates.


\subsubsection{Auxiliary estimates and spaces}

Here we introduce some auxiliary estimates needed to handle first order perturbations. Let $\phi_{1}$ and $\phi_{2}$
be smooth nonegative functions such that $\phi_{1}+\phi_{2}=1$,
$\phi_{1}\left(\lambda\right)=1$ for $\left|\lambda\right|\leq1$
and $\phi_{1}\left(\lambda\right)=0$ for $\left|\lambda\right|\ge2$.

Letting $u$ be a solution to the model problem \eqref{eq:modelu5}, we decompose $u=u_{L}+u_{H}$ where
\begin{equation}
\mathcal{F}_{T}\left(u_{L}\right)\left(\tau\right)=\phi_{1}\left(\tau\right)\mathcal{F}_{T}\left(u\right)\left(\tau\right)\label{eq:LFt}
\end{equation}
\begin{equation}
\mathcal{F}_{T}\left(u_{H}\right)\left(\tau\right)=\phi_{2}\left(\tau\right)\mathcal{F}_{T}\left(u\right)\left(\tau\right)\label{eq:HFt}
\end{equation}
and
\[
\mathcal{F}_{T}\left(u\right)\left(\tau\right)=\frac{1}{\sqrt{2\pi}}\int e^{-it\tau}u\left(t\right)\,dt
\]
is the  Fourier transform with respect to time.

Now we are ready to introduce two additional
auxiliary estimates and bootstrap them together with norms given by
\eqref{eq:bootstrap1-1}.

For the low frequency part, we impose the bootstrap assumption that
\begin{equation}
\left\Vert \left\langle x\right\rangle ^{-2}t\left(-2iE\left[z(t)\right]\partial_{t}+\partial_{t}^{2}\right)\left(u_{L}\right)\right\Vert _{L_{x}^{\infty}L_{t}^{2}\left[0,T\right]}\lesssim\epsilon T^{\alpha}.\label{eq:boot3}
\end{equation}
The estimate above in particular implies that
\begin{equation}
\left\Vert \left\langle x\right\rangle ^{-2}t\left(-2iE\left[z\left(T\right)\right]\partial_{t}+\partial_{t}^{2}\right)\left(u_{L}\right)\right\Vert _{L_{x}^{\infty}L_{t}^{2}\left[0,T\right]}\lesssim\epsilon T^{\alpha}\label{eq:boot3add}
\end{equation}
and vice versa. Since due to the bootstrap assumption and the modulation equation,  one has $\left|\frac{d}{dt}E\left[z(t)\right]\right|\lesssim\epsilon^{2}t^{-1+2\alpha}$ and
\[
\left\Vert \left\langle x\right\rangle ^{-2}t\left|E\left[z(t)\right]-E\left[z(T)\right]\right|\partial_{t}\left(u_{L}\right)(t)\right\Vert _{L_{x}^{\infty}}\lesssim\epsilon^3 t^{-1+3\alpha}
\]
which is $L^{2}$ integrable in time. The fixed phase $E\left[z(T)\right]$
above can also be replaced by $E\left[z(\infty)\right]$.

For the high frequency part, we impose the bootstrap assumption that
\begin{equation}
\left\Vert \left\langle x\right\rangle ^{-2}\partial_{x}^{j}t\left(u_{H}\right)\right\Vert _{L_{x}^{\infty}L_{t}^{2}\left[0,T\right]}\lesssim\epsilon T^{\alpha}\,j=0,1.\label{eq:boot4}
\end{equation}
Finally, we also record a direct consequence of Corollary \ref{cor:directXT}
\begin{equation}
\left\Vert \left\langle x\right\rangle ^{-1}u\right\Vert _{L_{x}^{\infty}L_{t}^{2}}\lesssim\epsilon\label{eq:boot5}
\end{equation}
which follows from the improved decay rate of $u$. Clearly all of  these
estimates hold for the homogeneous evolution $u_{\text{hm}}=e^{itH}u_{0}$, see Lemma \ref{lem:smoothing} and Lemma \ref{lem:imprsmoothing}.

\begin{rem}[Fourier transform in $t$]\label{rem:FTt}
Technically, here we assume the global existence of the solution $u$
which is true in the setting of $\eta$. Given the global existence
of $u$, the bootstrap space given by \eqref{eq:bootstrap1-1} is restricted
onto $\left[0,T\right]$.  But this is just a minor technical
point. One can treat our analysis here as \emph{a priori} estimates and one can overcome this easily by constructing a sequence $u_{n}$
converging to $u$ by the standard iteration. For each element in
this sequence, one can always perform the Fourier transform with respect
to $t$.  We can use the standard Picard iteration:
\begin{equation}
u_{n+1}(t)=u_{0}\left(t\right)+\int_{0<s<t}e^{i\left(t-s\right)H-\epsilon\left(t-s\right)}F\left(s,u_{n}\right)\,ds\label{eq:iteration}
\end{equation}
such that $u_n$ converges to $u$.
See Appendix \ref{sec:GWP} for more details.
\end{rem}

First of all, we also have the following observation based on \eqref{eq:boot3},
\eqref{eq:boot4} and \eqref{eq:boot5}.
\begin{lem}\label{lem:beta}
Given the estimates \eqref{eq:boot3}, \eqref{eq:boot4} and \eqref{eq:boot5},
one can choose an appropriate $\beta>\alpha>0$ such that
\[
\left\Vert \left\langle x\right\rangle ^{-2}t^{1-\beta}\left(-2iE\left[z(t)\right]\partial_{t}+\partial_{t}^{2}\right)\left(u_{L}\right)\right\Vert _{L_{x}^{\infty}L_{t}^{2}\left[0,T\right]}\lesssim\epsilon
\]
and
\[
\left\Vert \left\langle x\right\rangle ^{-2}\partial_{x}^{j}t^{1-\beta}\left(u_{H}\right)\right\Vert _{L_{x}^{\infty}L_{t}^{2}\left[0,T\right]}\lesssim\epsilon\,\quad j=0,1.
\]
 \end{lem}

The point here is that we can eliminate the mild growth in \eqref{eq:boot3}
and \eqref{eq:boot4}. These estimates will be helpful in our pointwise
estimates for $\tilde{f}\left(t,k\right)$ later on.
\begin{proof}
To obtain the desired result, we apply the dyadic decomposition in time.
For the low frequency part, we consider
\[
\left(-2iE\left[z(t)\right]\partial_{t}+\partial_{t}^{2}\right)\left(\varphi\left(\frac{t}{2^{j}}\right)u_{L}\right)
\]
where $\varphi$ is a non-negative smooth bump function that takes $1$ from $\frac{3}{4}$
to $\frac{5}{4}$ and decays to $0$  quickly. 

Direct computations give
\[
\partial_{t}\left(\varphi\left(\frac{t}{2^{j}}\right)u_{L}\right)=\frac{1}{2^{j}}\varphi'\left(\frac{t}{2^{j}}\right)u_{L}+\varphi\left(\frac{t}{2^{j}}\right)\partial_{t}u_{L}.
\]
\[
\partial_{t}^{2}\left(\varphi\left(\frac{t}{2^{j}}\right)u_{L}\right)=\frac{1}{2^{2j}}\varphi''\left(\frac{t}{2^{j}}\right)u_{L}+\frac{1}{2^{j}}\varphi'\left(\frac{t}{2^{j}}\right)\partial_{t}u_{L}+\varphi\left(\frac{t}{2^{j}}\right)\partial_{t}^{2}u_{L}.
\]
Then applying the desired norm to the expressions above, one has
\begin{align*}
\left\Vert \left\langle x\right\rangle ^{-2}t^{1-\beta}\left(-2iE\left[z(t)\right]\partial_{t}+\partial_{t}^{2}\right)\left(\varphi\left(\frac{t}{2^{j}}\right)u_{L}\right)\right\Vert _{L_{x}^{\infty}L_{t}^{2}\left[0,T\right]}\\
\lesssim\left\Vert \left\langle x\right\rangle ^{-2}\frac{t^{-\beta}}{2^{2j}}t\varphi''\left(\frac{t}{2^{j}}\right)u_{L}\right\Vert _{L_{x}^{\infty}L_{t}^{2}\left[0,T\right]}+\left\Vert \left\langle x\right\rangle ^{-2}\frac{t^{-\beta}}{2^{j}}t\varphi'\left(\frac{t}{2^{j}}\right)\partial_{t}u_{L}\right\Vert _{L_{x}^{\infty}L_{t}^{2}\left[0,T\right]}\\
+\left\Vert \left\langle x\right\rangle ^{-2}\varphi\left(\frac{t}{2^{j}}\right)t^{1-\beta}\left(-2iE\left[z(t)\right]\partial_{t}+\partial_{t}^{2}\right)u_{L}\right\Vert _{L_{x}^{\infty}L_{t}^{2}\left[0,T\right]}\\
\lesssim\epsilon\left(2^{j}\right)^{-\beta}+\epsilon\left(2^{j}\right)^{-\beta}\left(2^{j}\right)^{\alpha}
\end{align*}
where we applied \eqref{eq:boot5} to estimate the the first piece and
\eqref{eq:boot3} to bound the second piece. Therefore, summing up the
dyadic pieces, we get
\begin{align*}
\left\Vert \left\langle x\right\rangle ^{-2}t^{1-\beta}\left(-2iE\left[z(t)\right]\partial_{t}+\partial_{t}^{2}\right)\left(u_{L}\right)\right\Vert _{L_{x}^{\infty}L_{t}^{2}\left[0,T\right]}\\
\lesssim\sum_{j=0}^{\infty}\left\Vert \left\langle x\right\rangle ^{-2}t^{1-\beta}\left(-2iE\left[z(t)\right]\partial_{t}+\partial_{t}^{2}\right)\left(\varphi\left(\frac{t}{2^{j}}\right)u_{L}\right)\right\Vert _{L_{x}^{\infty}L_{t}^{2}\left[0,T\right]}\\
\lesssim\epsilon\sum_{j=0}^{\infty}\left(2^{j}\right)^{-\beta}\left(2^{j}\right)^{\alpha}\lesssim\epsilon
\end{align*}
provided that $\alpha<\beta$.

Similarly, one can bound that
\[
\left\Vert \left\langle x\right\rangle ^{-2}\partial_{x}^{j}t^{1-\beta}\left(u_{H}\right)\right\Vert _{L_{x}^{\infty}L_{t}^{2}\left[0,T\right]}\lesssim\epsilon\,\quad j=0,1.
\]
for $\beta>\alpha$. These conclude the proof.
\end{proof}
\begin{rem}
One can also apply the dyadic decomposition using $\varphi(\frac{2^j t}{T})$ and sum over $1\leq j\leq \log T$. The result will be the same.
\end{rem}
\subsubsection{Bootstrap}\label{subsubsec:bootauxmodel}
Here we show that we can close the additional bootstrap
assumptions for the smoothing estimates. 

Consider the model problem
\begin{align*}
i\partial_{t}u-\partial_{xx}u+Vu & =a_{1}\left(x\right)u+a_{2}\left(x\right)e^{2i\int_{0}^{t}E\left[z(\sigma)\right]\,d\sigma}\bar{u}+b\left(x\right)u^{2}+\left|u\right|^{2}u\\
 & =N_{1,1}+N_{1,2}+N_{2}+N_{3}=F
\end{align*}
where $N_{1,1}=a_{1}(x)u$, $N_{1,2}=a_{2}\left(x\right)e^{2i\int_{0}^{t}E\left[z(\sigma)\right]\,d\sigma}\bar{u}$,
$N_{2}=b\left(x\right)u^{2}$ and $N_{3}=\left|u\right|^{2}u$.

By the Duhamel expansion,
\[
u(t)=e^{iH\left(t-1\right)}u_{0}+\int_{1}^{t}e^{i\left(t-s\right)H}\left(N_{1,1}+N_{1,2}+N_{2}+N_{3}\right)\,ds.
\]
Taking decomposition as \eqref{eq:LFt} and \eqref{eq:HFt}, for the homogeneous
part, by Lemma \ref{lem:imprsmoothing}, one has
\begin{equation}
\left\Vert \left\langle x\right\rangle ^{-2}\left(-2iE\left[z(t)\right]\partial_{t}+\partial_{t}^{2}\right)\left(e^{iH\left(t-1\right)}u_{0}\right)_{L}\right\Vert _{L_{x}^{\infty}L_{t}^{2}}\lesssim\left\Vert \tilde{u}_{0}\right\Vert _{H^{1}},\label{eq:boot3-1-1}
\end{equation}
and
\begin{equation}
\left\Vert \left\langle x\right\rangle ^{-2}\partial_{x}^{j}t\left(e^{iH\left(t-1\right)}u_{0}\right)_{H}\right\Vert _{L_{x}^{\infty}L_{t}^{2}}\lesssim\epsilon\left\Vert \tilde{u}_{0}\right\Vert _{H^{1}},\,j=0,1.\label{eq:boot4-1-1}
\end{equation}
For the inhomogenous terms involving $N_{2}$ and $N_{3}$, we simply
apply Minkowski's inequality and the homogeneous estimates.

In the low frequency part, for $j=2,3$, one has
\begin{align}
\left|\left\langle x\right\rangle ^{-2}t\left(-2iE\left[z(t)\right]\partial_{t}+\partial_{t}^{2}\right)\left(\int_{1}^{t}\left(e^{i\left(t-s\right)H}\left(N_{j}\right)\right)\,ds\right)_{L}\right|\nonumber \\
\lesssim\left|\left\langle x\right\rangle ^{-2}tN_{j}\left(t\right)\right|
+\left|\int_{1}^{t}\left\langle x\right\rangle ^{-2}t\left(-2iE\left[z(t)\right]\partial_{t}+\partial_{t}^{2}\right)\left(e^{i\left(t-s\right)H}N_{j}\right)_{L}\,ds\right|\label{eq:low1explit}
\end{align}
since we restrict onto the low frequency part with respect to the
time frequency.

By the bootstrap assumption \eqref{eq:bootstrap1-1},
\[
\left|\left\langle x\right\rangle ^{-2}tN_{2}\left(t\right)\right|\lesssim t^{-1+2\alpha},\,\left|\left\langle x\right\rangle ^{-2}tN_{3}\left(t\right)\right|\lesssim t^{-2+2\alpha}
\]
which are both $L_{t}^{2}$ integrable.

To deal with the second term on the RHS of \eqref{eq:low1explit} and
the high frequency part, we note that
\[
\left\Vert \int_{1}^{t}\left\langle x\right\rangle ^{-2}t\left(-2iE\left[z(t)\right]\partial_{t}+\partial_{t}^{2}\right)\left(e^{i\left(t-s\right)H}N_{j}\right)_{L}\,ds\right\Vert _{L_{x}^{\infty}L_{t}^{2}\left[0,T\right]}\lesssim\int_{1}^{T}\left\Vert J\left(e^{-isH}N_{j}\right)\right\Vert _{L^{2}}\,ds
\]
\[
\left\Vert \int_{1}^{t}\left\langle x\right\rangle ^{-2}\partial_{x}^{j}t\left(e^{i\left(t-s\right)H}N_{j}\right)_{H}\,ds\right\Vert _{L_{x}^{\infty}L_{t}^{2}\left[0,T\right]}\lesssim\int_{0}^{T}\left\Vert J\left(e^{-isH}N_{j}\right)\right\Vert _{L^{2}}\,ds
\]
where the $J$ operator is given by \eqref{eq:Jop}. Therefore, two
expressions above can be bounded by the weighted estimates for the
quadratic term, Proposition \ref{pro:weightmain2} and the cubic term, see Proposition \ref{pro:weightmain}. In particular, we get
\begin{equation}
\left\Vert \left\langle x\right\rangle ^{-2}t\left(-2iE\left[z(t)\right]\partial_{t}+\partial_{t}^{2}\right)\left(\int_{1}^{t}\left(e^{i\left(t-s\right)H}N_{j}\right)_{L}\,ds\right)\right\Vert _{L_{x}^{\infty}L_{t}^{2}\left[0,T\right]}\lesssim\epsilon^{3}T^{\alpha}\label{eq:N1boot}
\end{equation}
\begin{equation}
\left\Vert \int_{1}^{t}\left\langle x\right\rangle ^{-2}\partial_{x}^{j}t\left(e^{i\left(t-s\right)H}N_{j}\right)_{H}\,ds\right\Vert _{L_{x}^{\infty}L_{t}^{2}\left[0,T\right]}\lesssim\epsilon^{3}T^{\alpha}.\label{eq:N2boot}
\end{equation}
It remains to analyze
\begin{equation}
h_{1,1}(t,x)=\int_{0}^{t}e^{i(t-s)H}N_{1,1}\left(s\right)\,ds=\int_{0}^{t}e^{i(t-s)H}a_{1}u(s)\,ds\label{eq:h11}
\end{equation}
and
\begin{equation}
h_{1,2}(t,x)=\int_{0}^{t}e^{i(t-s)H}N_{1,2}\left(s\right)\,ds=\int_{0}^{t}e^{i(t-s)H}a_{2}\left(x\right)e^{2i\int_{0}^{t}E\left[z(\sigma)\right]\,d\sigma}\bar{u}\,ds.\label{eq:h22}
\end{equation}
As we observed above by \eqref{eq:boot3add}, it suffices to show the
corresponding estimates with $E\left[z(t)\right]$ replaced by $E\left[z(\infty)\right]$
or $E\left[z(T)\right]$. For the sake of simplicity, we use $E\left[z\left(\infty\right)\right]$
here.

Taking the  Fourier transform in $t$, from the Laplace transform
of resolvents, one has
\begin{equation}
\mathcal{F}_{T}\left(h_{1,1}\right)\left(\tau\right)=\lim_{\epsilon\rightarrow0^{+}}\int_{0}^{\infty}e^{-is\tau}R\left(\tau-i\epsilon\right)a_{1}u(s)\,ds.\label{eq:Fth1}
\end{equation}
\begin{equation}
\mathcal{F}_{T}\left(h_{1,2}\right)\left(\tau\right)=\lim_{\epsilon\rightarrow0^{+}}\int_{0}^{\infty}e^{-is\tau}R\left(\tau-i\epsilon\right)a_{2}\left(x\right)e^{2i\int_{0}^{t}E\left[z(\sigma)\right]\,d\sigma}\bar{u}\,ds\label{eq:Fth2}
\end{equation}
where $R$ is the  resolvent of $H=-\partial_{xx}+V$.

By the duality of the Fourier space and the physical space, it suffices
to analyze
\[
\partial_{\tau}\left(-\tau\left(\tau+2E\left[z\left(\infty\right)\right]\right)\mathcal{F}_{T}\left[h_{1,j}\right]\left(\tau\right)\right),\ j=1,2
\]
restricted onto the low frequency part. By Plancherel's theorem, we
need to estimate the $L_{\tau}^{2}$ norm of the expression above.

\subsubsection*{Analysis of $h_{1,1}$}

We first compute bounds for $h_{1,1}$. Explicitly, we have
\begin{align}
\partial_{\tau}\left(\left(-\tau\left(\tau+2E\left[z\left(\infty\right)\right]\right)\right)\phi_{1}\left(\tau\right)\left(\int_{0}^{\infty}e^{-is\tau}R\left(\tau-i\epsilon\right)a_{1}u(s)\,ds\right)\right)\label{eq:LF1exp}\\
=\left[\partial_{\tau}\left(\left(-\tau\left(\tau+2E\left[z\left(\infty\right)\right]\right)\right)\phi_{1}\left(\tau\right)\right)\right]\int_{0}^{\infty}e^{-is\tau}R\left(\tau-i\epsilon\right)a_{1}u(s)\,ds\nonumber \\
+\left(-\tau\left(\tau+2E\left[z\left(\infty\right)\right]\right)\phi_{1}\left(\tau\right)\right)\int_{0}^{\infty}e^{is\tau}\partial_{\tau}R\left(\tau-i\epsilon\right)a_{1}u(s)\,ds\nonumber \\
+\left(-\tau\left(\tau+2E\left[z\left(\infty\right)\right]\right)\phi_{1}\left(\tau\right)\right)\int_{0}^{\infty}e^{-is\tau}R\left(\tau-i\epsilon\right)sa_{1}u(s)\,ds.\nonumber 
\end{align}
We estimate terms on the RHS of the expression above separately. From Lemma \ref{lem:limitingab}, one has{\footnotesize
%
\begin{align}
\lim_{\epsilon\rightarrow0^{+}}\left\Vert \left\langle x\right\rangle ^{-2}\left[\partial_{\tau}\left(\left(-\tau\left(\tau+2E\left[z\left(\infty\right)\right]\right)\right)\phi_{1}\left(\tau\right)\right)\right]\int_{0}^{\infty}e^{-is\tau}R\left(\tau-i\epsilon\right)a_{1}u(s)\,ds\right\Vert _{L_{x}^{\infty}L_{\tau}^{2}}\nonumber \\
\lesssim\lim_{\epsilon\rightarrow0^{+}}\left\Vert \left\langle x\right\rangle ^{-2}\left[\partial_{\tau}\left(\left(-\tau\left(\tau+2E\left[z\left(\infty\right)\right]\right)\right)\phi_{1}\left(\tau\right)\right)\right]R\left(\tau-i\epsilon\right)a_{1}\mathcal{F}_{T}\left[u\right]\left(\tau\right)\right\Vert _{L_{x}^{\infty}L_{\tau}^{2}}\nonumber \\
\lesssim\lim_{\epsilon\rightarrow0^{+}}\epsilon^{2}\left\Vert \left\langle x\right\rangle ^{-2}\phi_{1}\left(\tau\right)R\left(\tau-i\epsilon\right)\left\langle x\right\rangle ^{-2}\right\Vert _{L_{x}^{\infty}L_{\tau}^{\infty}}\left\Vert \left\langle x\right\rangle ^{-2}\mathcal{F}_{T}\left[u\right]\left(\tau\right)\right\Vert _{L_{x}^{\infty}L_{\tau}^{2}}
\lesssim\epsilon^{2}\left\Vert \left\langle x\right\rangle ^{-2}u\left(t\right)\right\Vert _{L_{x}^{\infty}L_{t}^{2}}.\label{eq:L1}
\end{align}}
Then one can apply \eqref{eq:boot5} to estimate the last line. 

Similarly, we also have
\begin{align}
\lim_{\epsilon\rightarrow0^{+}}\left\Vert \left\langle x\right\rangle ^{-2}\left(-\tau\left(\tau+2E\left[z\left(\infty\right)\right]\right)\phi_{1}\left(\tau\right)\right)\int_{0}^{\infty}e^{is\tau}\partial_{\tau}R\left(\tau-i\epsilon\right)a_{1}u(s)\,ds\right\Vert _{L_{x}^{\infty}L_{\tau}^{2}}\nonumber \\
\lesssim\lim_{\epsilon\rightarrow0^{+}}\left\Vert \left\langle x\right\rangle ^{-2}\left(\tau+2E\left[z\left(\infty\right)\right]\right)\phi_{1}\left(\tau\right)\tau\partial_{\tau}R\left(\tau-i\epsilon\right)a_{1}\mathcal{F}_{T}\left[u\right]\left(\tau\right)\right\Vert _{L_{x}^{\infty}L_{\tau}^{2}}\nonumber \\
\lesssim\lim_{\epsilon\rightarrow0^{+}}\epsilon^{2}\left\Vert \left\langle x\right\rangle ^{-2}\phi_{1}\left(\tau\right)\tau\partial_{\tau}R\left(\tau-i\epsilon\right)\left\langle x\right\rangle ^{-2}\right\Vert _{L_{x}^{\infty}L_{\tau}^{\infty}}\left\Vert \left\langle x\right\rangle ^{-2}\mathcal{F}_{T}\left[u\right]\left(\tau\right)\right\Vert _{L_{x}^{\infty}L_{\tau}^{2}}\nonumber \\
\lesssim\epsilon^{2}\left\Vert \left\langle x\right\rangle ^{-2}u\left(t\right)\right\Vert _{L_{x}^{\infty}L_{t}^{2}}.\label{eq:L2}
\end{align}
To deal with the last term from the RHS of \eqref{eq:LF1exp}, splitting
$u=u_{L}+u_{H}$, one has
\begin{align}
-\tau\left(\tau+2E\left[z\left(\infty\right)\right]\right)\phi_{1}\left(\tau\right)\int_{0}^{\infty}e^{-is\tau}R\left(\tau-i\epsilon\right)sa_{1}u(s)\,ds\label{eq:L3}\\
=-\tau\left(\tau+2E\left[z\left(\infty\right)\right]\right)\phi_{1}\left(\tau\right)\int_{0}^{\infty}e^{-is\tau}R\left(\tau-i\epsilon\right)sa_{1}\left(u_{L}+u_{H}\right)\,ds.\nonumber 
\end{align}
The high frequency part can be bounded as above:
\begin{align}
\lim_{\epsilon\rightarrow0^{+}}\left\Vert \left\langle x\right\rangle ^{-2}\tau\left(\tau+2E\left[z\left(\infty\right)\right]\right)\phi_{1}\left(\tau\right)\int_{0}^{\infty}e^{-is\tau}R\left(\tau-i\epsilon\right)sa_{1}u_{H}(s)\,ds\right\Vert _{L_{x}^{\infty}L_{\tau}^{2}}\nonumber \\
\lesssim\lim_{\epsilon\rightarrow0^{+}}\left\Vert \left\langle x\right\rangle ^{-2}\phi_{1}\left(\tau\right)R\left(\tau-i\epsilon\right)a_{1}\mathcal{F}_{T}\left[su_{H}\right]\left(\tau\right)\right\Vert _{L_{x}^{\infty}L_{\tau}^{2}}\nonumber \\
\lesssim\lim_{\epsilon\rightarrow0^{+}}\epsilon^{2}\left\Vert \left\langle x\right\rangle ^{-2}\phi_{1}\left(\tau\right)R\left(\tau-i\epsilon\right)\left\langle x\right\rangle ^{-2}\right\Vert _{L_{x}^{\infty}L_{\tau}^{\infty}}\left\Vert \left\langle x\right\rangle ^{-2}\mathcal{F}_{T}\left[su_{H}\right]\left(\tau\right)\right\Vert _{L_{x}^{\infty}L_{\tau}^{2}}\nonumber \\
\lesssim\epsilon^{2}\left\Vert \left\langle x\right\rangle ^{-2}su_{H}\left(t\right)\right\Vert _{L_{x}^{\infty}L_{t}^{2}}\lesssim\epsilon^{3}T^{\alpha}\label{eq:LH}
\end{align}
where in the last line, we applied \eqref{eq:boot4}.

For the low frequency part in the inhomogeneous term, we observe that
by the Fourier transform
\begin{align}
\tau\left(\tau+2E\left[z\left(\infty\right)\right]\right)\phi_{1}\left(\tau\right)\int_{0}^{\infty}e^{-is\tau}R\left(\tau-i\epsilon\right)sa_{1}\left(u_{L}\right)\,ds\nonumber \\
=\tau\left(\tau+2E\left[z\left(\infty\right)\right]\right)\phi_{1}\left(\tau\right)R\left(\tau-i\epsilon\right)\mathcal{F}_{T}\left[sa_{1}\left(u_{L}\right)\right](\tau)\nonumber \\
\sim\phi_{1}\left(\tau\right)R\left(\tau-i\epsilon\right)F_{T}^{-1}\left[\left(-2iE\left[z(\infty)\right]\partial_{t}+\partial_{t}^{2}\right)\left(sa_{1}\left(u_{L}\right)\right)\right](\tau & )\label{eq:IBPLL}
\end{align}
where in the last line above, we used the duality of Fourier transforms.
Then to bound the terms from the RHS of the expression above is similar
to estimates above as \eqref{eq:L1}, \eqref{eq:L2} and \eqref{eq:LH}.
Explicitly, we can bound{\small
\begin{align}
\lim_{\epsilon\rightarrow0^{+}}\left\Vert \left\langle x\right\rangle ^{-2}\tau\left(\tau+2E\left[z\left(\infty\right)\right]\right)\phi_{1}\left(\tau\right)\int_{0}^{\infty}e^{-is\tau}R\left(\tau-i\epsilon\right)sa_{1}\left(u_{L}\right)\,ds\right\Vert _{L_{x}^{\infty}L_{\tau}^{2}}\nonumber \\
\lesssim\lim_{\epsilon\rightarrow0^{+}}\left\Vert \left\langle x\right\rangle ^{-2}\phi_{1}\left(\tau\right)R\left(\tau-i\epsilon\right)F_{T}^{-1}\left[\left(-2iE\left[z(\infty)\right]\partial_{t}+\partial_{t}^{2}\right)\left(sa_{1}\left(u_{L}\right)\right)\right](\tau)\right\Vert _{L_{x}^{\infty}L_{\tau}^{2}}\nonumber \\
\lesssim\epsilon^{2}\left\Vert \left\langle x\right\rangle ^{-2}u_{L}\left(t\right)\right\Vert _{L_{x}^{\infty}L_{t}^{2}}+\left\Vert \left\langle x\right\rangle ^{-2}\partial_{t}u_{L}\left(t\right)\right\Vert _{L_{x}^{\infty}L_{t}^{2}}\label{eq:LL}\\
+\epsilon^{2}\left\Vert \left\langle x\right\rangle ^{-2}s\left(-2iE\left[z(\infty)\right]\partial_{t}+\partial_{t}^{2}\right)u_{L}\left(t\right)\right\Vert _{L_{x}^{\infty}L_{t}^{2}}\nonumber \\
\lesssim\epsilon^{2}\left\Vert \left\langle x\right\rangle ^{-2}u_{L}\left(t\right)\right\Vert _{L_{x}^{\infty}L_{t}^{2}}+\epsilon^{2}\left\Vert \left\langle x\right\rangle ^{-2}\partial_{t}u_{L}\left(t\right)\right\Vert _{L_{x}^{\infty}L_{t}^{2}}\nonumber \\
+\epsilon^{2}\left\Vert \left\langle x\right\rangle ^{-2}s\left(-2iE\left[z(t)\right]\partial_{t}+\partial_{t}^{2}\right)u_{L}\left(t\right)\right\Vert _{L_{x}^{\infty}L_{t}^{2}}
\lesssim\epsilon^{3}T^{\alpha}\nonumber 
\end{align}}
where in the last two lines, we used the bootstrap assumption \eqref{eq:boot3}
and the observation \eqref{eq:boot3add}.

For the high frequency part, we need to bound the $L_{\tau}^{2}$
bound of the following expression:
\begin{align*}
\int e^{it\tau}t\left(h_{1,1}\right)_{H}\,dt= & \partial_{\tau}\left(\phi_{2}\left(\tau\right)\left(\int_{0}^{\infty}e^{-is\tau}R\left(\tau-i\epsilon\right)a_{1}u(s)\,ds\right)\right).
\end{align*}
Expanding everything, one has
\begin{align}
\partial_{\tau}\left(\phi_{2}\left(\tau\right)\left(\int_{0}^{\infty}e^{-is\tau}R\left(\tau-i\epsilon\right)a_{1}u(s)\,ds\right)\right)\label{eq:Hsplit}\\
=\phi_{2}\left(\tau\right)\int_{0}^{\infty}e^{is\tau}\partial_{\tau}R\left(\tau-i\epsilon\right)a_{1}u(s)\,ds\nonumber \\
+\phi_{2}\left(\tau\right)\int_{0}^{\infty}e^{-is\tau}R\left(\tau-i\epsilon\right)sa_{1}u(s)\,ds\nonumber \\
+\partial_{\tau}\phi_{2}\left(\tau\right)\int_{0}^{\infty}e^{-is\tau}R\left(\tau-i\epsilon\right)a_{1}u(s)\,ds.\nonumber 
\end{align}
The last term above is easy to bound as the analysis of \eqref{eq:L1} for the low frequency part.

For the first term, one can bound
\begin{align}
\lim_{\epsilon\rightarrow0^{+}}\left\Vert \left\langle x\right\rangle ^{-2}\phi_{2}\left(\tau\right)\int_{0}^{\infty}e^{is\tau}\partial_{\tau}R\left(\tau-i\epsilon\right)a_{1}u(s)\,ds\right\Vert _{L_{x}^{\infty}L_{\tau}^{2}}\nonumber \\
\lesssim\lim_{\epsilon\rightarrow0^{+}}\left\Vert \left\langle x\right\rangle ^{-2}\lim_{\epsilon\rightarrow0^{+}}\left\Vert \left\langle x\right\rangle ^{-2}\phi_{2}\left(\tau\right)\partial_{\tau}R\left(\tau-i\epsilon\right)a_{1}\mathcal{F}_{T}\left[u\right]\left(\tau\right)\right\Vert _{L_{x}^{2}L_{\tau}^{2}}\right\Vert _{L_{x}^{\infty}L_{\tau}^{2}}\nonumber \\
\lesssim\lim_{\epsilon\rightarrow0^{+}}\epsilon^{2}\left\Vert \left\langle x\right\rangle ^{-2}\phi_{2}\left(\tau\right)\partial_{\tau}R\left(\tau-i\epsilon\right)\left\langle x\right\rangle ^{-2}\right\Vert _{L_{x}^{\infty}L_{\tau}^{\infty}}\left\Vert \left\langle x\right\rangle ^{-2}\mathcal{F}_{T}\left[u\right]\left(\tau\right)\right\Vert _{L_{x}^{\infty}L_{\tau}^{2}}\nonumber \\
\lesssim\epsilon^{2}\left\Vert \left\langle x\right\rangle ^{-2}u\left(t\right)\right\Vert _{L_{x}^{\infty}L_{t}^{2}}\lesssim\epsilon^{3} & .\label{eq:H1}
\end{align}
For the second term on the RHS of \eqref{eq:Hsplit}, as before we split
$u=u_{L}+u_{H}$ and rewrite it as
\[
\phi_{2}\left(\tau\right)\int_{0}^{\infty}e^{-is\tau}R\left(\tau-i\epsilon\right)sa_{1}u\left(s\right)\,ds=\phi_{2}\left(\tau\right)\int_{0}^{\infty}e^{-is\tau}R\left(\tau-i\epsilon\right)sa_{1}\left(u_{L}+u_{H}\right)\,ds.
\]
For the high frequency part, one can bound
\begin{align}
\lim_{\epsilon\rightarrow0^{+}}\left\Vert \left\langle x\right\rangle ^{-2}\phi_{2}\left(\tau\right)\int_{0}^{\infty}e^{-is\tau}R\left(\tau-i\epsilon\right)sa_{1}u_{H}(s)\,ds\right\Vert _{L_{x}^{\infty}L_{\tau}^{2}}\nonumber \\
\lesssim\lim_{\epsilon\rightarrow0^{+}}\left\Vert \left\langle x\right\rangle ^{-2}\phi_{2}\left(\tau\right)R\left(\tau-i\epsilon\right)a_{1}\mathcal{F}_{T}\left[su_{H}\right]\left(\tau\right)\right\Vert _{L_{x}^{\infty}L_{\tau}^{2}}\nonumber \\
\lesssim\lim_{\epsilon\rightarrow0^{+}}\epsilon^{2}\left\Vert \left\langle x\right\rangle ^{-2}\phi_{2}\left(\tau\right)R\left(\tau-i\epsilon\right)\left\langle x\right\rangle ^{-2}\right\Vert _{L_{x}^{\infty}L_{\tau}^{\infty}}\left\Vert \left\langle x\right\rangle ^{-2}\mathcal{F}_{T}\left[su_{H}\right]\left(\tau\right)\right\Vert _{L_{x}^{\infty}L_{\tau}^{2}}\nonumber \\
\lesssim\epsilon^{2}\left\Vert \left\langle x\right\rangle ^{-2}su_{H}\left(t\right)\right\Vert _{L_{x}^{\infty}L_{t}^{2}}\lesssim\epsilon^{3}T^{\alpha}.\label{eq:HH}
\end{align}
Again for the low frequency part, we also note that by the Fourier
transform
\begin{align}
\phi_{2}\left(\tau\right)\int_{0}^{\infty}e^{-is\tau}R\left(\tau-i\epsilon\right)sa_{1}\left(u_{L}\right)\,ds=\ \ \ \ \ \ \ \ \ \ \ \nonumber \\
=\frac{1}{\tau\left(\tau+2E\left[z\left(\infty\right)\right]\right)}\phi_{2}\left(\tau\right)R\left(\tau-i\epsilon\right)\left(\tau\left(\tau+2E\left[z\left(\infty\right)\right]\right)\right)\mathcal{F}_{T}\left[sa_{1}\left(u_{L}\right)\right](\tau)\nonumber \\
\sim\frac{1}{\tau\left(\tau+2E\left[z\left(\infty\right)\right]\right)}\phi_{2}\left(\tau\right)R\left(\tau-i\epsilon\right)\mathcal{F}_{T}\left[\left(-2iE\left[z(\infty)\right]\partial_{s}+\partial_{s}^{2}\right)\left(sa_{1}\left(u_{L}\right)\right)\right](\tau)\label{eq:IBPHL}
\end{align}
Note that in the support of $\phi_{2}$, $\frac{1}{\tau\left(\tau+2E\left[z\left(\infty\right)\right]\right)}$
is bounded. Therefore, we can bound{\small
\begin{align}
\lim_{\epsilon\rightarrow0^{+}}\left\Vert \left\langle x\right\rangle ^{-2}\phi_{2}\left(\tau\right)\int_{0}^{\infty}e^{-is\tau}R\left(\tau-i\epsilon\right)sa_{1}\left(u_{L}\right)\,ds\right\Vert _{L_{x}^{\infty}L_{\tau}^{2}}\nonumber \\
\lesssim\lim_{\epsilon\rightarrow0^{+}}\left\Vert \left\langle x\right\rangle ^{-2}\frac{1}{\tau\left(\tau+2E\left[z\left(\infty\right)\right]\right)}\phi_{2}\left(\tau\right)R\left(\tau-i\epsilon\right)\mathcal{F}_{T}\left[\left(-2iE\left[z(\infty)\right]\partial_{s}+\partial_{s}^{2}\right)\left(sa_{1}\left(u_{L}\right)\right)\right](\tau)\right\Vert _{L_{x}^{\infty}L_{\tau}^{2}}\nonumber \\
\lesssim\epsilon^{2}\left\Vert \left\langle x\right\rangle ^{-2}u_{L}\left(t\right)\right\Vert _{L_{x}^{\infty}L_{t}^{2}}+\epsilon^{2}\left\Vert \left\langle x\right\rangle ^{-2}\partial_{t}u_{L}\left(t\right)\right\Vert _{L_{x}^{\infty}L_{t}^{2}}\nonumber \\
+\epsilon^{2}\left\Vert \left\langle x\right\rangle ^{-2}s\left(-2iE\left[z(s)\right]\partial_{t}+\partial_{t}^{2}\right)u_{L}\left(t\right)\right\Vert _{L_{x}^{\infty}L_{t}^{2}}\lesssim\epsilon^{3}T^{\alpha}\label{eq:HL}
\end{align}}
as in the analysis of \eqref{eq:LL}.

Putting all of estimates above together, we finally conclude that
\begin{equation}
\left\Vert \left\langle x\right\rangle ^{-2}t\left(-2iE\left[z(t)\right]\partial_{t}+\partial_{t}^{2}\right)\left(h_{1,1}\right)_{L}\right\Vert _{L_{x}^{\infty}L_{t}^{2}\left[0,T\right]}+\left\Vert \left\langle x\right\rangle ^{-2}t\left(h_{1,1}\right)_{H}\right\Vert _{L_{x}^{\infty}L_{t}^{2}\left[0,T\right]}\lesssim\epsilon^{3}T^{\alpha}.\label{eq:addboot}
\end{equation}
Combing  with \eqref{eq:N1boot} and \eqref{eq:N2boot}, the estimate above  recovers the auxiliary bootstrap estimates \eqref{eq:boot3} and
\eqref{eq:boot4}
for the inhomogenous part $h_{1,1}$ defined by \eqref{eq:h11}.

\subsubsection*{Analysis of $h_{1,2}$}

The analysis of $h_{1,2}$ defined by \eqref{eq:h22} is similar. Starting
from the low frequency part, we have{\small
\begin{align}
\partial_{\tau}\left(\left(-\tau\left(\tau+2E\left[z\left(\infty\right)\right]\right)\right)\phi_{1}\left(\tau\right)\left(\int_{0}^{\infty}e^{-is\tau}R\left(\tau-i\epsilon\right)a_{2}\left(e^{2i\int_{0}^{s}E\left[z(\sigma)\right]\,d\sigma}\overline{u}\right)\,ds\right)\right)\label{eq:LF1exp2}\\
=\left[\partial_{\tau}\left(\left(-\tau\left(\tau+2E\left[z\left(\infty\right)\right]\right)\right)\phi_{1}\left(\tau\right)\right)\right]\int_{0}^{\infty}e^{-is\tau}R\left(\tau-i\epsilon\right)a_{2}\left(e^{2i\int_{0}^{s}E\left[z(\sigma)\right]\,d\sigma}\overline{u}\right)\,ds\nonumber \\
+\left(-\tau\left(\tau+2E\left[z\left(\infty\right)\right]\right)\phi_{1}\left(\tau\right)\right)\int_{0}^{\infty}e^{is\tau}\partial_{\tau}R\left(\tau-i\epsilon\right)a_{2}\left(e^{2i\int_{0}^{s}E\left[z(\sigma)\right]\,d\sigma}\overline{u}\right)\,ds\nonumber \\
+\left(-\tau\left(\tau+2E\left[z\left(\infty\right)\right]\right)\phi_{1}\left(\tau\right)\right)\int_{0}^{\infty}e^{-is\tau}R\left(\tau-i\epsilon\right)sa_{2}\left(e^{2i\int_{0}^{s}E\left[z(\sigma)\right]\,d\sigma}\overline{u}_{L}\right)\,ds.\nonumber 
\end{align}}
The first two expressions of \eqref{eq:LF1exp2} above can be bounded
by the same argument as $h_{1,1}$. The last term is different from
the setting of $h_{1,1}$ due to the extra time-dependent phase. Focusing
on the last term, one has
\begin{align}
\tau\left(\tau+2E\left[z\left(\infty\right)\right]\right)\phi_{1}\left(\tau\right)\int_{0}^{\infty}e^{-is\tau}R\left(\tau-i\epsilon\right)sa_{2}\left(e^{2i\int_{0}^{s}E\left[z(\sigma)\right]\,d\sigma}\overline{u}\right)\,ds\nonumber \\
=\tau\phi_{1}\left(\tau\right)\int_{0}^{\infty}\left(\tau+2E\left[z\left(s\right)\right]\right)e^{-is\tau}R\left(\tau-i\epsilon\right)sa_{2}\left(e^{2i\int_{0}^{s}E\left[z(\sigma)\right]\,d\sigma}\overline{u}\right)\,ds\label{eq:LF1exp3}\\
+\tau\phi_{1}\left(\tau\right)\int_{0}^{\infty}\left(2E\left[z\left(\infty\right)\right]-2E\left[z\left(s\right)\right]\right)e^{-is\tau}R\left(\tau-i\epsilon\right)sa_{2}\left(e^{2i\int_{0}^{s}E\left[z(\sigma)\right]\,d\sigma}\overline{u}\right)\,ds.\nonumber 
\end{align}
For the last term above, due to the decay of $\left|2E\left[z\left(\infty\right)\right]-2E\left[z\left(s\right)\right]\right|\lesssim \epsilon^2  s^{-1+2\alpha}$,
we can estimate{\small
\begin{align}
\left\Vert \left\langle x\right\rangle ^{-2}\tau\phi_{1}\left(\tau\right)\int_{0}^{\infty}\left(2E\left[z\left(\infty\right)\right]-2E\left[z\left(s\right)\right]\right)e^{-is\tau}R\left(\tau-i\epsilon\right)sa_{2}\left(e^{2i\int_{0}^{s}E\left[z(\sigma)\right]\,d\sigma}\overline{u}\right)\,ds\right\Vert _{L_{x}^{\infty}L_{\tau}^{2}}\nonumber \\
\lesssim\left\Vert \left\langle x\right\rangle ^{-2}\phi_{1}\left(\tau\right)R\left(\tau-i\epsilon\right)\left\langle x\right\rangle ^{-2}\right\Vert _{L_{x}^{\infty}L_{\tau}^{\infty}}\left\Vert \left(2E\left[z\left(\infty\right)\right]-2E\left[z\left(s\right)\right]\right)\left\langle x\right\rangle ^{2}sa_{2}\left(e^{2i\int_{0}^{s}E\left[z(\sigma)\right]\,d\sigma}\overline{u}\right)\right\Vert _{L_{x}^{\infty}L_{t}^{2}}\nonumber \\
\lesssim\epsilon^{4}\left\Vert \left\langle x\right\rangle ^{-2}ss^{-1+2\alpha}\overline{u}\right\Vert _{L_{x}^{\infty}L_{t}^{2}}\lesssim\epsilon^{5}\qquad\qquad\label{eq:h12L1} 
\end{align}}where in the last line, we used the improved localized decay of $u$ from Corollary \ref{cor:directXT}.

Then we analyze the first term on the RHS of \eqref{eq:LF1exp3}. Again
we write $u=u_{L}+u_{H}$. The high frequency part is easy to estimate:
\begin{align}
\left\Vert \left\langle x\right\rangle ^{-2}\tau\phi_{1}\left(\tau\right)\int_{0}^{\infty}\left(\tau+2E\left[z\left(s\right)\right]\right)e^{-is\tau}R\left(\tau-i\epsilon\right)sa_{2}\left(e^{2i\int_{0}^{s}E\left[z(\sigma)\right]\,d\sigma}\overline{u}_{H}\right)\,ds\right\Vert _{L_{x}^{\infty}L_{\tau}^{2}}\nonumber \\
\lesssim\epsilon^{2}\left\Vert \left\langle x\right\rangle ^{-2}\phi_{1}\left(\tau\right)R\left(\tau-i\epsilon\right)\left\langle x\right\rangle ^{-2}\right\Vert _{L_{x}^{\infty}L_{\tau}^{\infty}}\left\Vert \left\langle x\right\rangle ^{-2}s\overline{u}_{H}\right\Vert _{L_{x}^{\infty}L_{t}^{2}}
\lesssim\epsilon^{2}T^{\alpha}.\label{eq:LF1exp3h} 
\end{align}
For the low frequency part, we perform integration by parts in $s$ and obtain
\begin{align}
\tau\phi_{1}\left(\tau\right)\int_{0}^{\infty}\left(\tau+2E\left[z\left(s\right)\right]\right)e^{-is\tau}R\left(\tau-i\epsilon\right)sa_{2}\left(e^{2i\int_{0}^{s}E\left[z(\sigma)\right]\,d\sigma}\overline{u}_{L}\right)\,ds\nonumber \\
=\tau\phi_{1}\left(\tau\right)\int_{0}^{\infty}\partial_{s}\left(e^{-is\tau}e^{2i\int_{0}^{s}E\left[z(\sigma)\right]\,d\sigma}\right)R\left(\tau-i\epsilon\right)sa_{2}\left(\overline{u}_{L}\right)\,ds\nonumber \\
=-\tau\phi_{1}\left(\tau\right)\int_{0}^{\infty}e^{-is\tau}e^{2i\int_{0}^{s}E\left[z(\sigma)\right]\,d\sigma}R\left(\tau-i\epsilon\right)a_{2}\left(\overline{u}_{L}\right)\,ds\label{eq:LF1exp3L1}\\
-\tau\phi_{1}\left(\tau\right)\int_{0}^{\infty}e^{-is\tau}e^{2i\int_{0}^{s}E\left[z(\sigma)\right]\,d\sigma}R\left(\tau-i\epsilon\right)sa_{2}\left(\partial_{s}\overline{u}_{L}\right)\,ds.\nonumber 
\end{align}
The first term on the RHS of \eqref{eq:LF1exp3L1} is easy to bound
as the following
\begin{align}
\left\Vert \left\langle x\right\rangle ^{-2}\tau\phi_{1}\left(\tau\right)\int_{0}^{\infty}e^{-is\tau}e^{2i\int_{0}^{s}E\left[z(\sigma)\right]\,d\sigma}R\left(\tau-i\epsilon\right)a_{2}\left(\overline{u}_{L}\right)\,ds\right\Vert _{L_{x}^{\infty}L_{\tau}^{2}}\nonumber \\
\lesssim\epsilon^{2}\left\Vert \left\langle x\right\rangle ^{-2}\tau\phi_{1}\left(\tau\right)R\left(\tau-i\epsilon\right)\left\langle x\right\rangle ^{-2}\right\Vert _{L_{x}^{\infty}L_{\tau}^{\infty}}\left\Vert \left\langle x\right\rangle ^{-2}\overline{u}_{L}\right\Vert _{L_{x}^{\infty}L_{t}^{2}}\lesssim\epsilon^{3}.\label{eq:LF1exp3L11}
\end{align}
For the last term on the RHS of \eqref{eq:LF1exp3L1}, one has{\small
\begin{align}
\left\Vert \left\langle x\right\rangle ^{-2}\tau\phi_{1}\left(\tau\right)\int_{0}^{\infty}e^{-is\tau}e^{2i\int_{0}^{s}E\left[z(\sigma)\right]\,d\sigma}R\left(\tau-i\epsilon\right)sa_{2}\left(\partial_{s}\overline{u}_{L}\right)\right\Vert _{L_{x}^{\infty}L_{\tau}^{2}}\nonumber \\
\lesssim\epsilon^{2}\left\Vert \left\langle x\right\rangle ^{-2}\phi_{1}\left(\tau\right)R\left(\tau-i\epsilon\right)\left\langle x\right\rangle ^{-2}\right\Vert _{L_{x}^{\infty}L_{\tau}^{\infty}}\left\Vert \left\langle x\right\rangle ^{-2}\partial_{s}\left(se^{2i\int_{0}^{s}E\left[z(\sigma)\right]\,d\sigma}\partial_{s}\overline{u}_{L}\right)\right\Vert _{L_{x}^{\infty}L_{t}^{2}}\label{eq:LF1exp3L12}\\
\lesssim\epsilon^{2}\left\Vert \left\langle x\right\rangle ^{-2}u_{L}\left(t\right)\right\Vert _{L_{x}^{\infty}L_{t}^{2}}+\epsilon^{2}\left\Vert \left\langle x\right\rangle ^{-2}\partial_{t}u_{L}\left(t\right)\right\Vert _{L_{x}^{\infty}L_{t}^{2}}\nonumber \\
+\epsilon^{2}\left\Vert \left\langle x\right\rangle ^{-2}e^{2i\int_{0}^{s}E\left[z(\sigma)\right]\,d\sigma}\left(2iE\left[z\left(s\right)\right]\partial_{s}\overline{u}_{L}+\partial_{s}^{2}\overline{u}_{L}\right)\right\Vert _{L_{x}^{\infty}L_{t}^{2}}\nonumber \\
\lesssim\epsilon^{2}\left\Vert \left\langle x\right\rangle ^{-2}u_{L}\left(t\right)\right\Vert _{L_{x}^{\infty}L_{t}^{2}}+\epsilon^{2}\left\Vert \left\langle x\right\rangle ^{-2}\partial_{t}u_{L}\left(t\right)\right\Vert _{L_{x}^{\infty}L_{t}^{2}}\nonumber \\
+\epsilon^{2}\left\Vert \left\langle x\right\rangle ^{-2}s\left(-2iE\left[z(s)\right]\partial_{t}+\partial_{t}^{2}\right)u_{L}\left(t\right)\right\Vert _{L_{x}^{\infty}L_{t}^{2}}\nonumber 
\lesssim\epsilon^{3}T^{\alpha}\nonumber 
\end{align}}
which recovers the bootstrap conditions for the low frequency part.

For the high frequency part, we need to bound the $L_{\tau}^{2}$ norm
of\[
\partial_{\tau}\left(\phi_{2}\left(\tau\right)\left(\int_{0}^{\infty}e^{-is\tau}R\left(\tau-i\epsilon\right)a_{2}\left(e^{2i\int_{0}^{s}E\left[z(\sigma)\right]\,d\sigma}\overline{u}\right)\,ds\right)\right).
\]
Again explicitly, one has{\small
\begin{align}
\partial_{\tau}\left(\phi_{2}\left(\tau\right)\left(\int_{0}^{\infty}e^{-is\tau}R\left(\tau-i\epsilon\right)a_{2}\left(e^{2i\int_{0}^{s}E\left[z(\sigma)\right]\,d\sigma}\overline{u}\right)\,ds\right)\right)\label{eq:Hsplit-1}\\
=\phi_{2}\left(\tau\right)\int_{0}^{\infty}e^{is\tau}\partial_{\tau}R\left(\tau-i\epsilon\right)a_{2}\left(e^{2i\int_{0}^{s}E\left[z(\sigma)\right]\,d\sigma}\overline{u}\right)\,ds\nonumber \\
+\partial_{\tau}\phi_{2}\left(\tau\right)\int_{0}^{\infty}e^{-is\tau}R\left(\tau-i\epsilon\right)a_{2}\left(e^{2i\int_{0}^{s}E\left[z(\sigma)\right]\,d\sigma}\overline{u}\right)\,ds\nonumber \\
+\phi_{2}\left(\tau\right)\int_{0}^{\infty}e^{-is\tau}R\left(\tau-i\epsilon\right)sa_{2}\left(e^{2i\int_{0}^{s}E\left[z(\sigma)\right]\,d\sigma}\overline{u}\right)\,ds & .\nonumber 
\end{align}}
The first two terms on the RHS above can be bounded in the same manner
as the corresponding parts for $h_{1,1}$.

For the last term on the RHS of \eqref{eq:Hsplit-1}, as before we split
$u=u_{L}+u_{H}$ and rewrite it as
\begin{align}
\phi_{2}\left(\tau\right)\int_{0}^{\infty}e^{-is\tau}R\left(\tau-i\epsilon\right)sa_{2}\left(e^{2i\int_{0}^{s}E\left[z(\sigma)\right]\,d\sigma}\overline{u}\right)\,ds\nonumber \\
=\phi_{2}\left(\tau\right)\int_{0}^{\infty}e^{-is\tau}R\left(\tau-i\epsilon\right)sa_{2}\left(e^{2i\int_{0}^{s}E\left[z(\sigma)\right]\,d\sigma}\left(\overline{u}_{L}+\bar{u}_{H}\right)\right)\,ds.\label{eq:h12Hsplit}
\end{align}
For the high frequency part above, the analysis is the same as \eqref{eq:HH}.
For the low frequency part, we again note that by the Fourier transform{\small
\begin{align}
\phi_{2}\left(\tau\right)\int_{0}^{\infty}e^{-is\tau}R\left(\tau-i\epsilon\right)sa_{2}e^{2i\int_{0}^{s}E\left[z(\sigma)\right]\,d\sigma}\left(\overline{u}_{L}\right)\,ds=\ \ \ \ \ \ \ \ \ \ \ \nonumber \\
=\frac{1}{\tau\left(\tau+2E\left[z\left(\infty\right)\right]\right)}\phi_{2}\left(\tau\right)\left(\tau\left(\tau+2E\left[z\left(\infty\right)\right]\right)\right)\int_{0}^{\infty}e^{-is\tau}R\left(\tau-i\epsilon\right)sa_{2}e^{2i\int_{0}^{s}E\left[z(\sigma)\right]\,d\sigma}\left(\overline{u}_{L}\right)\,ds\nonumber \\
=\frac{1}{\tau\left(\tau+2E\left[z\left(\infty\right)\right]\right)}\phi_{2}\left(\tau\right)\int_{0}^{\infty}e^{is\tau}\left(\tau\left(\tau+2E\left[z\left(s\right)\right]\right)\right)R\left(\tau-i\epsilon\right)sa_{2}e^{2i\int_{0}^{s}E\left[z(\sigma)\right]\,d\sigma}\left(\overline{u}_{L}\right)\,ds\label{eq:IBPHL-1}\\+
\frac{1}{\tau\left(\tau+2E\left[z\left(\infty\right)\right]\right)}\phi_{2}\left(\tau\right)\int_{0}^{\infty}e^{is\tau}\left(\tau\left(2E\left[\left(\infty\right)\right]-2E\left[z\left(s\right)\right]\right)\right)R\left(\tau-i\epsilon\right)sa_{2}e^{2i\int_{0}^{s}E\left[z(\sigma)\right]\,d\sigma}\left(\overline{u}_{L}\right)\,ds.\nonumber 
\end{align}}
Note that in the support of $\phi_{2}$, $\frac{1}{\tau\left(\tau+2E\left[z\left(\infty\right)\right]\right)}$
is bounded. The last term on the RHS of \eqref{eq:IBPHL-1} can be estimated
by the same way as for \eqref{eq:h12L1}.

It remains to analyze
\[
\tau\phi_{2}\left(\tau\right)\int_{0}^{\infty}e^{is\tau}\left(\tau\left(\tau+2E\left[z\left(s\right)\right]\right)\right)R\left(\tau-i\epsilon\right)sa_{2}e^{2i\int_{0}^{s}E\left[z(\sigma)\right]\,d\sigma}\left(\overline{u}_{L}\right)\,ds.
\]
We again integrate by parts in $s$ as \eqref{eq:LF1exp3L1}{\small
\begin{align}
\tau\phi_{2}\left(\tau\right)\int_{0}^{\infty}\left(\tau+2E\left[z\left(s\right)\right]\right)e^{-is\tau}R\left(\tau-i\epsilon\right)sa_{2}\left(e^{2i\int_{0}^{s}E\left[z(\sigma)\right]\,d\sigma}\overline{u}_{L}\right)\,ds\nonumber \\
=\tau\phi_{2}\left(\tau\right)\int_{0}^{\infty}\partial_{s}\left(e^{-is\tau}e^{2i\int_{0}^{s}E\left[z(\sigma)\right]\,d\sigma}\right)R\left(\tau-i\epsilon\right)sa_{2}\left(\overline{u}_{L}\right)\,ds\nonumber \\
=-\tau\phi_{2}\left(\tau\right)\int_{0}^{\infty}e^{-is\tau}e^{2i\int_{0}^{s}E\left[z(\sigma)\right]\,d\sigma}R\left(\tau-i\epsilon\right)a_{2}\left(\overline{u}_{L}\right)\,ds\label{eq:LF1exp3L1-1}\\
-\tau\phi_{2}\left(\tau\right)\int_{0}^{\infty}e^{-is\tau}e^{2i\int_{0}^{s}E\left[z(\sigma)\right]\,d\sigma}R\left(\tau-i\epsilon\right)sa_{2}\left(\partial_{s}\overline{u}_{L}\right)\,ds.\nonumber 
\end{align}}
The remaining steps are identical as \eqref{eq:LF1exp3L11} and \eqref{eq:LF1exp3L12}.
Therefore, we can bound
\begin{align}
\lim_{\epsilon\rightarrow0^{+}}\left\Vert \left\langle x\right\rangle ^{-2}\phi_{2}\left(\tau\right)\partial_{\tau}\left(\phi_{2}\left(\tau\right)\left(\int_{0}^{\infty}e^{-is\tau}R\left(\tau-i\epsilon\right)a_{2}\left(e^{2i\int_{0}^{s}E\left[z(\sigma)\right]\,d\sigma}\overline{u}\right)\,ds\right)\right)\right\Vert _{L_{x}^{\infty}L_{\tau}^{2}}\nonumber \\
\lesssim\epsilon^{2}\left\Vert \left\langle x\right\rangle ^{-2}u_{L}\left(t\right)\right\Vert _{L_{x}^{\infty}L_{t}^{2}}+\epsilon^{2}\left\Vert \left\langle x\right\rangle ^{-2}\partial_{t}u_{L}\left(t\right)\right\Vert _{L_{x}^{\infty}L_{t}^{2}}\nonumber \\
+\epsilon^{2}\left\Vert \left\langle x\right\rangle ^{-2}s\left(-2iE\left[z(s)\right]\partial_{t}+\partial_{t}^{2}\right)u_{L}\left(t\right)\right\Vert _{L_{x}^{\infty}L_{t}^{2}}\lesssim\epsilon^{3}T^{\alpha}.\label{eq:HL-1}
\end{align}
Overall as \eqref{eq:addboot}, we conclude that
\begin{equation}
\left\Vert \left\langle x\right\rangle ^{-2}t\left(-2iE\left[z(t)\right]\partial_{t}+\partial_{t}^{2}\right)\left(h_{1,2}\right)_{L}\right\Vert _{L_{x}^{\infty}L_{t}^{2}\left[0,T\right]}+\left\Vert \left\langle x\right\rangle ^{-2}t\left(h_{1,2}\right)_{H}\right\Vert _{L_{x}^{\infty}L_{t}^{2}\left[0,T\right]}\lesssim\epsilon^{3}T^{\alpha}\label{eq:addboot-1}
\end{equation}
which also recasts the bootstrap conditions.


We recall that the following standard resolvent estimates. Also see Agmon \cite{Agm}.
\begin{lem}
\label{lem:limitingab}Suppose $H=-\partial_{xx}+V$ has no resonances nor eigenvalues.
Then we have
\[
\sup_{\tau}\left\Vert \left\langle x\right\rangle ^{-\alpha}R\left(\tau\right)\left\langle x\right\rangle ^{-\alpha}\right\Vert _{L_{\tau}^{\infty}}
\]
as a map from $L_{x}^{\infty}$ to $L_{x}^{\infty}$ for  $\alpha>1$.

Moreover, for the low frequency part, one has
\[
\sup_{\tau}\left\Vert \left\langle x\right\rangle ^{-\alpha-1}\phi_{1}\left(\tau\right)\tau\partial_{\tau}R\left(\tau\right)\left\langle x\right\rangle ^{-\alpha-1}\right\Vert _{L_{\tau}^{\infty}}
\]
as a map from  $L_{x}^{\infty}$ to $L_{x}^{\infty}$  for  $\alpha>1$.

For the high frequency part, one have
\[
\sup_{\tau}\left\Vert \left\langle x\right\rangle ^{-\alpha-1}\phi_{2}\left(\tau\right)\partial_{\tau}R\left(\tau\right)\left\langle x\right\rangle ^{-\alpha-1}\right\Vert _{L_{\tau}^{\infty}}
\]
as a map from  $L_{x}^{\infty}$ to $L_{x}^{\infty}$  for  $\alpha>1$. 
\end{lem}

\begin{proof}
To show these bounds, one can use the explicit formulae for resolvents.
Using the Jost functions, we can write down the Green's function,
i.e., the kernel of the resolvents $R\left(\tau\pm0i\right)$ as 
\begin{equation}
G_{\pm}\left(x,y,k\right)=\begin{cases}
-\frac{f_{+}\left(x,\pm k\right)f_{-}\left(y,\pm k\right)}{W\left(\pm k\right)} & x>y\\
-\frac{f_{-}\left(x,\pm k\right)f_{+}\left(y,\pm k\right)}{W\left(\pm k\right)} & x<y
\end{cases}\label{eq:Green-1}
\end{equation}
for $\tau\geq0$ where $k=\sqrt{\tau}$ and
\begin{equation}
G_{\pm}\left(x,y,k\right)=\begin{cases}
-\frac{f_{+}\left(x,ik\right)f_{-}\left(y,ik\right)}{W\left(ik\right)} & x>y\\
-\frac{f_{-}\left(x,ik\right)f_{+}\left(y,ik\right)}{W\left(ik\right)} & x<y
\end{cases}\label{eq:Green-1-1}
\end{equation}
for $\tau<0$ and $k=\sqrt{-\tau}$. 

Then the $L_{\tau}^{\infty}$ norm follows directly by differentiating
the expressions above. For $\tau$ large, the bound is clear. When
$\tau$ is small, $\partial_{\tau}R\left(\tau\pm0i\right)\sim\frac{1}{\sqrt{\tau}}$
which will be canceled out by the factor $\phi_{1}\left(\lambda\right)\lambda$
and the desired bounds follow.
\end{proof}
\begin{rem}
When $-\partial_{x}^{2}+V$ has a negative eigenvalue $-\rho^{2}$,
the only difference is that $W(ik)$ has a simple pole at $k=\rho$
and $W\left(ik\right)\neq0$ otherwise. The one has the following
estimates
\[
\sup_{\tau<0}\left(\left|\left\langle x\right\rangle ^{-\alpha}R\left(\tau\right)P_{c}\left\langle x\right\rangle ^{-\alpha}\right|\right)
\]
\[
\sup_{\tau<0}\left(\left|\left\langle x\right\rangle ^{-\alpha-1}\phi_{1}\left(\tau\right)\tau\partial_{\tau}R\left(\tau\right)P_{c}\left\langle x\right\rangle ^{-\alpha-1}\right|\right)
\]
and
\[
\sup_{\tau<0}\left(\left\langle x\right\rangle ^{-\alpha-1}\phi_{2}\left(\tau\right)\partial_{\tau}R\left(\tau\right)P_{c}\left\langle x\right\rangle ^{-\alpha-1}\right).
\]
\end{rem}

\subsubsection{Weighted estimates}\label{subsubsec:firstorderweight}

Finally, we consider the weighted estimates for
\[
h_{1,1}(t,x)=\int_{1}^{t}e^{i(t-s)H}N_{1,1}\left(s\right)\,ds=\int_{1}^{t}e^{i(t-s)H}a_{1}u(s)\,ds
\]
\[
h_{1,2}(t,x)=\int_{1}^{t}e^{i(t-s)H}N_{1,1}\left(s\right)\,ds=\int_{1}^{t}e^{i(t-s)H}a_{2}e^{2i\int_{0}^{s}E\left[z(\sigma)\right]\,d\sigma}\left(\overline{u}\right)\,ds
\]
In terms of profile, $h_{a,j}=e^{-itH}h_{1,j}\left(t,x\right),$ we
need to estimate for $j=1,2$
\begin{align}
\partial_{k}\tilde{h}_{a,j}\left(t,k\right) & =-\int_{1}^{t}e^{-isk^{2}}2isk\iint\tilde{a}_{j}\left(\ell\right)\tilde{\mathrm{u}}_{j}\left(n\right)\nu\left(k,\ell,n\right)\,d\ell dn\nonumber \\
 & +\int_{1}^{t}e^{-isk^{2}}\iint\tilde{a}_{j}\left(\ell\right)\tilde{\mathrm{u}}_{j}\left(n\right)\partial_{k}\nu\left(k,\ell,n\right)\,d\ell dn\label{eq:kha}
\end{align}
where
\[
\mathrm{u}_{1}=u,\ \mathrm{u}_{2}=e^{2i\int_{0}^{s}E\left[z(\sigma)\right]\,d\sigma}\left(\overline{u}\right)
\]
and
\begin{equation}
\nu\left(k,\ell,n\right)=\int\overline{\mathcal{K}}\left(x,k\right)\mathcal{K}\left(x,\ell\right)\mathcal{K}\left(x,n\right)\,dx.\label{eq:nu3}
\end{equation}
We first estimate the second term on the RHS of \eqref{eq:kha}. This
term is relatively easy since there is no growth in $s$ here. The
analysis for $j=1$ and $j=2$ are the same. We only presnet the case
that $j=1$ here. The analysis is quite similar to the analysis for
\eqref{eq:inhomq2}.

Applying Plancherel's theorem and the smoothing estimate,{\footnotesize
\begin{align*}
\left\Vert \int_{1}^{t}e^{-isk^{2}}\iint\tilde{a}_{1}\left(\ell\right)\tilde{u}\left(n\right)\partial_{k}\nu\left(k,\ell,n\right)\,d\ell dn\right\Vert _{L_{k}^{2}}
\lesssim\left\Vert \tilde{\mathcal{F}}^{-1}\int_{1}^{t}e^{-isk^{2}}\iint\tilde{a}_{1}\left(\ell\right)\tilde{u}\left(n\right)\partial_{k}\nu\left(k,\ell,n\right)\,d\ell dn\right\Vert _{L_{x}^{2}}\\
\lesssim\left\Vert \left\langle x\right\rangle ^{\frac{5}{2}}\tilde{\mathcal{F}}^{-1}\left[\iint\tilde{a}_{1}\left(\ell\right)\tilde{u}\left(n\right)\partial_{k}\nu\left(k,\ell,n\right)\,d\ell dn\right]\right\Vert _{L_{t}^{2}L_{x}^{2}}.
\end{align*}}
The same  as before. we first consider the estimates without weights:
\[
\left\Vert \tilde{\mathcal{F}}^{-1}\left[\iint\tilde{a}_{1}\left(\ell\right)\tilde{u}\left(n\right)\partial_{k}\nu\left(k,\ell,n\right)\,d\ell dn\right]\right\Vert _{L^{2}}.
\]
Applying Plancherel's theorem, we estimate the $L^{2}$ norm of
\[
\iint\tilde{a}_{1}\left(\ell\right)\tilde{u}\left(n\right)\partial_{k}\nu\left(k,\ell,n\right)\,d\ell dn.
\]
Again, we focus on the case that for $k\geq0$ and perform integrate
by parts{\small
\begin{align}
\partial_{k}\nu\left(k,\ell,n\right) & =\int\partial_{k}\overline{\mathcal{K}}\left(x,k\right)\mathcal{K}\left(x,\ell\right)\mathcal{K}\left(x,n\right)\,dx\nonumber \\
 & =\int-ixe^{-ikx}T\left(k\right)\overline{m}_{+}\left(x,k\right)\mathcal{K}\left(x,\ell\right)\mathcal{K}\left(x,n\right)\,dx\nonumber \\
 & +\int e^{-ikx}\partial_{k}T\left(k\right)\overline{m}_{+}\left(x,k\right)\mathcal{K}\left(x,\ell\right)\mathcal{K}\left(x,n\right)\,dx\nonumber \\
 & +\int e^{-ikx}T\left(k\right)\partial_{k}\overline{m}_{+}\left(x,k\right)\mathcal{K}\left(x,\ell\right)\mathcal{K}\left(x,n\right)\,dx.\label{eq:partialkmu-1}
\end{align}}
Note that again by explicit computations, one has{\small
\begin{align*}
\iiint\tilde{a}\left(\ell\right)\tilde{u}\left(n\right)ixe^{-ikx}T\left(k\right)\overline{m}_{+}\left(x,k\right)\mathcal{K}\left(x,\ell\right)\mathcal{K}\left(x,n\right)\,dxd\ell dn
=\int\overline{\mathcal{K}}\left(x,k\right)ixa\left(x\right)u\,dx
\end{align*}}
and
\begin{align*}
\left\Vert \int\overline{\mathcal{K}}\left(x,k\right)ixa_{1}\left(x\right)uixe^{-ikx}\,dx\right\Vert _{L^{2}} & \lesssim\epsilon^{2}\left\Vert \left\langle x\right\rangle ^{-2}u\right\Vert _{L^{\infty}}\lesssim\epsilon^{2}s^{-1+\alpha}\left\Vert u\right\Vert _{X_{T}}
\end{align*}
where again, we applied the decay estimate from Corollary \ref{cor:directXT}.

Similarly, we also have
\begin{align*}
\iiint\tilde{a}_{1}\left(\ell\right)\tilde{u}\left(n\right)e^{-ikx}\partial_{k}T\left(k\right)\overline{m}_{+}\left(x,k\right)\mathcal{K}\left(x,\ell\right)\mathcal{K}\left(x,n\right)\,dxd\ell dn\\
=\int e^{-ikx}\partial_{k}T\left(k\right)\overline{m}_{+}\left(x,k\right)a_{1}\left(x\right)u\,dx,
\end{align*}
whence it follows
\begin{align*}
\left\Vert \int e^{-ikx}\partial_{k}T\left(k\right)\overline{m}_{+}\left(x,k\right)a_{1}\left(x\right)u\,dx\right\Vert _{L^{2}} & \lesssim\epsilon^{2}\left\Vert \left\langle x\right\rangle ^{-2}u\right\Vert _{L^{\infty}}\lesssim\epsilon^{2}s^{-1+\alpha}\left\Vert u\right\Vert _{X_{T}}.
\end{align*}
For the last piece from \eqref{eq:partialkmu-1}, one has
\begin{align*}
\iiiint\tilde{a}_{1}\left(\ell\right)\tilde{u}\left(n\right)e^{-ikx}T\left(k\right)\partial_{k}\overline{m}_{+}\left(x,k\right)\mathcal{K}\left(x,\ell\right)\mathcal{K}\left(x,n\right)\,dxd\ell dn\\
=\int e^{-ikx}T\left(k\right)\partial_{k}\overline{m}_{+}\left(x,k\right)a\left(x\right)u\,dx.
\end{align*}
Then apply the boundedness of the pseudo-differential operator, we
conclude that
\begin{align*}
\left\Vert \int e^{-ikx}T\left(k\right)\partial_{k}\overline{m}_{+}\left(x,k\right)a_{1}\left(x\right)uixe^{-ikx}\,dx\right\Vert _{L^{2}}
\lesssim\left\Vert \left\langle x\right\rangle ^{2}a_{1}\left(x\right)u\right\Vert _{L^{2}}
\lesssim\epsilon^{2}s^{-1+\alpha}\left\Vert u\right\Vert _{X_{T}}.
\end{align*}
For the weighted version, we again take the flat Fourier transform
and use the boundedness of the wave operator as \eqref{eq:weightedWO},
it follows that
\[
\left\Vert \left\langle x\right\rangle ^{\frac{5}{2}}\tilde{\mathcal{F}}^{-1}\left[\iint\tilde{a}_{1}\left(\ell\right)\tilde{u}\left(n\right)\partial_{k}\nu\left(k,\ell,n\right)\,d\ell dn\right]\right\Vert _{L_{x}^{2}.}\lesssim\epsilon^{2}s^{-1+\alpha}\left\Vert u\right\Vert _{X_{T}}.
\]
Performing the $L_{t}^{2}$ integration and summing up the above pieces,
one has
\begin{align}
\left\Vert \tilde{F}^{-1}\left[\iint\tilde{a}_{1}\left(\ell\right)\tilde{u}\left(n\right)\partial_{k}\nu\left(k,\ell,n\right)\,d\ell dn\right]\right\Vert _{L^{2}} & \lesssim\epsilon^{2}\left(\int_{1}^{t}\left|s^{-1+\alpha}\left\Vert u\right\Vert _{X_{T}}\right|^{2}\,ds\right)^{\frac{1}{2}}\nonumber \\
 & \lesssim\epsilon^{3}\left(\left\langle t\right\rangle ^{-\frac{1}{2}+2\alpha}+1\right).\label{eq:quadbdF-2}
\end{align}
Finally, we analyze the first term on the RHS of \eqref{eq:kha}
\begin{equation}
\int_{1}^{t}e^{-isk^{2}}2isk\iint\tilde{a}_{j}\left(\ell\right)\tilde{\mathrm{u}}_{j}\left(n\right)\nu\left(k,\ell,n\right)\,d\ell dn.\label{eq:hakfirst}
\end{equation}
We decompose $\tilde{\mathrm{u}}_{j}$ into high and low frequency
parts:{\small
\begin{align*}
\int_{0}^{t}\iint e^{-ik^{2}s}ks\left(\tilde{a}_{j}\left(\ell\right)\tilde{\mathrm{u}}_{j}\left(n\right)\right)\,\nu\left(k,\ell,n\right)d\ell dnds=\\
\int_{0}^{t}\iint e^{-ik^{2}s}ks\left(\tilde{a}_{j}\left(\ell\right)\tilde{\mathrm{u}}_{j,L}\left(n\right)\right)\,\nu\left(k,\ell,n\right)d\ell dnds\\
+\int_{0}^{t}\iint e^{-ik^{2}s}ks\left(\tilde{a}\left(\ell\right)\tilde{\mathrm{u}}_{j,H}\left(n\right)\right)\,\nu\left(k,\ell,n\right)d\ell dnds.
\end{align*}}
where
\[
\mathrm{u}_{1,A}=u_{A},\ \ \ \mathrm{u}_{2,A}=e^{2i\int_{0}^{s}E\left[z(\sigma)\right]\,d\sigma}\left(\overline{u}_{A}\right),\,\ A\in\left\{ L,H\right\} .
\]
The ways to deal with the high frequency parts are the same for both
$j=1,2$. Again, we only present $j=1$:
\begin{equation}
\int_{0}^{t}\int e^{-ik^{2}s}ks\left(\tilde{a}_{1}\left(\ell\right)\tilde{u}_{H}\left(n\right)\right)\,\nu\left(k,\ell,n\right)d\ell dnds.\label{eq:hakfirstH-1}
\end{equation}
We perform direct estimates
\begin{align*}
\left\Vert \int_{0}^{t}\int e^{-ik^{2}s}ks\left(\tilde{a}_{1}\left(\ell\right)\tilde{u}_{H}\left(n\right)\right)\,\nu\left(k,\ell,n\right)d\ell dnds\right\Vert _{L_{k}^{2}}\\
\lesssim\left\Vert \left\langle x\right\rangle ^{3}sa_{1}u_{H}\right\Vert _{L_{x}^{2}L_{t}^{2}}+\left\Vert \left\langle x\right\rangle ^{3}sa_{1}\partial_{x}u_{H}\right\Vert _{L_{x}^{2}L_{t}^{2}}\\
\lesssim\epsilon^{2}\left(\left\Vert \left\langle x\right\rangle ^{-2}s\partial_{x}u_{H}\right\Vert _{L_{x}^{2}L_{t}^{2}}+\left\Vert \left\langle x\right\rangle ^{-2}su_{H}\right\Vert _{L_{x}^{2}L_{t}^{2}}\right)
\lesssim\epsilon^{3}T^{\alpha}
\end{align*}
where in the last line, we used our auxiliary estimate \eqref{eq:boot4}.

The analysis for the low frequency parts are a little bit more involved.

We start with the analysis of $N_{1,1,L}=a_{1}u_{L}$. First of all,
integration by parts in $s$ using 
$
\frac{1}{-ik^{2}}\partial_{s}\left(e^{-ik^{2}s}\right)=e^{-ik^{2}s}
$
gives us
\begin{align}
\int_{1}^{t}ikse^{-ik^{2}s}\tilde{N}_{1,1,L}\left(s,k\right)\,ds & \sim\int_{1}^{t}\frac{1}{k}se^{-ik^{2}s}\partial_{s}\tilde{N}_{1,1,L}\left(s,k\right)\,ds\nonumber \\
 & +\int_{1}^{t}\frac{1}{k}e^{-ik^{2}s}\tilde{N}_{1,1,L}\left(s,k\right)\,ds\label{eq:h11IBP1}\\
 & +\frac{1}{k}te^{-ik^{2}t}\tilde{N}_{1,1,L}\left(t,k\right)+\frac{1}{k}e^{-ik^{2}}\tilde{N}_{1,1,L}\left(1,k\right).\nonumber 
\end{align}
For the boundary terms, due to the generic condition of the potential,
$\frac{1}{k}$ here can be canceled by the transform with frequency
$k$ in the measure. Applying Lemma \ref{lem:denoPSO},
\begin{align}
\left\Vert \iint e^{-ik^{2}t}\frac{1}{-ik}t\left(\tilde{a}_{1}\left(\ell\right)\tilde{u}_{L}\left(n\right)\right)\,\nu\left(k,\ell,n\right)d\ell dn\right\Vert _{L_{k}^{2}}\nonumber \\
\lesssim\left\Vert \left\langle x\right\rangle a_{1}(x)tu_{L}\left(t,x\right)\right\Vert _{L_{x}^{2}}\lesssim\epsilon^{3}T^{\alpha}.\label{eq:IBPLB}
\end{align}
Similarly, one has
\begin{align}
\left\Vert \iint e^{-ik^{2}}\frac{1}{-ik}\left(\tilde{a}_{1}\left(\ell\right)\tilde{u}_{L}\left(n\right)\right)\,\nu\left(k,\ell,n\right)d\ell dn\right\Vert _{L_{k}^{2}}\nonumber \\
\lesssim\left\Vert \left\langle x\right\rangle a_{1}(x)u_{L}\left(1,x\right)\right\Vert _{L_{x}^{2}}\lesssim\epsilon^{3}.\label{eq:IBPLB-1}
\end{align}
The second term on the RHS of \eqref{eq:h11IBP1} can be bounded using
the regular inhomogeneous smoothing estimate \eqref{eq:smoothing2} and the $L^{2}$ bound in Lemma \ref{lem:denoPSO},
\begin{align}
\left\Vert \int_{1}^{t}e^{-iHs}\frac{1}{\sqrt{H}}\left(au_{L}\right)\,ds\right\Vert _{L_{x}^{2}} & \lesssim\left\Vert \left\langle x\right\rangle ^{^{3}}\frac{1}{\sqrt{H}}\left(a_1u_{L}\right)\right\Vert _{L_{x}^{2}L_{t}^{2}}\lesssim\left\Vert \left\langle x\right\rangle ^{4}a_1u_{L}\right\Vert _{L_{x}^{2}L_{t}^{2}}\lesssim\epsilon^{3}\label{eq:h11ibpbulk1}
\end{align}
after switching back to the physical space using Plancherel's theorem.

It remains to analyze the first term on the RHS of \eqref{eq:h11IBP1}.
We write
\begin{align}
\int_{1}^{t}\frac{1}{k}se^{-ik^{2}s}\partial_{s}\tilde{N}_{1,1}\left(s,k\right)\,ds & =\int_{1}^{t}\frac{1}{k}se^{-ik^{2}s}e^{2i\int_{0}^{s}E\left[z(\sigma)\right]\,d\sigma}\left(e^{-2i\int_{0}^{s}E\left[z(\sigma)\right]\,d\sigma}\partial_{s}\tilde{N}_{1,1}\left(s,k\right)\right)\,ds\nonumber \\
 & =\int_{1}^{t}\frac{1}{k}s\frac{1}{-ik^{2}+2iE\left[z(s)\right]}\partial_{s}\left(e^{-ik^{2}s}e^{2i\int_{0}^{s}E\left[z(\sigma)\right]\,d\sigma}\right)\left(\mathcal{N}_{1,1}\left(s,k\right)\right)\,ds\label{eq:h11IBP2}
\end{align}
where to save to spaces, we introduced the notation
\[
\mathcal{N}_{1,1}\left(s,k\right):=e^{-2i\int_{0}^{s}E\left[z(\sigma)\right]\,d\sigma}\partial_{s}\tilde{N}_{1,1,L}\left(s,k\right).
\]
We note that $k^{2}-2E\left[z(s)\right]\geq\rho^{2}$ due to the estimate
of the parameter $E\left[z(s)\right]$ so the denominator $\frac{1}{-ik^{2}+2iE\left[z(s)\right]}$
appearing in the expression \eqref{eq:h11IBP2} is harmless.

Integrating by parts in $s$, one has{\small
\begin{align}
\int_{1}^{t}\frac{1}{k}s\frac{1}{-ik^{2}+2iE\left[z(s)\right]}\partial_{s}\left(e^{-ik^{2}s}e^{2i\int_{0}^{s}E\left[z(\sigma)\right]\,d\sigma}\right)\left(\mathcal{N}_{1,1}\left(s,k\right)\right)\,ds\nonumber \\
=-\int_{1}^{t}\frac{1}{k}s\frac{1}{-ik^{2}+2iE\left[z(s)\right]}\left(e^{-ik^{2}s}e^{2i\int_{0}^{s}E\left[z(\sigma)\right]\,d\sigma}\right)\partial_{s}\left(\mathcal{N}_{1,1}\left(s,k\right)\right)\,ds\nonumber \\
-\int_{1}^{t}\frac{1}{k}\frac{1}{-ik^{2}+2iE\left[z(s)\right]}\left(e^{-ik^{2}s}e^{2i\int_{0}^{s}E\left[z(\sigma)\right]\,d\sigma}\right)\left(\mathcal{N}_{1,1}\left(s,k\right)\right)\,ds\label{eq:h11IBP2exp}\\
\int_{1}^{t}s\frac{1}{k}\frac{E'\left[z(s)\right]\frac{d}{ds}\left|z\left(s\right)\right|}{\left(-ik^{2}+2iE\left[z(s)\right]\right)^{2}}\left(e^{-ik^{2}s}e^{2i\int_{0}^{s}E\left[z(\sigma)\right]\,d\sigma}\right)\left(\mathcal{N}_{1,1}\left(s,k\right)\right)\,ds\nonumber \\
+\frac{1}{k}t\frac{1}{-ik^{2}+2iE\left[z(t)\right]}\left(e^{-ik^{2}t}e^{2i\int_{0}^{t}E\left[z(\sigma)\right]\,d\sigma}\right)\left(\mathcal{N}_{1,1}\left(t,k\right)\right)\nonumber \\
+\frac{1}{k}\frac{1}{-ik^{2}+2iE\left[z(1)\right]}\left(e^{-ik^{2}}e^{2i\int_{0}^{1}E\left[z(\sigma)\right]\,d\sigma}\right)\left(\mathcal{N}_{1,1}\left(1,k\right)\right).\nonumber 
\end{align}}
The last two boundary terms can be bounded as \eqref{eq:IBPLB} and
\eqref{eq:IBPLB-1}. The second term on the RHS of \eqref{eq:h11IBP2}
can be estimated as \eqref{eq:h11ibpbulk1}:
\begin{align}
\left\Vert \int_{1}^{t}\frac{1}{k}\frac{1}{-ik^{2}+2iE\left[z(s)\right]}\left(e^{-ik^{2}s}e^{2i\int_{0}^{s}E\left[z(\sigma)\right]\,d\sigma}\right)\left(\mathcal{N}_{1,1}\left(s,k\right)\right)\,ds\right\Vert _{L_{x}^{2}}\label{eq:h11ibpbulk21}\\
\lesssim\left\Vert \left\langle x\right\rangle ^{^{3}}\frac{1}{\sqrt{H}}a_{1}\partial_{t}u_{L}\right\Vert _{L_{x}^{2}L_{t}^{2}}\nonumber 
\lesssim\left\Vert \left\langle x\right\rangle ^{4}a_{1}u_{L}\right\Vert _{L_{x}^{2}L_{t}^{2}}\lesssim\epsilon^{3} & .\nonumber 
\end{align}
Using the decay of $\frac{d}{ds}\left|z\left(s\right)\right|\sim s^{-2+2\alpha}$,
the third term on the RHS of \eqref{eq:h11IBP2} can be estimated as
\eqref{eq:h11ibpbulk21}.

Finally, as before, the most difficult part is given by
\begin{align*}
\int_{1}^{t}\frac{1}{k}s\frac{1}{-ik^{2}+2iE\left[z(s)\right]}\left(e^{-ik^{2}s}e^{2i\int_{0}^{s}E\left[z(\sigma)\right]\,d\sigma}\right)\partial_{s}\left(\mathcal{N}_{1,1}\left(s,k\right)\right)\,ds\\
=\int_{1}^{t}\frac{1}{k}s\frac{1}{-ik^{2}+2iE\left[z(s)\right]}\left(-2iE\left[z(s)\right]\partial_{s}+\partial_{s}^{2}\right)\tilde{N}_{1,1}\left(s,k\right)\,ds.
\end{align*}
As above, going back to the physical space and using the generic condition
to cancel the singularity $\frac{1}{k}$, we have
\begin{align}
\left\Vert \int_{1}^{t}\frac{1}{k}s\frac{1}{-ik^{2}+2iE\left[z(s)\right]}\left(-2iE\left[z(s)\right]\partial_{s}+\partial_{s}^{2}\right)\tilde{N}_{1,1,L}\left(s,k\right)\,ds\right\Vert _{L_{x}^{2}}\label{eq:h11ibpbulk22}\\
\lesssim\left\Vert \left\langle x\right\rangle ^{^{3}}\frac{1}{\sqrt{H}}sa_{1}\left(-2iE\left[z(s)\right]\partial_{s}+\partial_{s}^{2}\right)u_{L}\right\Vert _{L_{x}^{2}L_{t}^{2}}\nonumber \\
\lesssim\left\Vert \left\langle x\right\rangle ^{4}a_{1}\left(-2iE\left[z(s)\right]\partial_{s}+\partial_{s}^{2}\right)u_{L}\right\Vert _{L_{x}^{2}L_{t}^{2}}\nonumber
\lesssim\epsilon^{3}T^{\alpha} 
\end{align}
where in the last line we applied the auxiliary estimate  \eqref{eq:boot3}.

Therefore, we conclude that
\[
\sup_{t\in[0,T]}\left\Vert \int_{1}^{t}ikse^{-ik^{2}s}\tilde{N}_{1,1,L}\left(s,k\right)\,ds\right\Vert _{L_{k}^{2}}\lesssim\epsilon^{3}T^{\alpha}.
\]
It remains to bound
\[
\int_{0}^{t}\iint e^{-ik^{2}s}kse^{2i\int_{0}^{s}E\left[z(\sigma)\right]\,d\sigma}\left(\tilde{N}_{1,2,L}\right)\,ds
\]
where $N_{1,2,L}=a_{2}\bar{u}_{L}$. As before, we note that
\[
e^{-ik^{2}s}e^{2i\int_{0}^{s}E\left[z(\sigma)\right]\,d\sigma}=\frac{1}{-ik^{2}+2iE\left[z(s)\right]}\partial_{s}\left(e^{-ik^{2}s}e^{2i\int_{0}^{s}E\left[z(\sigma)\right]\,d\sigma}\right)
\]
where the denominator is strictly negative. Integrating by parts in
$s$ one has{\small
\begin{align}
-2i\int_{1}^{t}kse^{-ik^{2}s+2i\int_{0}^{s}E\left[z(\sigma)\right]\,d\sigma}\left(\tilde{N}_{1,2,L}\left(s,k\right)\right)\,ds\nonumber \\
=-2i\int_{1}^{t}ks\frac{1}{-ik^{2}+2iE\left[z(s)\right]}\partial_{s}\left(e^{-ik^{2}s}e^{2i\int_{0}^{s}E\left[z(\sigma)\right]\,d\sigma}\right)\left(\tilde{N}_{1,2,L}\left(s,k\right)\right)\,ds\label{eq:h12IBP1exp}\\
=-2ik\frac{1}{-ik^{2}+2iE\left[z(1)\right]}e^{-ik^{2}}e^{2i\int_{0}^{1}E\left[z(\sigma)\right]\,d\sigma}\left(\tilde{N}_{1,2,L}\left(1,k\right)\right)\nonumber \\
-2ik\frac{1}{-ik^{2}+2iE\left[z(t)\right]}e^{-ik^{2}t}e^{2i\int_{0}^{t}E\left[z(\sigma)\right]\,d\sigma}\left(\tilde{N}_{1,2,L}\left(t,k\right)\right)\nonumber \\
+2i\int_{1}^{t}k\frac{1}{-ik^{2}+2iE\left[z(s)\right]}\left(e^{-ik^{2}s}e^{2i\int_{0}^{s}E\left[z(\sigma)\right]\,d\sigma}\right)\left(\tilde{N}_{1,2,L}\left(s,k\right)\right)\,ds\nonumber \\
-2i\int_{1}^{t}ks\frac{2iE'\left[z(s)\right]\partial_{s}\left|z(s)\right|}{\left(-ik^{2}+2iE\left[z(s)\right]\right)^{2}}\left(e^{-ik^{2}s}e^{2i\int_{0}^{s}E\left[z(\sigma)\right]\,d\sigma}\right)\left(\tilde{N}_{1,2,L}\left(s,k\right)\right)\,ds\nonumber \\
+2i\int_{1}^{t}sk\frac{1}{-ik^{2}+2iE\left[z(s)\right]}\left(e^{-ik^{2}s}e^{2i\int_{0}^{s}E\left[z(\sigma)\right]\,d\sigma}\right)\partial_{s}\left(\tilde{N}_{1,2,L}\left(s,k\right)\right)\,ds.\nonumber 
\end{align}}
Both boundary terms appearing above can be estimated in the same manner
as \eqref{eq:IBPLB} and \eqref{eq:IBPLB-1}.

The third term on the RHS of \eqref{eq:h12IBP1exp} can be bouneded
by the same way as \eqref{eq:h11ibpbulk21} and the fourth term can
be estimated similarly using the decay of $\partial_{s}\left|z(s)\right|\sim s^{-2+2\alpha}$.

The most difficult term is again the last term above
\[
2i\int_{1}^{t}sk\frac{1}{-ik^{2}+2iE\left[z(s)\right]}\left(e^{-ik^{2}s}e^{2i\int_{0}^{s}E\left[z(\sigma)\right]\,d\sigma}\right)\partial_{s}\left(\tilde{N}_{1,2,L}\left(s,k\right)\right)\,ds.
\]
Performing integration by parts in $s$ using the identity $e^{-ik^{2}s}=\frac{1}{-ik^{2}}\partial_{s}\left(e^{-ik^{2}s}\right)$,
one has{\small
\begin{align}
2i\int_{1}^{t}sk\frac{1}{-ik^{2}+2iE\left[z(s)\right]}\frac{1}{-ik^{2}}\partial_{s}\left(e^{-ik^{2}s}\right)e^{2i\int_{0}^{s}E\left[z(\sigma)\right]\,d\sigma}\partial_{s}\left(\tilde{N}_{1,2,L}\left(s,k\right)\right)\,ds\nonumber \\
=-2\int_{1}^{t}s\frac{1}{-ik^{2}+2iE\left[z(s)\right]}\frac{1}{-k}e^{-ik^{2}s}\partial_{s}\left(e^{2i\int_{0}^{s}E\left[z(\sigma)\right]\,d\sigma}\partial_{s}\left(\tilde{N}_{1,2,L}\left(s,k\right)\right)\right)\,ds\nonumber \\
-2\int_{1}^{t}\frac{1}{-ik^{2}+2iE\left[z(s)\right]}\frac{1}{-k}e^{-ik^{2}s}e^{2i\int_{0}^{s}E\left[z(\sigma)\right]\,d\sigma}\partial_{s}\left(\tilde{N}_{1,2,L}\left(s,k\right)\right)\,ds\label{eq:h12IBP2exp}\\
+2\int_{1}^{t}s\frac{2iE'\left[z(s)\right]\partial_{s}\left|z(s)\right|}{\left(-ik^{2}+2iE\left[z(s)\right]\right)^{2}}\frac{1}{-k}e^{-ik^{2}s}e^{2i\int_{0}^{s}E\left[z(\sigma)\right]\,d\sigma}\partial_{s}\left(\tilde{N}_{1,2,L}\left(s,k\right)\right)\,ds\nonumber \\
+2\frac{1}{-ik^{2}+2iE\left[z(1)\right]}\frac{1}{-k}e^{-ik^{2}}e^{2i\int_{0}^{1}E\left[z(\sigma)\right]\,d\sigma}\partial_{s}\left(\tilde{N}_{1,2,L}\left(1,k\right)\right)\,ds\nonumber \\
+2t\frac{1}{-ik^{2}+2iE\left[z(s)\right]}\frac{1}{-k}e^{-ik^{2}t}e^{2i\int_{0}^{t}E\left[z(\sigma)\right]\,d\sigma}\partial_{s}\left(\tilde{N}_{1,2,L}\left(t,k\right)\right)\,ds & .\nonumber 
\end{align}}By the same as discussions above, boundary terms can be bounded directly
as \eqref{eq:IBPLB} and \eqref{eq:IBPLB-1}. For the bulk terms, the
second and third terms on the RHS of \eqref{eq:h12IBP2exp} can estimated
in the same manner as for those terms appearing in the analysis of
$N_{1,1,L}$.

Finally, for the first term on the RHS of \eqref{eq:h12IBP2exp}, we
note that
\begin{align*}
-2\int_{1}^{t}s\frac{1}{-ik^{2}+2iE\left[z(s)\right]}\frac{1}{-k}e^{-ik^{2}s}\partial_{s}\left(e^{2i\int_{0}^{s}E\left[z(\sigma)\right]\,d\sigma}\partial_{s}\left(\tilde{N}_{1,2,L}\left(s,k\right)\right)\right)\,ds\\
=-2\int_{1}^{t}s\frac{1}{-ik^{2}+2iE\left[z(s)\right]}\frac{1}{-k}e^{-ik^{2}s}e^{2i\int_{0}^{s}E\left[z(\sigma)\right]\,d\sigma}\left(2iE\left[z(s)\right]\partial_{s}+\partial_{s}^{2}\right)\tilde{N}_{1,2,L}\left(s,k\right)\,ds & .
\end{align*}
Therefore, proceeding in the same manner as \eqref{eq:h11ibpbulk22},
we can conclude that
\begin{align}
\left\Vert \int_{1}^{t}\frac{1}{k}s\frac{1}{-ik^{2}+2iE\left[z(s)\right]}\left(2iE\left[z(s)\right]\partial_{s}+\partial_{s}^{2}\right)\tilde{N}_{1,2,L}\left(s,k\right)\,ds\right\Vert _{L_{x}^{2}}\label{eq:h11ibpbulk22-1}\\
\lesssim\left\Vert \left\langle x\right\rangle ^{^{3}}\frac{1}{\sqrt{H}}se^{2i\int_{0}^{s}E\left[z(\sigma)\right]\,d\sigma}a_{2}\left(2iE\left[z(s)\right]\partial_{s}+\partial_{s}^{2}\right)\bar{u}_{L}\right\Vert _{L_{x}^{2}L_{t}^{2}}\nonumber \\
\lesssim\left\Vert \left\langle x\right\rangle ^{4}\bar{a}_{2}\left(-2iE\left[z(s)\right]\partial_{s}+\partial_{s}^{2}\right)u_{L}\right\Vert _{L_{x}^{2}L_{t}^{2}}
\lesssim\epsilon^{3}T^{\alpha} & .\nonumber 
\end{align}
Therefore, we conclude that
\[
\sup_{t\in[0,T]}\left\Vert \int_{1}^{t}ikse^{-ik^{2}s}\tilde{N}_{1,2,L}\left(s,k\right)\,ds\right\Vert _{L_{k}^{2}}\lesssim\epsilon^{3}T^{\alpha}.
\]
Finally, putting all estimates above together, we obtain that
\[
\left\Vert \partial_{k}\tilde{h}_{a,j}\left(t,k\right)\right\Vert _{L_{k}^{2}}\lesssim\epsilon^{3}T^{\alpha},\ j=1,2
\]
which recast the bootstrap conditions on the weighted estimates for
the first order perturbation.

\begin{prop}\label{pro:weightmainfirst}
For $1\leq t\leq T$, one has that, for some $C>0$,
\begin{align}\label{weightmainconcfirst}
\sup_{0\leq t\leq T}{\big\| \partial_{k}\int_{1}^{t}e^{-ik^{2}s}\tilde{N}_{1}\left(s\right)\,ds \big\|}_{L_{k}^{2}}
  \leq 
  C \epsilon^2 T^{\alpha} {\big\| u\big\|}_{X_T}.
\end{align}
\end{prop}
Finally, we record a pseudo-differential operators bound to deal with the singularities when we performed integration by parts in time.

\begin{lem}
\label{lem:denoPSO}Suppose that $H=-\partial_{xx}+V$ has no zero
resonances nor zero eigenvalues. Then we have the following estimate:
\[
\left\Vert \left|x\right|^{m}\frac{1}{\sqrt{H}}G\right\Vert _{L^{2}}\lesssim\left\Vert \left\langle x\right\rangle ^{m+1}G\right\Vert _{L^{2}}.
\]
\end{lem}

\begin{proof}
Denote $g=\frac{1}{\sqrt{H}}\left(G\right)$. Then using Plancherel's
theorem, one has
\[
\left\Vert \left|x\right|^{m}\frac{1}{\sqrt{H}}\left(G\right)\right\Vert _{L_{x}^{2}}\sim\left\Vert \partial_{\ell}^{m}\mathcal{F}\circ\tilde{\mathcal{F}}^{-1}\left(\tilde{G}\right)\left(\ell\right)\right\Vert _{L_{\ell}^{2}}.
\]
Note that $\mathcal{F}\circ\tilde{\mathcal{F}}^{-1}=\mathcal{W}$
is the wave operator. By the boundedness of the wave operators in
Sobolev spaces, we have
\[
\left\Vert \partial_{\ell}^{m}\mathcal{F}\circ\tilde{\mathcal{F}}^{-1}\left(\tilde{g}\right)\left(\ell\right)\right\Vert _{L_{\ell}^{2}}\lesssim\left\Vert \tilde{g}\left(k\right)\right\Vert _{H_{k}^{m}}.
\]
Note that for $k\geq1$, by the explicit formula of distorted transform
\[
\tilde{g}\left(k\right)=\int\frac{1}{k}T\left(k\right)e^{-ikx}m\left(x,k\right)G(x)\,dx.
\]
Focusing on small frequency, we treat $\frac{1}{k}T\left(k\right)$
has a constant.

Therefore, using the boundedness of pseudo-differential operators,
\[
\left\Vert \int e^{-ikx}m\left(x,k\right)G(x)\,dx\right\Vert _{H_{k}^{m}}\lesssim\left\Vert \left\langle x\right\rangle ^{m+1}G\right\Vert _{L^{2}}
\]
the desired result follows.
\end{proof}

\subsection{Application to the full problem}\label{subsec:apptofull}
Here we give some details on the application of the arguments above to the full problem \eqref{eq:eta}.
Consider the equation for the radiation term
\begin{align*}
i\partial_{t}\eta-\partial_{xx}\eta+V\eta & =-2\left|\mathcal{Q}\left[z(\infty)\right]\right|^{2}\eta-\left(\mathcal{Q}\left[z(\infty)\right]\right)^{2}e^{2i\int_{0}^{t}E\left[z(\sigma)\right]\,d\sigma}\bar{\eta}\\
 & -2\left(\left|\mathcal{Q}\left[z\right]\right|^{2}-\left|\mathcal{Q}\left[z(\infty)\right]\right|\right)\eta\\
 & -\left(\left(\mathcal{Q}\left[z(\infty)\right]\right)^{2}-\left(\mathcal{Q}\left[z(\infty)\right]\right)^{2}\right)e^{2i\int_{0}^{t}E\left[z(\sigma)\right]\,d\sigma}\bar{\eta}\\
 & +\overline{Q\left[z\right]}\eta^{2}+2Q\left[z\right]\left|\eta\right|^{2}\\
 & +\left|\eta\right|^{2}\eta\\
 & +i\text{D}Q\left(\dot{z}-izE\left[z\right]\right)\\
 & =:N_{0,1}+N_{0,2}+N_{1,1}+N_{1,2}+N_{2}+N_{3}+M=:F.
\end{align*}
Collecting the first line together and putting the remaining terms
as $\mathrm{F}$, one has
\begin{align}
i\partial_{t}\eta-\partial_{xx}\eta+V\eta & =-2\left|\mathcal{Q}\left[z(\infty)\right]\right|^{2}\eta-\left(\mathcal{Q}\left[z(\infty)\right]\right)^{2}e^{2i\int_{0}^{t}E\left[z(\sigma)\right]\,d\sigma}\bar{\eta}+\mathrm{F}\label{eq:rewriteeta}\
\end{align}
Note that from our bootstrap assumptions and the modulation equations,
$N_{1,1}$ and $N_{1,2}$ satisfy
\[
\left\Vert \left\langle x\right\rangle ^{m}\partial_{x}^{j}2\left(\left|\mathcal{Q}\left[z\right]\right|^{2}-\left|\mathcal{Q}\left[z(\infty)\right]\right|\right)\eta\right\Vert _{L^{2}\bigcap L^{\infty}}\lesssim\epsilon^{5}t^{-2+3\alpha}
\]
and
\[
\left\Vert \left\langle x\right\rangle ^{m}\partial_{x}^{j}\left(\left(\mathcal{Q}\left[z(\infty)\right]\right)^{2}-\left(\mathcal{Q}\left[z(\infty)\right]\right)^{2}\right)e^{2i\int_{0}^{t}E\left[z(\sigma)\right]\,d\sigma}\bar{\eta}\right\Vert _{L^{2}\bigcap L^{\infty}}\lesssim\epsilon^{5}t^{-2+3\alpha}
\]
for $j=0,1$. These two terms have the similar estimates of quadratic
terms and the modulation term.

To establish the weighted estimates for profiles of $\eta$, we first
project the equation \eqref{eq:rewriteeta} onto the continuous spectrum with respect to
$H$. With the notations
\begin{equation}
\eta=g+\mathsf{a}\left(t\right)\phi\label{eq:decompeta}
\end{equation}
where $\mathsf{a}\left(t\right)\phi=\left(\eta,\phi\right)\phi=P_{d}\eta$
and $g=P_{c}\eta$. Now the profile is given by $f=e^{-iHt}g$.

\subsubsection{Linear transformation}

We now analyze the linear part of the equation for
$\eta$ more carefully. The key point is that the orthogonality conditions
do not imply that $\eta$ is in the continuous spectrum of $H=-\partial_{xx}+V$.
Although due to the comparison of continuous spaces, one has $\eta\sim P_{c}\eta$.
But this comparison is time-dependent and more importantly, the time
derivative of this comparison given by $\mathcal{K}(z)$ in Lemma \ref{lem:Diff} has no good estimates. This comparison only works well when we
compute the decay estimates for a  fixed time, see Subsection \ref{subsec:bound}. These
are sufficient to analyze the quadratic terms and the cubic term.
To handle the first order perturbations, we need to explore the Fourier
transform in time and smoothing estimates. Then this comparison is
not effective anymore. 


Taking the first line of the equation \eqref{eq:rewriteeta}, one has 
\begin{equation}\label{eq:etalinear}
    i\eta_{t}=-H\eta+A\eta+Be^{i2\int_{0}^{t}E\left[z(\sigma)\right]\,d\sigma}\overline{\eta}
\end{equation}
where we denoted $A=2\left|\mathcal{Q}\left[z(\infty)\right]\right|^{2}$
and $B=\mathcal{Q}^{2}\left[z(\infty)\right]$. Without loss of generality,
we can assume that $A$ and $B$ are real-valued. Otherwise, we just
multiply $\eta$ by the constant phase given by $z(\infty)$. Here
again, we use $t=\infty$ for the sake of convenience. One can also
use $t=T$ to define $A$ and $B$. 


Projecting \eqref{eq:etalinear} onto the continuous spectrum with respect to $H$, one
has
\begin{equation}
ig_{t}=-Hg+P_{c}\left(A\left(g+\mathsf{a}(t)\phi\right)\right)+e^{i2\int_{0}^{t}E\left[z(\sigma)\right]\,d\sigma}P_{c}\left(B\left(\overline{g}+\overline{\mathsf{a}}(t)\phi\right)\right).\label{eq:linearg}
\end{equation}
To get a better understanding of $g$, we introduce a refined decomposition:
to find $\mathfrak{A}(x)\in P_{c}L^{2}$ and $\mathfrak{B}(x)\in P_{c}L^{2}$
and decompose
\begin{equation}
g=r+\mathsf{a}(t)\mathfrak{A}(x)+e^{i2\int_{0}^{t}E\left[z(\sigma)\right]\,d\sigma}\overline{\mathsf{a}(t)}\mathfrak{B}(x)\label{eq:refdecomp}
\end{equation}
such that plugging the decomposition \eqref{eq:refdecomp} above into
the linear equation \eqref{eq:linearg}, it results in an equation for
$r$, whose the RHS, approximately only $r$ is involved.

Now we recall the equation for $\mathsf{a}(t)$. Again, we are only
interested in the linear level. One has the linear equation:
\begin{align*}
i\dot{\mathsf{a}}(t) & =\rho^{2}\mathsf{a}(t)+\left(Ar+e^{i2\int_{0}^{t}E\left[z(\sigma)\right]\,d\sigma}B\overline{r},\phi\right)\\
 & +\mathsf{a}(t)\left(A\phi,\phi\right)+e^{i2\int_{0}^{t}E\left[z(\sigma)\right]\,d\sigma}\overline{\mathsf{a}}(t)\left(B\phi,\phi\right)\\
 & +\left(A\left(\mathsf{a}(t)\mathfrak{A}(x)+e^{i2\int_{0}^{t}E\left[z(\sigma)\right]\,d\sigma}\overline{\mathsf{a}(t)}\mathfrak{B}(x)\right),\phi\right)\\
 & +\left(e^{i2\int_{0}^{t}E\left[z(\sigma)\right]\,d\sigma}B\overline{\left(\mathsf{a}(t)\mathfrak{A}(x)+e^{i2\int_{0}^{t}E\left[z(\sigma)\right]\,d\sigma}\overline{\mathsf{a}(t)}\mathfrak{B}(x)\right)},\phi\right)
\end{align*}
We compute that
\[
\left(A\left(\mathsf{a}(t)\mathfrak{A}(x)+e^{i2\int_{0}^{t}E\left[z(\sigma)\right]\,d\sigma}\overline{\mathsf{a}(t)}\mathfrak{B}(x)\right),\phi\right)=\mathsf{a}(t)\left(A\mathfrak{A},\phi\right)+e^{i2\int_{0}^{t}E\left[z(\sigma)\right]\,d\sigma}\overline{\mathsf{a}(t)}\left(A\mathfrak{V},\phi\right)
\]
and{\small
\[
\left(e^{i2\int_{0}^{t}E\left[z(\sigma)\right]\,d\sigma}B\overline{\left(\mathsf{a}(t)\mathfrak{A}(x)+e^{i2\int_{0}^{t}E\left[z(\sigma)\right]\,d\sigma}\overline{\mathsf{a}(t)}\mathfrak{B}(x)\right)},\phi\right)=\mathsf{a}\left(t\right)\left(B\mathfrak{B},\phi\right)+e^{i2\int_{0}^{t}E\left[z(\sigma)\right]\,d\sigma}\overline{\mathsf{a}(t)}\left(B\mathfrak{A},\phi\right).
\]}
Therefore, one has
\begin{align}
i\dot{\mathsf{a}}(t) & =\rho^{2}\mathsf{a}(t)\label{eq:aeq}\\
 & +\left[\left(A\phi,\phi\right)+\left(A\mathfrak{A},\phi\right)+\left(B\mathfrak{B},\phi\right)\right]\mathsf{a}(t)\nonumber \\
 & +\left[\left(B\phi,\phi\right)+\left(A\mathfrak{V},\phi\right)+\left(B\mathfrak{A},\phi\right)\right]e^{i2\int_{0}^{t}E\left[z(\sigma)\right]\,d\sigma}\overline{\mathsf{a}}(t)\nonumber \\
 & +\left(Ar+e^{i2\int_{0}^{t}E\left[z(\sigma)\right]\,d\sigma}B\overline{r},\phi\right)\nonumber 
\end{align}
and
\begin{align*}
-ie^{i2\int_{0}^{t}E\left[z(\sigma)\right]\,d\sigma}\overline{\dot{\mathsf{a}}(t)} & =\rho^{2}e^{i2\int_{0}^{t}E\left[z(\sigma)\right]\,d\sigma}\overline{\mathsf{a}}(t)\\
 & +\left[\left(A\phi,\phi\right)+\left(A\mathfrak{A},\phi\right)+\left(B\mathfrak{B},\phi\right)\right]e^{i2\int_{0}^{t}E\left[z(\sigma)\right]\,d\sigma}\overline{\mathsf{a}}(t)\\
 & +\left[\left(B\phi,\phi\right)+\left(A\mathfrak{V},\phi\right)+\left(B\mathfrak{A},\phi\right)\right]\mathsf{a}(t)\\
 & +\left(Br+Ae^{i2\int_{0}^{t}E\left[z(\sigma)\right]\,d\sigma}\overline{r},\phi\right).
\end{align*}
We also record the formula
\begin{align}
ie^{i2\int_{0}^{t}E\left[z(\sigma)\right]\,d\sigma}\overline{\dot{\mathsf{a}}(t)}-2E\left[z(t)\right]e^{i2\int_{0}^{t}E\left[z(\sigma)\right]\,d\sigma}\overline{\mathsf{a}(t)}\label{eq:thetaaeq}\\
=\left(-\rho^{2}-2E\left[z(t)\right]\right)e^{i2\int_{0}^{t}E\left[z(\sigma)\right]\,d\sigma}\overline{\mathsf{a}}(t)\nonumber \\
-\left[\left(A\phi,\phi\right)+\left(A\mathfrak{A},\phi\right)+\left(B\mathfrak{B},\phi\right)\right]e^{i2\int_{0}^{t}E\left[z(\sigma)\right]\,d\sigma}\overline{\mathsf{a}}(t)\nonumber \\
-\left[\left(B\phi,\phi\right)+\left(A\mathfrak{V},\phi\right)+\left(B\mathfrak{A},\phi\right)\right]\mathsf{a}(t)\nonumber \\
-\left(Br+Ae^{i2\int_{0}^{t}E\left[z(\sigma)\right]\,d\sigma}\overline{r},\phi\right).\nonumber 
\end{align}
Now we compute the RHS of the equation \eqref{eq:linearg} with the
refined decomposition \eqref{eq:refdecomp}. Explicitly, one has
\begin{align*}
-Hg & =-H\left(r+\mathsf{a}(t)\mathfrak{A}(x)+e^{i2\int_{0}^{t}E\left[z(\sigma)\right]\,d\sigma}\overline{\mathsf{a}(t)}\mathfrak{B}(x)\right)\\
 & =-Hr-\mathsf{a}(t)H\mathfrak{A}-e^{i2\int_{0}^{t}E\left[z(\sigma)\right]\,d\sigma}\overline{\mathsf{a}(t)}H\mathfrak{B}.
\end{align*}
We also have
\begin{align}
P_{c}\left(A\left(g+\mathsf{a}(t)\phi\right)\right) & =P_{c}\left(A\left(r+\mathsf{a}(t)\mathfrak{A}(x)+e^{i2\int_{0}^{t}E\left[z(\sigma)\right]\,d\sigma}\overline{\mathsf{a}(t)}\mathfrak{B}(x)+\mathsf{a}(t)\phi\right)\right)\nonumber \\
 & =P_{c}\left(Ar\right)+\mathsf{a}(t)P_{c}\left(A\mathfrak{A}+A\phi\right)+e^{i2\int_{0}^{t}E\left[z(\sigma)\right]\,d\sigma}\overline{\mathsf{a}(t)}P_{c}\left(A\mathfrak{B}\right)\label{eq:g11}
\end{align}
and{\small
\begin{align}
e^{i2\int_{0}^{t}E\left[z(\sigma)\right]\,d\sigma}P_{c}\left(B\left(\overline{g}+\overline{\mathsf{a}}(t)\phi\right)\right) & =e^{i2\int_{0}^{t}E\left[z(\sigma)\right]\,d\sigma}P_{c}\left(B\left(\overline{r+\mathsf{a}(t)\mathfrak{A}(x)+e^{i2\int_{0}^{t}E\left[z(\sigma)\right]\,d\sigma}\overline{\mathsf{a}(t)}\mathfrak{B}(x)}+\overline{\mathsf{a}}(t)\phi\right)\right)\nonumber \\
 & =e^{i2\int_{0}^{t}E\left[z(\sigma)\right]\,d\sigma}P_{c}\left(B\overline{r}\right)+e^{i2\int_{0}^{t}E\left[z(\sigma)\right]\,d\sigma}\overline{\mathsf{a}(t)}P_{c}\left(B\mathfrak{A}+B\phi\right)+\mathsf{a}(t)P_{c}\left(B\mathfrak{B}\right).\label{eq:g12}
\end{align}}
Now we expand the RHS of the equation \eqref{eq:linearg},{\small
\[
i\left(r+\mathsf{a}(t)\mathfrak{A}(x)+e^{i2\int_{0}^{t}E\left[z(\sigma)\right]\,d\sigma}\overline{\mathsf{a}(t)}\mathfrak{B}(x)\right)_{t}=ir_{t}+i\mathsf{a}_{t}(t)\mathfrak{A}+i\frac{d}{dt}\left(e^{i2\int_{0}^{t}E\left[z(\sigma)\right]\,d\sigma}\overline{\mathsf{a}(t)}\right)\mathfrak{B}.
\]}
From equations \eqref{eq:aeq}, \eqref{eq:thetaaeq}, we have
\begin{align*}
i\dot{\mathsf{a}}(t)\mathfrak{A} & =\rho^{2}\mathsf{a}(t)\mathfrak{A}\\
 & +\left[\left(A\phi,\phi\right)+\left(A\mathfrak{A},\phi\right)+\left(B\mathfrak{B},\phi\right)\right]\mathsf{a}(t)\mathfrak{A}\\
 & +\left[\left(B\phi,\phi\right)+\left(A\mathfrak{V},\phi\right)+\left(B\mathfrak{A},\phi\right)\right]e^{i2\int_{0}^{t}E\left[z(\sigma)\right]\,d\sigma}\overline{\mathsf{a}}(t)\mathfrak{A}\\
 & +\left(Ar+e^{i2\int_{0}^{t}E\left[z(\sigma)\right]\,d\sigma}B\overline{r},\phi\right)\mathfrak{A}
\end{align*}
and{\small
\begin{align*}
\left(ie^{i2\int_{0}^{t}E\left[z(\sigma)\right]\,d\sigma}\overline{\dot{\mathsf{a}}(t)}-2E\left[z(t)\right]e^{i2\int_{0}^{t}E\left[z(\sigma)\right]\,d\sigma}\overline{\mathsf{a}(t)}\right)\mathfrak{B} & =\left(-\rho^{2}-2E\left[z(t)\right]\right)e^{i2\int_{0}^{t}E\left[z(\sigma)\right]\,d\sigma}\overline{\mathsf{a}}(t)\mathfrak{B}\\
 & -\left[\left(A\phi,\phi\right)+\left(A\mathfrak{A},\phi\right)+\left(B\mathfrak{B},\phi\right)\right]e^{i2\int_{0}^{t}E\left[z(\sigma)\right]\,d\sigma}\overline{\mathsf{a}}(t)\mathfrak{B}\\
 & -\left[\left(B\phi,\phi\right)+\left(A\mathfrak{V},\phi\right)+\left(B\mathfrak{A},\phi\right)\right]\mathsf{a}(t)\mathfrak{B}\\
 & -\left(Br+Ae^{i2\int_{0}^{t}E\left[z(\sigma)\right]\,d\sigma}\overline{r},\phi\right)\mathfrak{B}.
\end{align*}}Putting the computations above together, then from the LHS of the
equation of $g$, for the parts with $\mathsf{a}$ and $e^{i2\int_{0}^{t}E\left[z(\sigma)\right]\,d\sigma}\overline{\mathsf{a}}$,
one has{\footnotesize
\begin{align}
i\dot{\mathsf{a}}(t)\mathfrak{A}+\left(ie^{i2\int_{0}^{t}E\left[z(\sigma)\right]\,d\sigma}\overline{\dot{\mathsf{a}}(t)}-2E\left[z(t)\right]e^{i2\int_{0}^{t}E\left[z(\sigma)\right]\,d\sigma}\overline{\mathsf{a}(t)}\right)\mathfrak{B}\label{eq:togethera}\\
=\rho^{2}\mathsf{a}(t)\mathfrak{A}+\left(-\rho^{2}-2E\left[z(t)\right]\right)e^{i2\int_{0}^{t}E\left[z(\sigma)\right]\,d\sigma}\mathfrak{B}\nonumber \\
+\left\{ \left[\left(A\phi,\phi\right)+\left(A\mathfrak{A},\phi\right)+\left(B\mathfrak{B},\phi\right)\right]\mathfrak{A}-\left[\left(B\phi,\phi\right)+\left(A\mathfrak{V},\phi\right)+\left(B\mathfrak{A},\phi\right)\right]\mathfrak{B}\right\} \mathsf{a}(t)\nonumber \\
+\left\{ \left[\left(B\phi,\phi\right)+\left(A\mathfrak{V},\phi\right)+\left(B\mathfrak{A},\phi\right)\right]\mathfrak{A}-\left[\left(A\phi,\phi\right)+\left(A\mathfrak{A},\phi\right)+\left(B\mathfrak{B},\phi\right)\right]\mathfrak{B}\right\} e^{i2\int_{0}^{t}E\left[z(\sigma)\right]\,d\sigma}\overline{\mathsf{a}}(t)\nonumber \\
\left(Ar+e^{i2\int_{0}^{t}E\left[z(\sigma)\right]\,d\sigma}B\overline{r},\phi\right)\mathfrak{A}-\left(Br+Ae^{i2\int_{0}^{t}E\left[z(\sigma)\right]\,d\sigma}\overline{r},\phi\right)\mathfrak{B}\nonumber 
\end{align}}
and on the RHS of the equation, from \eqref{eq:g11} and \eqref{eq:g12},
we have
\begin{align}
-\mathsf{a}(t)H\mathfrak{A}-e^{i2\int_{0}^{t}E\left[z(\sigma)\right]\,d\sigma}\overline{\mathsf{a}(t)}H\mathfrak{B}\label{eq:RHSaa}\\
+\mathsf{a}(t)P_{c}\left(A\mathfrak{A}+A\phi\right)+e^{i2\int_{0}^{t}E\left[z(\sigma)\right]\,d\sigma}\overline{\mathsf{a}(t)}P_{c}\left(A\mathfrak{B}\right)\nonumber \\
+e^{i2\int_{0}^{t}E\left[z(\sigma)\right]\,d\sigma}\overline{\mathsf{a}(t)}P_{c}\left(B\mathfrak{A}+B\phi\right)+\mathsf{a}(t)P_{c}\left(B\mathfrak{B}\right)\nonumber \\
=\left(-H\mathfrak{A}+P_{c}\left(A\mathfrak{A}+A\phi\right)+P_{c}\left(B\mathfrak{B}\right)\right)\mathsf{a}(t)\nonumber \\
+\left(-H\mathfrak{B}+P_{c}\left(B\mathfrak{A}+B\phi\right)+P_{c}\left(A\mathfrak{B}\right)\right)e^{i2\int_{0}^{t}E\left[z(\sigma)\right]\,d\sigma}\overline{\mathsf{a}(t)}.\nonumber 
\end{align}
Matching the coefficients of $\mathsf{a}(t)$ and $e^{i2\int_{0}^{t}E\left[z(\sigma)\right]\,d\sigma}\overline{\mathsf{a}(t)}$
from \eqref{eq:togethera} and \eqref{eq:RHSaa}, one has the following
system:
\begin{align}
-H\mathfrak{A}+P_{c}\left(A\mathfrak{A}+A\phi\right)+P_{c}\left(B\mathfrak{B}\right) =\qquad\qquad\qquad\qquad\qquad\qquad\qquad\qquad\qquad\qquad\qquad& \label{eq:ell1}\\
\left\{ \left[\left(A\phi,\phi\right)+\left(A\mathfrak{A},\phi\right)+\left(B\mathfrak{B},\phi\right)\right]\mathfrak{A}-\left[\left(B\phi,\phi\right)+\left(A\mathfrak{V},\phi\right)+\left(B\mathfrak{A},\phi\right)\right]\mathfrak{B}\right\} +\rho^{2}\mathfrak{A}\nonumber 
\end{align}
and
\begin{align*}
-H\mathfrak{B}+P_{c}\left(B\mathfrak{A}+B\phi\right)+P_{c}\left(A\mathfrak{B}\right)=\qquad\qquad\qquad\qquad\qquad\qquad\qquad\qquad\qquad\qquad\qquad\qquad\qquad& \\
\left\{ \left[\left(B\phi,\phi\right)+\left(A\mathfrak{V},\phi\right)+\left(B\mathfrak{A},\phi\right)\right]\mathfrak{A}-\left[\left(A\phi,\phi\right)+\left(A\mathfrak{A},\phi\right)+\left(B\mathfrak{B},\phi\right)\right]\mathfrak{B}\right\} +\left(-\rho^{2}-2E\left[z(t)\right]\right)\mathfrak{B}.
\end{align*}
Notice that the later one is time-dependent. We replace $E\left[z(t)\right]$
by $E\left[z(\infty)\right]$. From the bootstrap assumption and the
modulation equation, $\left|E\left[z(t)\right]-E\left[z(\infty)\right]\right|\lesssim\epsilon^{2}t^{-1+2\alpha}.$
So we obtain an approximate equation for the second equation{\small
\begin{align}
-H\mathfrak{B}+P_{c}\left(B\mathfrak{A}+B\phi\right)+P_{c}\left(A\mathfrak{B}\right)=\qquad\qquad\qquad\qquad\qquad\qquad\qquad\qquad\qquad\qquad\qquad\qquad\qquad& \qquad\qquad\label{eq:ell2}\\
\left\{ \left[\left(B\phi,\phi\right)+\left(A\mathfrak{V},\phi\right)+\left(B\mathfrak{A},\phi\right)\right]\mathfrak{A}-\left[\left(A\phi,\phi\right)+\left(A\mathfrak{A},\phi\right)+\left(B\mathfrak{B},\phi\right)\right]\mathfrak{B}\right\} +\left(-\rho^{2}-2E\left[z(\infty)\right]\right)\mathfrak{B}.\nonumber 
\end{align}}Using the smallness of of the coefficients $A$ and $B$, via the
mapping properties of $(H+\rho^{2})P_c$ and $(H-\rho^{2}-2E\left[z(\infty)\right])P_c$
(note that $-\rho^{2}-2E\left[z(\infty)\right]\geq\rho^{2}$) , one
can construct a pair of smooth solutions $\left(\mathfrak{A},\mathfrak{B}\right)$
which decay exponentially to the system \eqref{eq:ell1} and \eqref{eq:ell2}.
This is similar to the construction of the nonlinear bound states
as Lemma \ref{lem:NLB}. 

After finding $\left(\mathfrak{A},\mathfrak{B}\right)$, combing the
computations above, the equation for $r$ is given by $r=P_{c}r$,
\begin{align}
ir_{t} & =-Hr+P_{c}\left(Ar\right)+e^{i2\int_{0}^{t}E\left[z(\sigma)\right]\,d\sigma}P_{c}\left(B\overline{r}\right)\nonumber \\
 & -\left(Ar+e^{i2\int_{0}^{t}E\left[z(\sigma)\right]\,d\sigma}B\overline{r},\phi\right)\mathfrak{A}-\left(Br+Ae^{i2\int_{0}^{t}E\left[z(\sigma)\right]\,d\sigma}\overline{r},\phi\right)\mathfrak{B}\label{eq:rlinear}\\
 & -2\left(E\left[z(t)\right]-E\left[z(\infty)\right]\right)\overline{\mathsf{a}(t)}e^{i2\int_{0}^{t}E\left[z(\sigma)\right]\,d\sigma}\mathfrak{B}.\nonumber 
\end{align}
Notice that the last term of the equation above is of quadratic form.
Importantly, we realize that the equation above has the same structure
of the first order perturbation of the model problem.

\subsubsection{Auxiliary bootstrap estimates}

Projecting the equation \eqref{eq:rewriteeta} onto the continuous spectrum,
the first line of it recasts the form of \eqref{eq:linearg}. From our
discussion above, we have a refined decomposition
\begin{equation}
g=r+\mathsf{a}(t)\mathfrak{A}(x)+e^{i2\int_{0}^{t}E\left[z(\sigma)\right]\,d\sigma}\overline{\mathsf{a}(t)}\mathfrak{B}(x)\label{eq:refinedg}
\end{equation}
with $r=P_{c}r$. From the discussion on \eqref{eq:rlinear}, the equation
for $r$ is given by
\begin{align}
ir_{t} & =-Hr+P_{c}\left(Ar\right)+e^{i2\int_{0}^{t}E\left[z(\sigma)\right]\,d\sigma}P_{c}\left(B\overline{r}\right)\nonumber \\
 & -\left(Ar+e^{i2\int_{0}^{t}E\left[z(\sigma)\right]\,d\sigma}B\overline{r},\phi\right)\mathfrak{A}-\left(Br+Ae^{i2\int_{0}^{t}E\left[z(\sigma)\right]\,d\sigma}\overline{r},\phi\right)\mathfrak{B}\label{eq:rlinear-1}\\
 & -2\left(E\left[z(t)\right]-E\left[z(\infty)\right]\right)\overline{\mathsf{a}(t)}e^{i2\int_{0}^{t}E\left[z(\sigma)\right]\,d\sigma}\mathfrak{B}+P_{c}\mathrm{F}.\nonumber 
\end{align}
Note that the last line of the form of quadratic and cubic terms. 

Now we impose the auxiliary estimates for $r$:
\begin{equation}
\left\Vert \left\langle x\right\rangle ^{-2}t\left(\partial_{t}^{2}-2E\left[z(t)\right]\partial_{t}\right)r_{L}\right\Vert _{L_{x}^{\infty}L^{2}\left[1,T\right]}+\left\Vert \left\langle x\right\rangle ^{-2}tr_{H}\right\Vert _{L_{x}^{\infty}L^{2}\left[1,T\right]}\lesssim\epsilon T^{\alpha}.\label{eq:auxestimater}
\end{equation}
Note that the linear perturbation of on the RHS of \eqref{eq:rlinear-1}
has the same form as the the model problem. Therefore, with the bootstrap
assumptions on $X_{T}$, one can recover the bootstrap auxiliary estimates
for $r$ just as the model problem in \S \ref{subsubsec:bootauxmodel}.

\subsubsection{Weighted estimates for the profile and analysis for ODEs}

We write the Duhamel expansion for the profile using the equation
\eqref{eq:rewriteeta}
\begin{equation}
\tilde{f}\left(t,k\right)=\tilde{f}\left(1,k\right)+\int_{1}^{t}e^{-ik^{2}s}\left(\tilde{N}_{0,1}\left(s\right)+\tilde{N}_{0,2}(s)+\widetilde{P_{c}\mathrm{F}}\left(s\right)\right)\,ds\label{eq:duhamelf-2}
\end{equation}
where
\[
N_{0,1}=-2P_{c}\left(\left|\mathcal{Q}\left[z(\infty)\right]\right|^{2}\eta\right)
\]
and
\[
N_{0,2}=-P_{c}\left(\left(\mathcal{Q}\left[z(\infty)\right]\right)^{2}e^{2i\int_{0}^{t}E\left[z(\sigma)\right]\,d\sigma}\bar{\eta}\right).
\]
The analysis of the higher order terms in $\mathrm{F}$ will be the
same as the model problem. So we focus on  the pieces contributed
by $N_{0,1}$ and $N_{0,2}$. Now we plug the decomposition of $\eta$, 
\eqref{eq:decompeta}, and the refined decomposition \eqref{eq:refinedg}
onto $N_{0,1}$ and $N_{0,2}$:
\begin{align*}
N_{0,1} & =-2P_{c}\left(\left|\mathcal{Q}\left[z(\infty)\right]\right|^{2}r\right)-2P_{c}\left(\left|\mathcal{Q}\left[z(\infty)\right]\right|^{2}\left(\phi+\mathfrak{A}\right)\right)\mathsf{a}(t)\\
 & -2P_{c}\left(\left|\mathcal{Q}\left[z(\infty)\right]\right|^{2}\mathfrak{B}\right)e^{i2\int_{0}^{t}E\left[z(\sigma)\right]\,d\sigma}\overline{\mathsf{a}(t)}\\
 & =:N_{0,1,r}+N_{0,1,a}+N_{0,1,\bar{a}}
\end{align*}
and
\begin{align*}
N_{0,2} & =-P_{c}\left(\left(\mathcal{Q}\left[z(\infty)\right]\right)^{2}e^{2i\int_{0}^{t}E\left[z(\sigma)\right]\,d\sigma}\bar{r}\right)-P_{c}\left(\left(\mathcal{Q}\left[z(\infty)\right]\right)^{2}\mathfrak{V}\right)\mathsf{a}(t)\\
 & -P_{c}\left(\left(\mathcal{Q}\left[z(\infty)\right]\right)^{2}\left(\phi+\mathfrak{A}\right)\right)e^{i2\int_{0}^{t}E\left[z(\sigma)\right]\,d\sigma}\overline{\mathsf{a}(t)}\\
 & =:N_{0,2,r}+N_{0,2,a}+N_{0,2,\bar{a}}.
\end{align*}
We focus on the analysis of $N_{0,1}$:
\begin{align}
\partial_{k}\int_{1}^{t}e^{-ik^{2}s}\left(\tilde{N}_{0,1}\left(s\right)\right)\,ds & =\partial_{k}\int_{1}^{t}e^{-ik^{2}s}\left(\tilde{N}_{0,1,r}\left(s\right)\right)\,ds+\partial_{k}\int_{1}^{t}e^{-ik^{2}s}\left(\tilde{N}_{0,1,a}\left(s\right)\right)\,ds\nonumber \\
 & +\partial_{k}\int_{1}^{t}e^{-ik^{2}s}\left(\tilde{N}_{0,1,\bar{a}}\left(s\right)\right)\,ds\label{eq:PkN10}
\end{align}
since the other piece will be similar as in the model problem. 

First of all, with the estimates for $r$, see \eqref{eq:auxestimater},
one can split $r$ into the high and low frequency parts and the the
same argument as the model problem will give us
\begin{equation}
\left\Vert \partial_{k}\int_{1}^{t}e^{-ik^{2}s}\left(\tilde{N}_{0,1,r}\left(s\right)\right)\,ds\right\Vert _{L_{k}^{2}}\lesssim\epsilon^{3}t^{\alpha}.\label{eq:rpart1}
\end{equation}
It remains to analyze the last two terms on the RHS of \eqref{eq:PkN10}.
We will give the detailed analysis of the second term on the RHS of
\eqref{eq:PkN10} since the other one can be analyzed in the same manner.

We write
\begin{align}
\partial_{k}\int_{1}^{t}e^{-ik^{2}s}\left(\tilde{N}_{0,1,a}\left(s\right)\right)\,ds & =-\int_{1}^{t}e^{-isk^{2}}2isk\int\mathsf{a}(t)\tilde{\mathcal{U}}_{1}\left(\ell\right)\nu\left(k,\ell\right)\,d\ell ds\nonumber \\
 & +\int_{1}^{t}e^{-isk^{2}}\int\mathsf{a}(t)\tilde{\mathrm{\mathcal{U}}}_{1}\left(\ell\right)\partial_{k}\nu\left(k,\ell\right)\,d\ell ds\label{eq:khanb}
\end{align}
where 
\[
\mathcal{U}_{1}=2P_{c}\left(\left|\mathcal{Q}\left[z(\infty)\right]\right|^{2}\left(\phi+\mathfrak{A}\right)\right)
\]
and
\begin{equation}
\nu\left(k,\ell\right)=\int\overline{\mathcal{K}}\left(x,k\right)\mathcal{K}\left(x,\ell\right)\,dx.\label{eq:nu3-1}
\end{equation}
Note that the last term of \eqref{eq:khanb} can estimated in the same
manner as those terms with $r$ involved since $\mathsf{a}(t)$ enjoys
the same decay estimates due to the comparison of the continuous spectrum.

Here we only analyze
\begin{equation}
\int_{1}^{t}e^{-isk^{2}}2isk\int\mathsf{a}(t)\tilde{\mathcal{U}}_{1}\left(t\right)\nu\left(k,\ell\right)\,d\ell ds.\label{eq:U1a}
\end{equation}
Note that the $k$ here is less important since morally it results
in taking spacial derivatives of $2P_{c}\left(\left|\mathcal{Q}\left[z(\infty)\right]\right|^{2}\left(\phi+\mathfrak{A}\right)\right)\mathsf{a}(t)$.
By construction, the coefficient $2P_{c}\left(\left|\mathcal{Q}\left[z(\infty)\right]\right|^{2}\left(\phi+\mathfrak{A}\right)\right)$
is smooth. So the influence of $k$ here is easy to understand.

Denote
\[
\mathsf{b}\left(t\right):=e^{i\rho^{2}t}\mathsf{a}\left(t\right).
\]
We first derive the equations for $\mathsf{b}(t)$. Recall that the
related ODEs for $\mathsf{a}$ are given by
\begin{align*}
i\dot{\mathsf{a}}(t) & =\rho^{2}\mathsf{a}(t)\\
 & +\left[\left(A\phi,\phi\right)+\left(A\mathfrak{A},\phi\right)+\left(B\mathfrak{B},\phi\right)\right]\mathsf{a}(t)\\
 & +\left[\left(B\phi,\phi\right)+\left(A\mathfrak{V},\phi\right)+\left(B\mathfrak{A},\phi\right)\right]e^{i2\int_{0}^{t}E\left[z(\sigma)\right]\,d\sigma}\overline{\mathsf{a}}(t)\\
 & +\left(Ar+e^{i2\int_{0}^{t}E\left[z(\sigma)\right]\,d\sigma}B\overline{r},\phi\right)+\left(\mathrm{F},\phi\right)\\
 & =:\mathsf{c}_{1}\mathsf{a}(t)+\mathsf{c}_{2}e^{i2\int_{0}^{t}E\left[z(\sigma)\right]\,d\sigma}\overline{\mathsf{a}}(t)+M
\end{align*}
and
\begin{align*}
-ie^{i2\int_{0}^{t}E\left[z(\sigma)\right]\,d\sigma}\overline{\dot{\mathsf{a}}(t)} & =\rho^{2}e^{i2\int_{0}^{t}E\left[z(\sigma)\right]\,d\sigma}\overline{\mathsf{a}}(t)\\
 & +\mathsf{c}_{1}e^{i2\int_{0}^{t}E\left[z(\sigma)\right]\,d\sigma}\overline{\mathsf{a}}(t)+\mathsf{c}_{2}\mathsf{a}(t)\\
 & +\left(Br+Ae^{i2\int_{0}^{t}E\left[z(\sigma)\right]\,d\sigma}\overline{r},\phi\right)+e^{i2\int_{0}^{t}E\left[z(\sigma)\right]\,d\sigma}\left(\mathrm{\overline{F}},\phi\right).
\end{align*}
Note that due to the smallness of the soliton, $\left|\mathsf{c}_{j}\right|\lesssim\epsilon^{2}$.
Here we also remark that in the expressions above, $\left(Ar+e^{i2\int_{0}^{t}E\left[z(\sigma)\right]\,d\sigma}B\overline{r},\phi\right)$
have the same structures as $N_{0,j,r}$ which can be estimated using
the bootstrap estimates \eqref{eq:auxestimater}. And the term $\left(\mathrm{F},\phi\right)$
can be bounded as the quadratic and cubic terms. Therefore, $M$ can
be estimated without further manipulations.

From the ODEs from $\mathsf{a}$, we have the following formulae
for the ODEs of $\mathsf{b}$:
\begin{equation}
e^{-i\rho^{2}t}\dot{\mathsf{b}}\left(t\right)=e^{-i\rho^{2}t}\mathsf{b}\left(t\right)\mathsf{c}_{1}+\mathsf{c}_{2}e^{2i\int_{0}^{t}E\left[z(\sigma)\right]\,d\sigma}e^{i\rho^{2}t}\mathrm{\overline{\mathsf{b}}}\left(t\right)+M\left(t\right).\label{eq:beq2}
\end{equation}
\begin{equation}
e^{2i\int_{0}^{t}E\left[z(\sigma)\right]\,d\sigma}e^{i\rho^{2}t}\dot{\mathsf{\bar{b}}}\left(t\right)=e^{2i\int_{0}^{t}E\left[z(\sigma)\right]\,d\sigma}e^{i\rho^{2}t}\mathsf{\bar{b}}\left(t\right)\mathsf{c}_{1}+\mathsf{c}_{2}e^{-i\rho^{2}t}\mathrm{\mathsf{b}}\left(t\right)+e^{2i\int_{0}^{t}E\left[z(\sigma)\right]\,d\sigma}\overline{M}\left(t\right)\label{eq:beq3}
\end{equation}
Now we consider the weighted estimate \eqref{eq:U1a}:
\[
\int_{1}^{t}e^{-isk^{2}}2isk\int\mathsf{a}(t)\tilde{\mathcal{U}}_{1}(\ell)\nu\left(k,\ell\right)\,d\ell ds=\int_{1}^{t}e^{-isk^{2}}2isk\int e^{-i\rho^{2}s}\mathsf{b}\left(s\right)\tilde{\mathcal{U}}_{1}\left(\ell\right)\nu\left(k,\ell\right)\,d\ell ds.
\]
Note that $k^{2}+\rho^{2}\ge\rho^{2}>0$. Using the identity
\[
e^{-isk^{2}}e^{-i\rho^{2}s}=-\frac{1}{i\left(k^{2}+\rho^{2}\right)}\frac{d}{ds}\left(e^{-isk^{2}}e^{-i\rho^{2}s}\right).
\]
we can perform integration by parts in $s$ and get{\footnotesize
\begin{align}
\int_{1}^{t}e^{-isk^{2}}2isk\int e^{-i\rho^{2}s}\mathsf{b}\left(s\right)\tilde{\mathcal{U}}_{1}\left(\ell\right)\nu\left(k,\ell\right)\,d\ell ds & \sim\frac{1}{i\left(k^{2}+\rho^{2}\right)}\int_{1}^{t}e^{-isk^{2}}2ik\int e^{-i\rho^{2}s}\mathsf{b}\left(s\right)\tilde{\mathcal{U}}_{1}\left(\ell\right)\nu\left(k,\ell\right)\,d\ell ds\nonumber \\
 & +\frac{1}{i\left(k^{2}+\rho^{2}\right)}\int_{1}^{t}e^{-isk^{2}}2isk\int e^{-i\rho^{2}s}\dot{\mathsf{b}}\left(s\right)\tilde{\mathcal{U}}_{1}\left(\ell\right)\nu\left(k,\ell\right)\,d\ell ds\label{eq:bibp1}\\
 & +\frac{1}{i\left(k^{2}+\rho^{2}\right)}e^{-itk^{2}}2ikt\int e^{-i\rho^{2}t}\mathsf{b}\left(t\right)\tilde{\mathcal{G}}_{1}\left(\ell\right)\nu\left(k,\ell\right)\,d\ell\nonumber \\
 & +\frac{1}{i\left(k^{2}+\rho^{2}\right)}e^{-ik^{2}}2ik\int e^{-i\rho^{2}}\mathsf{b}\left(1\right)\tilde{\mathcal{G}}_{1}\left(\ell\right)\nu\left(k,\ell\right)\,d\ell.\nonumber 
\end{align}}
For convenience, we denote the term on the LHS of the equation above
as
\[
\mathcal{I}_{1}\left(\mathsf{b}\right)\left(k,t\right):=\int_{1}^{t}e^{-isk^{2}}2isk\int e^{-i\rho^{2}s}\mathsf{b}\left(s\right)\tilde{\mathcal{U}}_{1}\left(\ell\right)\nu\left(k,\ell\right)\,d\ell ds.
\]
The first term of the RHS of \eqref{eq:bibp1} above can be easily bounded
using the decay estimate of $\mathsf{b}\left(s\right)\sim s^{-1+\alpha}$
\begin{equation}
\left\Vert \frac{1}{i\left(k^{2}+\rho^{2}\right)}\int_{1}^{t}e^{-isk^{2}}2ik\int e^{-\rho^{2}s}\mathsf{b}\left(s\right)\tilde{\mathcal{U}}_{1}\left(\ell\right)\nu\left(k,\ell\right)\,d\ell\right\Vert _{L_{k}^{2}}\lesssim\epsilon^{3}.\label{eq:b11}
\end{equation}
The boundary terms can be also estimated as
\begin{equation}
\left\Vert \frac{1}{i\left(k^{2}+\rho^{2}\right)}e^{-itk^{2}}2ikt\int e^{-i\rho^{2}t}\mathsf{b}\left(t\right)\tilde{\mathcal{G}}_{1}\left(\ell\right)\nu\left(k,\ell\right)\,d\ell\right\Vert _{L_{k}^{2}}\lesssim\epsilon^{3}t^{\alpha}\label{eq:b13}
\end{equation}
and the other one is easier. 

We focus on the second term of the RHS of \eqref{eq:bibp1},
\[
\frac{1}{i\left(k^{2}+\rho^{2}\right)}\int_{1}^{t}e^{-isk^{2}}2isk\int e^{-i\rho^{2}s}\dot{\mathsf{b}}\left(s\right)\tilde{\mathcal{U}}_{1}\left(\ell\right)\nu\left(k,\ell\right)\,d\ell ds.
\]
Using the equation for $\mathsf{b}(t)$, \eqref{eq:beq2}, one has
\begin{align}
\frac{1}{i\left(k^{2}+\rho^{2}\right)}\int_{1}^{t}e^{-isk^{2}}2isk\int e^{-i\rho^{2}s}\dot{\mathsf{b}}\left(s\right)\tilde{\mathcal{U}}_{1}\left(\ell\right)\nu\left(k,\ell\right)\,d\ell ds\nonumber \\
=\frac{1}{i\left(k^{2}+\rho^{2}\right)}\int_{1}^{t}e^{-isk^{2}}2isk\int e^{-i\rho^{2}s}\mathsf{c}_{1}\mathsf{b}\left(s\right)\tilde{\mathcal{U}}_{1}\left(\ell\right)\nu\left(k,\ell\right)\,d\ell ds\label{eq:firstb}\\
+\frac{1}{i\left(k^{2}+\rho^{2}\right)}\int e^{-isk^{2}}2isk\int\mathsf{c}_{2}e^{2i\int_{0}^{s}E\left[z(\sigma)\right]\,d\sigma}e^{i\rho^{2}s}\mathrm{\overline{\mathsf{b}}}\left(s\right)\tilde{\mathcal{U}}_{1}\left(\ell\right)\nu\left(k,\ell\right)\,d\ell ds\nonumber \\
+\frac{1}{i\left(k^{2}+\rho^{2}\right)}\int_{1}^{t}e^{-isk^{2}}2isk\int M\left(s\right)\tilde{\mathcal{U}}_{1}\left(\ell\right)\nu\left(k,\ell\right)\,d\ell ds.\nonumber 
\end{align}
The last term can be again handled as the quadratic terms
\begin{equation}
\left\Vert \frac{1}{i\left(k^{2}+\rho^{2}\right)}\int_{1}^{t}e^{-isk^{2}}2isk\int M\left(s\right)\tilde{\mathcal{U}}_{1}\left(\ell\right)\nu\left(k,\ell\right)\,d\ell ds\right\Vert _{L_{k}^{2}}\lesssim\epsilon^{4}.\label{eq:Mb1}
\end{equation}
We denote
\begin{equation}
\mathcal{I}_{2}\left(\bar{\mathsf{b}}\right)\left(k,t\right):=\int_{1}^{t}e^{-isk^{2}}2isk\int e^{2i\int_{0}^{s}E\left[z(\sigma)\right]\,d\sigma}e^{i\rho^{2}s}\mathrm{\overline{\mathsf{b}}}\left(s\right)\tilde{\mathcal{U}}_{1}\left(\ell\right)\nu\left(k,\ell\right)\,d\ell ds.\label{eq:I2b}
\end{equation}
Using these notations, from \eqref{eq:firstb} and estimates \eqref{eq:b11},
\eqref{eq:b13}, \eqref{eq:Mb1}, we conclude that
\begin{align}
\left\Vert \mathcal{I}_{1}\left(\mathsf{b}\right)\left(k,t\right)\right\Vert _{L_{k}^{2}} & \lesssim\left|\mathsf{c}_{1}\right|\left\Vert \mathcal{I}_{1}\left(\mathsf{b}\right)\left(k,t\right)\right\Vert _{L_{k}^{2}}+\left|\mathsf{c}_{2}\right|\left\Vert \mathcal{I}_{2}\left(\bar{\mathsf{b}}\right)\left(k,t\right)\right\Vert _{L_{k}^{2}}\label{eq:I1I21}\\
 & +\epsilon^{3}t^{\alpha}\nonumber 
\end{align}
which implies
\[
\left\Vert \mathcal{I}_{1}\left(\mathsf{b}\right)\left(k,t\right)\right\Vert _{L_{k}^{2}}\lesssim\left|\mathsf{c}_{2}\right|\left\Vert \mathcal{I}_{2}\left(\bar{\mathsf{b}}\right)\left(k,t\right)\right\Vert _{L_{k}^{2}}+\epsilon^{3}t^{\alpha}
\]
due to smallness of $\left|\mathsf{c}_{j}\right|$.

Now we consider integral given by $\mathcal{I}_{2}\left(\bar{\mathsf{b}}\right)$.
The analysis here is similar to what we did for $\mathcal{I}_{1}\left(\mathsf{b}\right)$.
Here we note that $k^{2}-2E\left[z(t)\right]-\rho^{2}\gtrsim\rho^{2}$.
Now we perform integration by parts using the identity
\[
e^{-isk^{2}}e^{i\rho^{2}s}e^{2i\int_{0}^{s}E\left[z(\sigma)\right]\,d\sigma}=-\frac{1}{i\left(k^{2}-\rho^{2}-2E\left[s\right]\right)}\frac{d}{ds}\left(e^{-isk^{2}}e^{i\rho^{2}s}e^{2i\int_{0}^{s}E\left[z(\sigma)\right]\,d\sigma}\right)
\]
and obtain that{\small
\begin{align}
\int_{1}^{t}e^{-isk^{2}}2isk\int e^{2i\int_{0}^{s}E\left[z(\sigma)\right]\,d\sigma}e^{i\rho^{2}s}\mathrm{\overline{\mathsf{b}}}\left(s\right)\tilde{\mathcal{U}}_{1}\left(\ell\right)\nu\left(k,\ell\right)\,d\ell ds\label{eq:IBPbbar}\\
\sim\int_{1}^{t}e^{-isk^{2}}2ik\frac{1}{i\left(k^{2}-\rho^{2}-2E\left[s\right]\right)}\int e^{2i\int_{0}^{s}E\left[z(\sigma)\right]\,d\sigma}e^{i\rho^{2}s}\mathrm{\overline{\mathsf{b}}}\left(s\right)\tilde{\mathcal{U}}_{1}\left(\ell\right)\nu\left(k,\ell\right)\,d\ell ds\nonumber \\
+\frac{1}{i\left(k^{2}+\rho^{2}\right)}\int_{1}^{t}e^{-isk^{2}}2isk\frac{1}{i\left(k^{2}-\rho^{2}-2E\left[s\right]\right)}\int e^{2i\int_{0}^{s}E\left[z(\sigma)\right]\,d\sigma}e^{i\rho^{2}s}\mathrm{\overline{\dot{\mathsf{b}}}}\left(s\right)\tilde{\mathcal{U}}_{1}\left(\ell\right)\nu\left(k,\ell\right)\,d\ell ds\nonumber \\
+\frac{1}{i\left(k^{2}+\rho^{2}\right)}e^{-itk^{2}}2ikt\frac{1}{i\left(k^{2}-\rho^{2}-2E\left[t\right]\right)}\int e^{2i\int_{0}^{t}E\left[z(\sigma)\right]\,d\sigma}e^{i\rho^{2}t}\mathrm{\overline{\mathsf{b}}}\left(t\right)\tilde{\mathcal{U}}_{1}\left(\ell\right)\nu\left(k,\ell\right)\,d\ell\nonumber \\
+\frac{1}{i\left(k^{2}+\rho^{2}\right)}e^{-ik^{2}}2ik\frac{1}{i\left(k^{2}-\rho^{2}-2E\left[1\right]\right)}\int e^{2i\int_{0}^{s}E\left[z(\sigma)\right]\,d\sigma}e^{i\rho^{2}}\mathrm{\overline{\mathsf{b}}}\left(1\right)\tilde{\mathcal{U}}_{1}\left(\ell\right)\nu\left(k,\ell\right)\,d\ell & .\nonumber 
\end{align}}
The first bulk term above can be estimated as \eqref{eq:b11}. The boundary
terms can be bounded in the same manner as \eqref{eq:b13}

For the second bulk term, we again use the equation for $\dot{\mathsf{b}}$, \eqref{eq:beq3},
and get{\small
\begin{align}
\frac{1}{i\left(k^{2}+\rho^{2}\right)}\int_{1}^{t}e^{-isk^{2}}2isk\frac{1}{i\left(k^{2}-\rho^{2}-2E\left[s\right]\right)}\int e^{2i\int_{0}^{s}E\left[z(\sigma)\right]\,d\sigma}e^{i\rho^{2}s}\mathrm{\overline{\dot{\mathsf{b}}}}\left(s\right)\tilde{\mathcal{U}}_{1}\left(\ell\right)\nu\left(k,\ell\right)\,d\ell ds\label{eq:bbarode}\\
=\frac{1}{i\left(k^{2}+\rho^{2}\right)}\int_{1}^{t}e^{-isk^{2}}2isk\frac{1}{i\left(k^{2}-\rho^{2}-2E\left[s\right]\right)}\int e^{2i\int_{0}^{s}E\left[z(\sigma)\right]\,d\sigma}e^{i\rho^{2}s}\mathsf{\bar{b}}\left(s\right)\mathsf{c}_{1}\tilde{\mathcal{U}}_{1}\left(\ell\right)\nu\left(k,\ell\right)\,d\ell ds\nonumber \\
+\frac{1}{i\left(k^{2}+\rho^{2}\right)}\int_{1}^{t}e^{-isk^{2}}2isk\frac{1}{i\left(k^{2}-\rho^{2}-2E\left[s\right]\right)}\int\mathsf{c}_{2}\mathrm{\mathsf{b}}\left(s\right)\tilde{\mathcal{U}}_{1}\left(\ell\right)\nu\left(k,\ell\right)\,d\ell ds\nonumber \\
+\frac{1}{i\left(k^{2}+\rho^{2}\right)}\int_{1}^{t}e^{-isk^{2}}2isk\frac{1}{i\left(k^{2}-\rho^{2}-2E\left[s\right]\right)}\int e^{2i\int_{0}^{s}E\left[z(\sigma)\right]\,d\sigma}M\left(s\right)\tilde{\mathcal{U}}_{1}\left(\ell\right)\nu\left(k,\ell\right)\,d\ell ds.\nonumber 
\end{align}}
As in \eqref{eq:b13}, we know the last term of the RHS above can be
bounded as
\[
\left\Vert \frac{1}{i\left(k^{2}+\rho^{2}\right)}\int_{1}^{t}e^{-isk^{2}}2isk\frac{1}{i\left(k^{2}-\rho^{2}-2E\left[s\right]\right)}\int e^{2i\int_{0}^{s}E\left[z(\sigma)\right]\,d\sigma}M\left(s\right)\tilde{\mathcal{G}}_{1}\left(\ell\right)\nu\left(k,\ell\right)\,d\ell ds\right\Vert _{L_{k}^{2}}\lesssim\epsilon^{4}.
\]
Therefore, similarly to \eqref{eq:I1I21}, from the expansion \eqref{eq:bbarode}
and expression \eqref{eq:IBPbbar}, we conclude that
\begin{align}
\left\Vert \mathcal{I}_{2}\left(\bar{\mathsf{b}}\right)\left(k,t\right)\right\Vert _{L_{k}^{2}} & \lesssim\left|\mathsf{c}_{1}\right|\left\Vert \mathcal{I}_{2}\left(\bar{\mathsf{b}}\right)\left(k,t\right)\right\Vert _{L_{k}^{2}}+\left|\mathsf{c}_{2}\right|\left\Vert \mathcal{I}_{1}\left(\mathsf{b}\right)\left(k,t\right)\right\Vert _{L_{k}^{2}}\label{eq:I1I21-1}\\
 & +\epsilon^{3}t^{\alpha}\nonumber 
\end{align}
which implies
\[
\left\Vert \mathcal{I}_{2}\left(\bar{\mathsf{b}}\right)\left(k,t\right)\right\Vert _{L_{k}^{2}}\lesssim\left|\mathsf{c}_{2}\right|\left\Vert \mathcal{I}_{1}\left(\mathsf{b}\right)\left(k,t\right)\right\Vert _{L_{k}^{2}}+\epsilon^{3}t^{\alpha}.
\]
Combing \eqref{eq:I1I21} with \eqref{eq:I1I21-1}, one obtains
\[
\left\Vert \mathcal{I}_{1}\left(\mathsf{b}\right)\left(k,t\right)\right\Vert _{L_{k}^{2}}+\left\Vert \mathcal{I}_{2}\left(\bar{\mathsf{b}}\right)\left(k,t\right)\right\Vert _{L_{k}^{2}}\lesssim\epsilon^{3}t^{\alpha}.
\]
Finally, we note that the same analysis can be applied to the term
\[
\partial_{k}\int_{1}^{t}e^{-ik^{2}s}\left(\tilde{N}_{0,1,\bar{a}}\left(s\right)\right)\,ds.
\]
We just note that the structure of the integral above is similar to
the structure of $\mathcal{I}_{2}\left(\bar{\mathsf{b}}\right)$.

Putting the part with $r$, we conclude that
\[
\left\Vert \partial_{k}\int_{1}^{t}e^{-ik^{2}s}\left(\tilde{N}_{0,1}\left(s\right)\right)\,ds\right\Vert _{L_{k}^{2}}\lesssim\epsilon^{3}t^{\alpha}
\]
and then
\[
\left\Vert \int_{1}^{t}e^{-ik^{2}s}\left(\tilde{N}_{0,1}\left(s\right)+\tilde{N}_{0,2}(s)\right)\,ds\right\Vert _{L_{k}^{2}}\lesssim\epsilon^{3}t^{\alpha}.
\]
Therefore, we recover the bootstrap assumptions for the weighted estimates
of the profile.
\begin{rem}
\label{rem:odenormal} We remark that the same analysis above can
also be applied in other settings, for example, the pointwise bounds
for the profile. The logic is to combine the decomposition \eqref{eq:decompeta}
and the refined one \eqref{eq:refinedg}. For the $r$ part, we use
the analysis of the model problem and for the remaining pieces, we
apply the integration by parts in time to analyze  $\mathsf{a}(t)$. 
\end{rem}

\begin{rem}
Note that in the analysis of $\mathcal{I}_{1}\left(\mathsf{b}\right)$
and $\mathcal{I}_{2}\left(\bar{\mathsf{b}}\right)$ above , we only
perform integration by parts in $s$ twice the equations for $\mathsf{b}\left(t\right)$
\eqref{eq:beq2} and \eqref{eq:beq3}. Actually, one can iterate this
process and use the smallness of $\mathsf{c}_{j}$ to obtain summable
series. By doing this, one can obtain bulk terms with only $r$ and
$e^{i2\int_{0}^{t}E\left[z(\sigma)\right]\,d\sigma}\overline{r}$
involved after paying the price of the boundary terms.
\end{rem}

\section{Pointwise bound}\label{sec:pointwise}

In this section, we show the global pointwise bound for the Fourier
transform of the profile $\tilde{f}\left(t,k\right)$ from the model problem \eqref{eq:modelu} in the bootstrap space \eqref{eq:bootstrap1-1}.

Again the key difficulties here will be the first order perturbations.  So we write  Duhamel's formula for the profile as
\begin{align}\label{eq:eqprofile}
\partial_{t}\tilde{f}\left(t,k\right) & =\iint e^{-isk^{2}}\tilde{a}_{1}\left(n\right)\tilde{u}\left(s,m\right)\,\nu\left(n,m,k\right)dmdn\\
 & +\iint e^{-isk^{2}}\tilde{a}_{2}\left(n\right)e^{2i\int_{0}^{t}E\left[z(\sigma)\right]\,d\sigma}\tilde{\bar{u}}\left(s,m\right)\,\nu\left(n,m,k\right)dmdn\\
 & +e^{-isk^{2}}\tilde{F}\left(\left|u\right|^{2}u\right)\left(t,k\right)\\
 & +e^{-isk^{2}}\tilde{F}\left(bu^{2}\right)\left(t,k\right).
\end{align}

Note that  the quadratic terms have localized coefficients, so one can use the improved
decay rate and obtain  an integrable inhomogenous term. The cubic term structure is computed in Chen-Pusateri \cite{CP}. 
\begin{prop}\label{proasy}
For $1 \leq t \leq T$, and $|k|\gtrsim t^{-3\alpha}$, using notations above and defining the modified profile
\begin{align}
\label{secasmod}
w (t,k):= \exp\Big(\frac{i}{2} \int_0^t |\wt{f}(s,k)|^2 \, \frac{ds}{1+s} \Big) \wt{f}(t,k),
\end{align}
for every  $1<t_1 < t_2 < T$ we have,
\begin{align}
\label{secas10}
\big| w(t_1,k) - w(t_2,k) | \lesssim (\epsilon  + {\| u \|}_{X_T}^{3} ) \, t_1^{-\varepsilon/2}.
\end{align}
for some $\varepsilon>0$.

\end{prop}

 \begin{proof}
We work on inhomogeneous terms from the equation for the profile \eqref{eq:eqprofile} separately. First of all, we analyze the quadratic term.   By the bootstrap assumption on the pointwise decay
of $u$, we know that
\begin{equation}\label{eq:quadpoint}
\left\Vert e^{-isk^{2}}\tilde{F}\left(bu^{2}\right)\left(t,k\right)\right\Vert _{L_{k}^{\infty}}\lesssim\epsilon^{3}t^{-2+2\alpha}
\end{equation}
see Corollary \ref{cor:directXT}.

Secondly, from the computations of Proposition 6.1 in Chen-Pusateri \cite{CP}, for the cubic term, one has
\begin{equation}\label{eq:cubicpoint}
    e^{-isk^{2}}\tilde{F}\left(\left|u\right|^{2}u\right)\left(t,k\right)=\frac{1}{2t}\left|\tilde{f}\left(t,k\right)\right|^{2}\tilde{f}\left(t,k\right)+\mathcal{O}\left(t^{-1-\varrho}\right)
\end{equation}
for $\left|k\right|\geq t^{-3\alpha}$ and some $\varrho>0$.

Multiplying the equation \eqref{eq:eqprofile} by the factor $\exp\left(\frac{i}{2}\int_{0}^{t}\left|\tilde{f}\left(s,k\right)\right|\frac{ds}{1+s}\right)$
and setting
\[
w\left(t,k\right)=\exp\left(\frac{i}{2}\int_{0}^{t}\left|\tilde{f}\left(s,k\right)\right|^{2}\frac{ds}{1+s}\right)\tilde{f}\left(t,k\right)
\]
from the equation \eqref{eq:eqprofile}, estimates \eqref{eq:quadpoint} and \eqref{eq:cubicpoint}, we have
we have
we have
\begin{align*}
\left|w\left(t_{2},k\right)-w\left(t_{1},k\right)\right| & \lesssim\left|\int_{t_{1}}^{t_{2}}\exp\left(\frac{i}{2}\int_{0}^{s}\left|\tilde{f}\left(\sigma,k\right)\right|^{2}\frac{d\sigma}{1+\sigma}\right)e^{-isk^{2}}\tilde{N}_{1,1}\left(s,k\right)\,ds\right|\\
 & +\left|\int_{t_{1}}^{t_{2}}\exp\left(\frac{i}{2}\int_{0}^{s}\left|\tilde{f}\left(\sigma,k\right)\right|^{2}\frac{d\sigma}{1+\sigma}\right)e^{-isk^{2}}\tilde{N}_{1,2}\left(s,k\right)\,ds\right|\\
 & +\epsilon^{3}\int_{t_{1}}^{t_{2}}s^{-2+2\alpha}\,ds+\int_{t_{1}}^{t_{2}}\mathcal{O}\left(s^{-1-\varrho}\right)\,ds
\end{align*}
where we denoted
\[
\tilde{N}_{1,1}\left(s,k\right):=\iint\tilde{a}_{1}\left(n\right)\tilde{u}\left(s,m\right)\,\nu\left(n,m,k\right)dmdn
\]
and
\[
\tilde{N}_{1,2}\left(s,k\right):=\iint\tilde{a}_{2}\left(n\right)e^{2i\int_{0}^{t}E\left[z(\sigma)\right]\,d\sigma}\tilde{\bar{u}}\left(s,m\right)\,\nu\left(n,m,k\right)dmdn.
\]
It remains to analyze
\begin{equation}
\int_{t_{1}}^{t_{2}}\exp\left(\frac{i}{2}\int_{0}^{s}\left|\tilde{f}\left(\sigma,k\right)\right|^{2}\frac{d\sigma}{1+\sigma}\right)\iint e^{-isk^{2}}\tilde{a}_{1}\left(n\right)\tilde{u}\left(s,m\right)\,\nu\left(n,m,k\right)dmdnds\label{eq:pointwiseh1}
\end{equation}
and{\small
\begin{equation}
\int_{t_{1}}^{t_{2}}\exp\left(\frac{i}{2}\int_{0}^{s}\left|\tilde{f}\left(\sigma,k\right)\right|^{2}\frac{d\sigma}{1+\sigma}\right)\iint e^{-isk^{2}}\tilde{a}_{2}\left(n\right)e^{2i\int_{0}^{t}E\left[z(\sigma)\right]\,d\sigma}\tilde{\bar{u}}\left(s,m\right)\,\nu\left(n,m,k\right)dmdnds\label{eq:pointwiseh2}
\end{equation}}
We will only present the analysis for \eqref{eq:pointwiseh1} since
the analysis of the second one will be similar to the first one with the same
manipulation as in the analysis of weighted estimates for first order
perturbations.

We divide the frequency $k$ into two pieces $\left|k\right|\leq s^{-\gamma}$
and $\left|k\right|\geq s^{-\gamma}$. Let $\varPsi_{j}$ be smooth
functions such that $\varPsi_{1}+\varPsi_{2}=1$, $\varPsi_{1}=1$
on $\left[0,1\right]$ and $\varPsi_{1}=0$ on $[2,\infty)$.

In the first region, due to the localization of $a\left(x\right)$
and the genericity of the potential, we can bound{\footnotesize
\begin{align}
\int_{t_{1}}^{t_{2}}\left|\varPsi_{1}\left(\left|k\right|s^{\gamma}\right)\exp\left(\frac{i}{2}\int_{0}^{s}\left|\tilde{f}\left(\sigma,k\right)\right|^{2}\frac{d\sigma}{1+\sigma}\right)\iint e^{-isk^{2}}\tilde{a}_{1}\left(n\right)\tilde{u}\left(s,m\right)\,\nu\left(n,m,k\right)dmdn\right|ds\nonumber \\
\lesssim\int_{t_{1}}^{t_{2}}\left|\varPsi_{1}\left(\left|k\right|s^{\gamma}\right)\exp\left(\frac{i}{2}\int_{0}^{s}\left|\tilde{f}\left(\sigma,k\right)\right|^{2}\frac{d\sigma}{1+\sigma}\right)e^{-isk^{2}}\int\mathcal{K}\left(x,k\right)\left(a_{1}\left(x\right)u\left(s,x\right)\right)dx\right|ds\nonumber \\
\lesssim\int_{t_{1}}^{t_{2}}\left|s^{-\gamma}\int\left(\left\langle x\right\rangle a_{1}\left(x\right)u\left(s,x\right)\right)dx\right|ds\nonumber \\
\lesssim\epsilon^{3}\left(t_{1}^{-\gamma+\alpha}+t_{2}^{-\gamma+\alpha}\right).\label{eq:lowfrepointwise}
\end{align}}
In the region $\left|k\right|\geq s^{-\gamma},$ we split $u=u_{L}+u_{H}$
as before with respect to the Fourier transform in $t$. Then we can
write{\footnotesize
\begin{align*}
\int_{t_{1}}^{t_{2}}\varPsi_{2}\left(\left|k\right|s^{\gamma}\right)\exp\left(\frac{i}{2}\int_{0}^{s}\left|\tilde{f}\left(\sigma,k\right)\right|^{2}\frac{d\sigma}{1+\sigma}\right)e^{-isk^{2}}\int\mathcal{K}\left(x,k\right)\left(a_{1}\left(x\right)u\left(s,x\right)\right)dxds\\
=\int_{t_{1}}^{t_{2}}\varPsi_{2}\left(\left|k\right|s^{\gamma}\right)\exp\left(\frac{i}{2}\int_{0}^{s}\left|\tilde{f}\left(\sigma,k\right)\right|^{2}\frac{d\sigma}{1+\sigma}\right)e^{-isk^{2}}\int\mathcal{K}\left(x,k\right)\left(a_{1}\left(x\right)u_{L}\left(s,x\right)\right)dxds\\
+\int_{t_{1}}^{t_{2}}\varPsi_{2}\left(\left|k\right|s^{\gamma}\right)\exp\left(\frac{i}{2}\int_{0}^{s}\left|\tilde{f}\left(\sigma,k\right)\right|^{2}\frac{d\sigma}{1+\sigma}\right)e^{-isk^{2}}\int\mathcal{K}\left(x,k\right)\left(a_{1}\left(x\right)u_{H}\left(s,x\right)\right)dxds.
\end{align*}}
For the high frequency part, one can bound
\begin{align*}
\left|\int_{t_{1}}^{t_{2}}\varPsi_{2}\left(\left|k\right|s^{\gamma}\right)\exp\left(\frac{i}{2}\int_{0}^{s}\left|\tilde{f}\left(\sigma,k\right)\right|^{2}\frac{d\sigma}{1+\sigma}\right)e^{-isk^{2}}\int\mathcal{K}\left(x,k\right)\left(a_{1}\left(x\right)u_{H}\left(s,x\right)\right)dxds\right|\\
\lesssim\left\Vert \left\langle s\right\rangle ^{-\frac{1}{2}-\varepsilon}\right\Vert _{L^{2}\left[t_{1},t_{2}\right]}\left\Vert \int\mathcal{K}\left(x,k\right)\left(a\left(x\right)\left\langle s\right\rangle ^{\frac{1}{2}+\varepsilon}u_{H}\left(s,x\right)\right)dx\right\Vert _{L_{t}^{2}\left[1,t\right]}
\lesssim\epsilon^{3}\left(t_{1}^{-\frac{\varepsilon}{2}}+t_{2}^{-\frac{\varepsilon}{2}}\right)
\end{align*}
provided that $0<\varepsilon<\frac{1}{2}-\alpha$, where in the last line, we applied Lemma \ref{lem:beta}.

For the low frequency part, integration by parts in $s$ results in{\small
\begin{align*}
\int_{t_{1}}^{t_{2}}\varPsi_{2}\left(\left|k\right|s^{\gamma}\right)\exp\left(\frac{i}{2}\int_{0}^{s}\left|\tilde{f}\left(\sigma,k\right)\right|^{2}\frac{d\sigma}{1+\sigma}\right)\int e^{-isk^{2}}\mathcal{K}\left(x,k\right)\left(a_{1}\left(x\right)u_{L}\left(s,x\right)\right)dxds\\
\sim\frac{1}{k^{2}}\int\varPsi_{2}\left(\left|k\right|t_{2}^{\gamma}\right)e^{-ik^{2}}\mathcal{K}\left(x,k\right)\left(a_{1}\left(x\right)u_{L}\left(t_{2},x\right)\right)dx\\
+\frac{1}{k^{2}}\int\varPsi_{2}\left(\left|k\right|t_{1}^{\gamma}\right)e^{-isk^{2}}\mathcal{K}\left(x,k\right)\left(a_{1}\left(x\right)u_{L}\left(t_{1},x\right)\right)dx\\
+\int_{t_{1}}^{t_{2}}\varPsi_{2}\left(\left|k\right|s^{\gamma}\right)\exp\left(\frac{i}{2}\int_{0}^{s}\left|\tilde{f}\left(\sigma,k\right)\right|^{2}\frac{d\sigma}{1+\sigma}\right)\frac{1}{k^{2}}\int e^{-isk^{2}}\mathcal{K}\left(x,k\right)\left(a_{1}\left(x\right)\partial_{s}u_{L}\left(s,x\right)\right)dxds\\
+\int_{t_{1}}^{t_{2}}\partial_{s}\left(\varPsi_{2}\left(\left|k\right|s^{\gamma}\right)\exp\left(\frac{i}{2}\int_{0}^{s}\left|\tilde{f}\left(\sigma,k\right)\right|^{2}\frac{d\sigma}{1+\sigma}\right)\right)\frac{1}{k^{2}}\int e^{-isk^{2}}\mathcal{K}\left(x,k\right)\left(a_{1}\left(x\right)u_{L}\left(s,x\right)\right)dxds.\end{align*}}
We note that{\footnotesize
\begin{align*}
\partial_{s}\left(\varPsi_{2}\left(\left|k\right|s^{\gamma}\right)\exp\left(\frac{i}{2}\int_{0}^{s}\left|\tilde{f}\left(\sigma,k\right)\right|^{2}\frac{d\sigma}{1+\sigma}\right)\right)\\
\sim\left|\tilde{f}\left(s,k\right)\right|^{2}\frac{1}{1+s}\left(\varPsi_{2}\left(\left|k\right|s^{\gamma}\right)\exp\left(\frac{i}{2}\int_{0}^{s}\left|\tilde{f}\left(\sigma,k\right)\right|^{2}\frac{d\sigma}{1+\sigma}\right)\right)\\
-\gamma s^{\gamma-1}\left|k\right|\varPsi_{2}'\left(\left|k\right|s^{\gamma}\right)\left(\varPsi_{2}\left(\left|k\right|s^{\gamma}\right)\exp\left(\frac{i}{2}\int_{0}^{s}\left|\tilde{f}\left(\sigma,k\right)\right|^{2}\frac{d\sigma}{1+\sigma}\right)\right).
\end{align*}}
Therefore with the localization of $a\left(x\right),$ we can estimate
\begin{align}
\left|\int_{t_{1}}^{t_{2}}\partial_{s}\left(\exp\left(\frac{i}{2}\int_{0}^{s}\left|\tilde{f}\left(\sigma,k\right)\right|^{2}\frac{d\sigma}{1+\sigma}\right)\right)\frac{1}{k^{2}}\int e^{-isk^{2}}\mathcal{K}\left(x,k\right)\left(a_{1}\left(x\right)u_{L}\left(s,x\right)\right)dxds\right|\nonumber \\
\lesssim\epsilon^{3}\int_{t_{1}}^{t_{2}}s^{-1}s^{-1+\alpha}s^{2\gamma}\,ds\lesssim\epsilon^{3}\left(t_{1}^{-1+\alpha+2\gamma}+t_{2}^{-1+\alpha+2\gamma}\right) & .\label{eq:pointwisebulk1}
\end{align}
For the boundary terms, it is easy to bound for $j=1,2$
\begin{equation}
\left|\frac{1}{k^{2}}\varPsi_{2}\left(\left|k\right|t_{j}^{\gamma}\right)\int e^{-it_{j}k^{2}}\mathcal{K}\left(x,k\right)\left(a\left(x\right)u_{L}\left(t_{j},x\right)\right)dx\right|\lesssim\epsilon^{3}t_{j}^{2\gamma}t_{j}^{-1+\alpha}\label{eq:pointwiseboundary}
\end{equation}
It remains to estimate
\[
\int_{t_{1}}^{t_{2}}\varPsi_{2}\left(\left|k\right|s^{\gamma}\right)\exp\left(\frac{i}{2}\int_{0}^{s}\left|\tilde{f}\left(\sigma,k\right)\right|^{2}\frac{d\sigma}{1+\sigma}\right)\frac{1}{k^{2}}\int e^{-isk^{2}}\mathcal{K}\left(x,k\right)\left(a_{1}\left(x\right)\partial_{s}u_{L}\left(s,x\right)\right)dxds.
\]
As in the analysis of weighted estimates, we denote
\[
\mathcal{N}_{1,1}\left(s,k\right):=e^{-2i\int_{0}^{s}E\left[z(\sigma)\right]\,d\sigma}\partial_{s}\tilde{N}_{1,1,L}\left(s,k\right)
\]
and $N_{1,1,L}=a_{1}u_{L}$. Then we write{\small
\begin{align*}
\int_{t_{1}}^{t_{2}}\varPsi_{2}\left(\left|k\right|s^{\gamma}\right)\exp\left(\frac{i}{2}\int_{0}^{s}\left|\tilde{f}\left(\sigma,k\right)\right|^{2}\frac{d\sigma}{1+\sigma}\right)\frac{1}{k^{2}}\int e^{-isk^{2}}\mathcal{K}\left(x,k\right)\left(a_{1}\left(x\right)\partial_{s}u_{L}\left(s,x\right)\right)dxds\\
=\int_{t_{1}}^{t_{2}}\varPsi_{2}\left(\left|k\right|s^{\gamma}\right)\exp\left(\frac{i}{2}\int_{0}^{s}\left|\tilde{f}\left(\sigma,k\right)\right|^{2}\frac{d\sigma}{1+\sigma}\right)e^{-isk^{2}}e^{2i\int_{0}^{s}E\left[z(\sigma)\right]\,d\sigma}\frac{1}{k^{2}}\mathcal{N}_{1,1}\left(s,k\right)\,ds.
\end{align*}}
Using the identity
\[
e^{-ik^{2}s}e^{2i\int_{0}^{s}E\left[z(\sigma)\right]\,d\sigma}=\frac{1}{-ik^{2}+2iE\left[z(s)\right]}\partial_{s}\left(e^{-ik^{2}s}e^{2i\int_{0}^{s}E\left[z(\sigma)\right]\,d\sigma}\right)
\]
to integrate by parts in $s$ again, we have{\footnotesize
\begin{align*}
\int_{t_{1}}^{t_{2}}\frac{1}{k^{2}}\varPsi_{2}\left(\left|k\right|s^{\gamma}\right)\exp\left(\frac{i}{2}\int_{0}^{s}\left|\tilde{f}\left(\sigma,k\right)\right|^{2}\frac{d\sigma}{1+\sigma}\right)\frac{1}{-ik^{2}+2iE\left[z(s)\right]}\partial_{s}\left(e^{-ik^{2}s}e^{2i\int_{0}^{s}E\left[z(\sigma)\right]\,d\sigma}\right)\left(\mathcal{N}_{1,1}\left(s,k\right)\right)\,ds\nonumber \\
=-\int_{1}^{t}\frac{1}{k^{2}}\partial_{s}\left(\varPsi_{2}\left(\left|k\right|s^{\gamma}\right)\exp\left(\frac{i}{2}\int_{0}^{s}\left|\tilde{f}\left(\sigma,k\right)\right|^{2}\frac{d\sigma}{1+\sigma}\right)\right)\frac{1}{-ik^{2}+2iE\left[z(s)\right]}\left(e^{-ik^{2}s}e^{2i\int_{0}^{s}E\left[z(\sigma)\right]\,d\sigma}\right)\left(\mathcal{N}_{1,1}\left(s,k\right)\right)\,ds\nonumber \\
+\int_{t_{1}}^{t_{2}}\frac{1}{k^{2}}\frac{E'\left[z(s)\right]\frac{d}{ds}\left|z\left(s\right)\right|}{\left(-ik^{2}+2iE\left[z(s)\right]\right)^{2}}\varPsi_{2}\left(\left|k\right|s^{\gamma}\right)\exp\left(\frac{i}{2}\int_{0}^{s}\left|\tilde{f}\left(\sigma,k\right)\right|^{2}\frac{d\sigma}{1+\sigma}\right)\left(e^{-ik^{2}s}e^{2i\int_{0}^{s}E\left[z(\sigma)\right]\,d\sigma}\right)\left(\mathcal{N}_{1,1}\left(s,k\right)\right)\,ds\nonumber \\
-\int_{t_{1}}^{t_{2}}\frac{1}{k^{2}}\frac{1}{-ik^{2}+2iE\left[z(s)\right]}\varPsi_{2}\left(\left|k\right|s^{\gamma}\right)\exp\left(\frac{i}{2}\int_{0}^{s}\left|\tilde{f}\left(\sigma,k\right)\right|^{2}\frac{d\sigma}{1+\sigma}\right)\left(e^{-ik^{2}s}e^{2i\int_{0}^{s}E\left[z(\sigma)\right]\,d\sigma}\right)\partial_{s}\left(\mathcal{N}_{1,1}\left(s,k\right)\right)\,ds\\
+\frac{1}{k^{2}}\left(\varPsi_{2}\left(\left|k\right|t_{1}^{\gamma}\right)\exp\left(\frac{i}{2}\int_{0}^{t_{1}}\left|\tilde{f}\left(\sigma,k\right)\right|^{2}\frac{d\sigma}{1+\sigma}\right)\right)\frac{1}{-ik^{2}+2iE\left[z(t_{1})\right]}\left(e^{-ik^{2}t_{1}}e^{2i\int_{0}^{t_{1}}E\left[z(\sigma)\right]\,d\sigma}\right)\left(\mathcal{N}_{1,1}\left(t_{1},k\right)\right)\nonumber \\
+\frac{1}{k^{2}}\left(\varPsi_{2}\left(\left|k\right|t_{2}^{\gamma}\right)\exp\left(\frac{i}{2}\int_{0}^{t_{2}}\left|\tilde{f}\left(\sigma,k\right)\right|^{2}\frac{d\sigma}{1+\sigma}\right)\right)\frac{1}{-ik^{2}+2iE\left[z(t_{2})\right]}\left(e^{-ik^{2}t_{2}}e^{2i\int_{0}^{t_{2}}E\left[z(\sigma)\right]\,d\sigma}\right)\left(\mathcal{N}_{1,1}\left(t_{2},k\right)\right).\nonumber 
\end{align*}}
The last two terms above are boundary terms which can be bounded as
\eqref{eq:pointwiseboundary}. The first two terms above can be estimated
by the same way as \eqref{eq:pointwisebulk1} with the decay of $\frac{d}{ds}\left|z\left(s\right)\right|$.

Finally, it remains to estimate{\footnotesize
\[
\int_{t_{1}}^{t_{2}}\frac{1}{k^{2}}\frac{1}{-ik^{2}+2iE\left[z(s)\right]}\varPsi_{2}\left(\left|k\right|s^{\gamma}\right)\exp\left(\frac{i}{2}\int_{0}^{s}\left|\tilde{f}\left(\sigma,k\right)\right|^{2}\frac{d\sigma}{1+\sigma}\right)\left(e^{-ik^{2}s}e^{2i\int_{0}^{s}E\left[z(\sigma)\right]\,d\sigma}\right)\partial_{s}\left(\mathcal{N}_{1,1}\left(s,k\right)\right)\,ds.
\]}
By construction,
\[
\left(e^{-ik^{2}s}e^{2i\int_{0}^{s}E\left[z(\sigma)\right]\,d\sigma}\right)\partial_{s}\left(\mathcal{N}_{1,1}\left(s,k\right)\right)=e^{-ik^{2}s}\int\mathcal{K}\left(x,k\right)a_{1}\left(-2iE\left[z(s)\right]\partial_{s}+\partial_{s}^{2}\right)u_{L}\,dx.
\]
Therefore as the high frequency part, one has{\footnotesize
\begin{align*}
\left|\int_{t_{1}}^{t_{2}}\frac{1}{k^{2}}\frac{1}{-ik^{2}+2iE\left[z(s)\right]}\varPsi_{2}\left(\left|k\right|s^{\gamma}\right)\exp\left(\frac{i}{2}\int_{0}^{s}\left|\tilde{f}\left(\sigma,k\right)\right|^{2}\frac{d\sigma}{1+\sigma}\right)\left(e^{-ik^{2}s}e^{2i\int_{0}^{s}E\left[z(\sigma)\right]\,d\sigma}\right)\partial_{s}\left(\mathcal{N}_{1,1}\left(s,k\right)\right)\,ds\right|\\
\lesssim\left\Vert \left\langle s\right\rangle ^{-\frac{1}{2}-\varepsilon}\right\Vert _{L^{2}\left[t_{1},t_{2}\right]}\left\Vert \frac{1}{k^{2}}\varPsi\left(\left|k\right|\geq s^{-\gamma}\right)\int\mathcal{K}\left(x,k\right)a_{1}\left(-2iE\left[z(s)\right]\partial_{s}+\partial_{s}^{2}\right)u_{L}\,dx\right\Vert _{L_{t}^{2}\left[1,t\right]}\\
\lesssim\left\Vert \left\langle s\right\rangle ^{-\frac{1}{2}-\varepsilon}\right\Vert _{L^{2}\left[t_{1},t_{2}\right]}\left\Vert \int\frac{\mathcal{K}\left(x,k\right)}{k}\left(\left\langle x\right\rangle a_{1}\left(x\right)\left\langle s\right\rangle ^{\frac{1}{2}+\varepsilon}s^{\gamma}\left(-2iE\left[z(s)\right]\partial_{s}+\partial_{s}^{2}\right)u_{L}\left(s,x\right)\right)dx\right\Vert _{L_{t}^{2}\left[1,t\right]}\\
\lesssim\left(t_{1}^{-\frac{\varepsilon}{2}}+t_{2}^{-\frac{\varepsilon}{2}}\right)\left\Vert \left\langle x\right\rangle a\left(x\right)\left\langle s\right\rangle ^{\frac{1}{2}+\varepsilon}s^{\gamma}\left(-2iE\left[z(s)\right]\partial_{s}+\partial_{s}^{2}\right)u_{L}\left(s,x\right)\right\Vert _{L_{x}^{1}L_{t}^{2}\left[1,t\right]}
\lesssim\left(t_{1}^{-\frac{\varepsilon}{2}}+t_{2}^{-\frac{\varepsilon}{2}}\right)\epsilon^{3}
\end{align*}}
provided that $0<\varepsilon<\frac{1}{2}-\gamma-\alpha$, where in the last line, we again applied Lemma \ref{lem:beta}.

Finally, we pick $\gamma=2\alpha$ and $0<\varepsilon<\frac{1}{2}-3\alpha$.
It follows that{\footnotesize
\[
\left|\int_{t_{1}}^{t_{2}}\exp\left(\frac{i}{2}\int_{0}^{s}\left|\tilde{f}\left(\sigma,k\right)\right|^{2}\frac{d\sigma}{1+\sigma}\right)\iint e^{-isk^{2}}\tilde{a}_{1}\left(n\right)\tilde{u}\left(m\right)\,\nu\left(n,m,k\right)dmdnds\right|\lesssim\epsilon^{3}\left(t_{1}^{-\frac{\varepsilon}{2}}+t_{2}^{-\frac{\varepsilon}{2}}\right).
\]}
Similarly, one also has{\footnotesize
\[
\left|\int_{t_{1}}^{t_{2}}\exp\left(\frac{i}{2}\int_{0}^{s}\left|\tilde{f}\left(\sigma,k\right)\right|^{2}\frac{d\sigma}{1+\sigma}\right)\iint e^{-isk^{2}}\tilde{a}_{2}\left(n\right)e^{2i\int_{0}^{t}E\left[z(\sigma)\right]\,d\sigma}\tilde{\bar{u}}\left(s,m\right)\,\nu\left(n,m,k\right)dmdnds\right|\lesssim\epsilon^{3}\left(t_{1}^{-\frac{\varepsilon}{2}}+t_{2}^{-\frac{\varepsilon}{2}}\right).
\]}
Therefore we conclude that
\begin{align*}
\left|w\left(t_{2},k\right)-w\left(t_{1},k\right)\right| & \lesssim\epsilon^{3}\left(t_{1}^{-\frac{\varepsilon}{2}}+t_{2}^{-\frac{\varepsilon}{2}}\right).
\end{align*}
These recover the bootstrap conditions for the pointwise estimates
of the profile.
\end{proof}


\smallskip
\subsection{Bootstrap argument and proof of Theorem \ref{thm:mainu}}
Finally, we use Proposition \ref{proasy} and Proposition \ref{pro:weightmain1}  to close 
our bootstrap argument, obtain a global solution, and complete the proof of Theorem \ref{thm:mainu}.

Recall the definition of $X_T$ in \eqref{eq:bootstrap1-1}, 
and assume that for $\epsilon_1 := \epsilon^{2/3}$. We make the \emph{a priori} assumption
\begin{align}\label{apriori0}
{\| u \|}_{X_T} \leq \epsilon_1.
\end{align}
Proposition \ref{pro:weightmain1} implies, see also \eqref{eq:weiF},
\begin{align}\label{apriori1}
{\big\| \partial_{k}\tilde{f}(t) \big\|}_{L_{k}^{2}}
  \leq \epsilon + C |t|^{\alpha} \epsilon_1^3 \leq 2\epsilon |t|^{\alpha},
\end{align}
provided $\epsilon\leq \epsilon_0$ small enough.
Then observe that, by our assumptions, 
$\tilde{f}\left(t,0\right)=0$ for all $t \in [0,T]$ so that, for $\left|k\right|\lesssim\left|t\right|^{-3\alpha}$, 
we have
\begin{align}\label{apriori2}
|\tilde{f}\left(t,k\right)| \leq \left|\int_{0}^{k} \partial_{\eta}\tilde{f}\left(t,\eta\right)\,d\eta\right|
  \leq \sqrt{\left|k\right|} \,
  {\| \partial_{k}\tilde{f} \|}_{L^{2}} \lesssim |t|^{-\frac{\alpha}{2}} 2\epsilon . 
\end{align}
In particular, the low-frequency part of $\tilde{f}(t,k)$ goes to zero as $t\rightarrow \infty$,
and we can reduce matters to considering only $|k| \geq |t|^{-3\alpha}$.
Under this latter condition, using \eqref{secasmod} and \eqref{secas10} in Proposition \ref{proasy}
we deduce that
\begin{align}\label{apriori3}
|\tilde{f}\left(t,k\right)| = |w(t,k)| \leq |w(1,k)| + C \epsilon \lesssim \epsilon. 
\end{align}
\eqref{apriori1}-\eqref{apriori3} imply ${\| u \|}_{X_T} \leq \epsilon_1/2$, improving on \eqref{apriori0},
so that a standard continuation argument gives us a global solution which is bounded in the $X_\infty$ norm. 
Eventually, we obtain the following:

\begin{cor} 
Let $u = e^{it(\partial_{xx}+V)}f$ be the global-in-time solution obtained above.
With the same notation of Proposition \ref{proasy}, we have that
$w(t)$ is a Cauchy sequence in time with values in $L^\infty$.
Letting $W_{+\infty} := \lim_{t \rightarrow \infty} w(t)$ we obtain the asymptotics \eqref{mainasy1}.
\end{cor}
\subsection{The pointwise bound for the full problem}
The pointwise bound for the profile of the full problem can be obtained using the refined decomposition \eqref{eq:refdecomp} with a similar argument for weighted estimates, see Remark \ref{rem:odenormal}.  We omit details here.

\appendix
\section{Elliptic equations and projections}\label{sec:NBS}
In this appendix, for the sake of completeness, we provide some details on the existence of small nonlinear bound states and the comparison of continuous spectrum.
\begin{lem}[Nonlinear bound states]
\label{lem:NLB} There exists $\delta>0$
such that for $z\in\mathbb{C}$ with $\left|z\right|\le\delta$, there
is a solution $Q\left[z\right]\left(x\right)\in H^{2}\bigcap W^{1,1}$
of 
\[
\left(-\partial_{xx}+V\right)Q-\left|Q\right|^{2}Q=EQ
\]
with $E=E\left[\left|z\right|\right]\in\mathbb{R}$ such that
\[
Q\left[z\right]=z\phi+q\left[z\right],\,\,\left(q,\phi_{0}\right)=0.
\]
The pair $\left(q,E\right)$ is unique in the class
\begin{equation}
\left\Vert q\right\Vert _{H^{2}}\leq\delta,\,\,\left|E-\left(-\rho^{2}\right)\right|\leq\delta.\label{eq:class}
\end{equation}
Moreover, $Q\left[ze^{i\alpha}\right]=Q\left[z\right]e^{i\alpha}$,
$Q\left[\left|z\right|\right]$ is a real-valued function, and 
\[
q\left[z\right]=o\left(z^{2}\right),\,\text{in}\,H^{2}\bigcap W^{1,1}
\]
\[
\text{D}Q\left[z\right]=\left(1,i\right)\phi_{0}+o\left(z\right),\,\,\text{D}^{2}Q\left[z\right]=o\left(1\right)\,\text{in}\,\,H^{2}\bigcap W^{1,1},
\]
\[
E\left[z\right]=-\rho^{2}+o\left(z\right),\,\,\text{D}E\left[z\right]=o\left(1\right)
\]
as $z\rightarrow0$.

We finally notice that since our potential decays fast and by \eqref{eq:class},
the nonlinear bound state $Q\left[z\right]$ decays exponentially.
\end{lem}

\begin{proof}
For the completed proof, see the Appendix of Gustafson-Nakanishi-Tsai
\cite{GNT}.
\end{proof}
The important point for the small soliton is that the spectral projections
in our stability analysis stay comparable with projections with respect
to the fixed Schr\"odinger operator:
\[
H=-\partial_{xx}+V
\]
not the one influenced by the soliton. In the modulation process,
we need to project onto the continuous spectrum with respect to the
operator
\[
\mathrm{H}\left(z\right)g=-\partial_{xx}g+Vg-2\left|Q\left[z\right]\right|^{2}g-\left(Q\left[z\right]\right)^{2}\bar{g}.
\]
The following lemma precisely characterizes the difference between
the projections onto the continuous subspaces for $H$ and $\mathrm{H}\left(z\right)$ 
respectively. See Definition \ref{def:Conti} for the definition for the continuous space for $\mathrm{H}(z)$.
\begin{lem}[Difference of continuous spectral projections]
\label{lem:Diff}Suppose
$\delta>0$ is small enough. Then $\forall z\in\mathbb{C}$ and $\left|z\right|<\delta$,
there is a bijection 
\[
\mathcal{K}\left(z\right):\,\mathcal{H}_{c}\left[0\right]\rightarrow\mathcal{H}_{c}\left[z\right]
\]
such that\textup{
\[
\mathcal{K}\left(z\right)\left(P_{c}|_{\mathcal{H}_{c}\left[z\right]}\right)=I,\,\left(P_{c}|_{\mathcal{H}\left[z\right]}\right)\mathcal{K}\left(z\right)=I
\]
and}
\[
\mathcal{K}\left(z\right)-I
\]
\textup{is compact and continuous in $z$ with respect to the operator
norm on any space $Y$ such that 
\[
H^{2}\bigcap W^{1,1}\subset Y\subset H^{-2}+L^{\infty}.
\]
}
\end{lem}

\begin{proof}
Seee Lemma 2.2 in Gustafson-Nakanishi-Tsai \cite{GNT}.

\end{proof}

\section{Wellposedness}\label{sec:GWP}

Consider the cubic equation 
\begin{equation}
i\partial_{t}u-\partial_{xx}u+Vu\pm\left|u\right|^{2}u=0\label{eq:nonlineq}
\end{equation}
where $V$ is a generic potential. For the $L^{2}$
problem, one can use $L_{t}^{4}L_{x}^{\infty}\bigcap C_{t}L_{x}^{2}$
to do contraction. 

To establish the $H^{1}$ wellposedness, one consider the boundedness
of $\frac{\partial_{x}}{\sqrt{-\partial_{xx}+V+\left\Vert V\right\Vert _{L^{\infty}}+1}}$
in $L^{p}$ for $p=2$. Denote $\mathcal{D}=\sqrt{-\partial_{xx}+V+\left\Vert V\right\Vert _{L^{\infty}}+1}$.
Then $\mathcal{D}$ commutes with $e^{itH}$. Then one should show
that $\mathcal{D}u\in C_{t}L_{x}^{2}$. 
\begin{thm}\label{thm:L2existence}
For any $u_{0}\in L^{2}$, there  exists a global-in-time solution
to 
\[
i\partial_{t}u-\partial_{xx}u+Vu\pm\left|u\right|^{2}u=0
\]
whose $L^{2}$ norm is preserved.
\end{thm}

\begin{proof}
It is easy to check that if $u$ is smooth solution to the above equation,
then the $L^{2}$ norm is preserved since $V$ is a real-valued function. 

We write the solution using Duhamel's formula with respect to the perturbed
Schr\"odinger operator $H$,
\begin{equation}
u=e^{-itH}u_{0}-i\int_{0}^{t}e^{-\left(t-s\right)iH}\left(\left|u\right|^{2}u\left(s\right)\right)\,ds.\label{eq:mildsol}
\end{equation}
Then for given $T\geq0$ to be chose later, we define the following
space
\[
X_{T}:=\left\{ f\left(x,t\right),\,\left\Vert f\right\Vert _{L_{t}^{\infty}\left(\left[0,T\right]:L_{x}^{2}\right)}<\infty,\left\Vert f\right\Vert _{L_{t}^{4}\left(\left[0,T\right]:L_{x}^{\infty}\right)}<\infty\right\} .
\]
Define the map $\mathrm{F}$ from $X_T$ to $X_T$ as
\begin{equation}
    \mathrm{F}(f):= e^{-itH}u_{0}-i\int_{0}^{t}e^{-\left(t-s\right)iH}\left(\left|f\right|^{2}f\left(s\right)\right)\,ds.
\end{equation}
The desired solution will be the fixed point of the map $\mathrm{F}$. 
We will show that if $T$ is small depending on the size the $L^2$ norm of $u_0$, $\mathrm{F}$ is a contraction map from $X_T$ to $X_T$.  

To show \emph{a priori} estimates, putting the $X_{T}$ norm on both sides of \eqref{eq:mildsol}, by the Strichartz estimate
for $e^{-itH}$ and the fact it is unitary in $L^{2}$, one has 
\[
\left\Vert e^{-itH}u_{0}\right\Vert _{X_{T}}\lesssim\left\Vert u_{0}\right\Vert _{L^{2}}.
\]
For the inhomogeneous term, 
\[
\left\Vert i\int_{0}^{t}e^{-\left(t-s\right)iH}\left(\left|u\right|^{2}u\left(s\right)\right)\,ds\right\Vert _{X_{T}}\lesssim\left\Vert \left|u\right|^{2}u\right\Vert _{L_{t}^{1}\left(\left[0,T\right]:L_{x}^{2}\right)}.
\]
Then we notice that 
\begin{align*}
\left\Vert \left|u\right|^{2}u\right\Vert _{L_{t}^{1}\left(\left[0,T\right]:L_{x}^{2}\right)} & \lesssim T^{\frac{1}{2}}\left\Vert u\right\Vert _{L_{t}^{4}\left(\left[0,T\right]:L_{x}^{\infty}\right)}^{2}\sup_{0\leq t\leq T}\left\Vert u\right\Vert _{L_{x}^{2}}\\
 & \lesssim T^{\frac{1}{2}}\left\Vert u\right\Vert _{X_{T}}^{3}.
\end{align*}
Suppose $\left\Vert u_{0}\right\Vert _{L^{2}}=M$ and choose $T=M^{-4}$,
then one can use the  estimate above to show $\mathrm{F}$ is a contraction in a ball in
$X_{T}$ with radius $2CM$ where $C$ is the constant from the $L_{t}^{4}L_{x}^{\infty}$
Strichartz estimates.

Hence by the fixed point theorem, we can construct a unique solution $u\in X_{T}$ which solves
the equation \eqref{eq:mildsol} satisfying 
\[
\left\Vert u\right\Vert _{X_{T}}\leq2CM.
\]
Moreover by similar construction, one can obtain the continuity dependence
with respect to the initial data. In other words, if $u_{0,n}\rightarrow u_{0}$
in $L^{2}$ as $n\rightarrow\infty$ with $u_{0,n}$ smooth then the
associated solutions satisfy 
\[
\left\Vert u_{n}-u\right\Vert _{X_{T}}\rightarrow0,\,n\rightarrow\infty.
\]
Since the $L^{2}$ norm of $u_{n}$ is preserved, i.e., $\left\Vert u_{n}\right\Vert _{L^{2}}=\left\Vert u_{0,n}\right\Vert _{L^{2}}$,
we can obtain that 
\[
\left\Vert u\right\Vert _{L_{t}^{\infty}\left(\left[0,T\right]:L_{x}^{2}\right)}=\left\Vert u_{0}\right\Vert _{L^{2}}.
\]
Then we can repeat the above construction from $T$ to $2T$ and inductively
to obtain a global-in-time solution. 
\end{proof}

We record a local $H^{1}$ well-posedness result and a global
$H^{1}$ wellposedness result for solutions with small $H^{1}$ norms.
\begin{thm}\label{thm:lwpH1}
For any $u_{0}\in H^{1}$, there  exists a local-in-time solution
to 
\[
i\partial_{t}u-\partial_{xx}u+Vu\pm\left|u\right|^{2}u=0.
\]
The solution can be extended as long as the $H^{1}$ norm is finite.
\end{thm}

\begin{proof}
Taking $\mathcal{D}=\sqrt{-\partial_{xx}+V+\left\Vert V\right\Vert _{L^{\infty}}+1}$
and by the general theory of boundedness of wave operators, one has 
\[
\left\Vert \mathcal{D}f\right\Vert _{L^{2}}\sim\left\Vert \partial_{x}f\right\Vert _{L^{2}}+\left\Vert f\right\Vert _{L^{2}}.
\]
Using the weighted norms and local improved decay, we might use the
regular differentiation $\partial_{x}$. But here we try to avoid
the weighted norms. Using $\mathcal{D}$ which commutes with $-\partial_{xx}+V$,
we have 
\[
\partial_{t}\left(\mathcal{D}u\right)-\partial_{xx}\left(\mathcal{D}u\right)+V\left(\mathcal{D}u\right)+\mathcal{D}\left(|u|^{2}u\right)=0.
\]
To estimate the $L^{2}$ norm of $\partial_{x}u$, it suffices to
estimate the $L^{2}$ norm of $\mathcal{D}u$ since by the $L^{2}$
theory, we have nice control of $u$ in the $L^{2}$ space. Using the Duhamel formula, one has 
\[
\mathcal{D}u=e^{-itH}\left(\mathcal{D}u_{0}\right)-i\int_{0}^{t}e^{-\left(t-s\right)iH}\mathcal{D}\left(\left|u\right|^{2}u\left(s\right)\right)\,ds.
\]
To estimate the $L^{2}$ norm, we put the $L^{2}$ norm on both sides
and conclude that
\[
\left\Vert \mathcal{D}u\right\Vert _{L^{2}}\leq\left\Vert \mathcal{D}u_{0}\right\Vert _{L^{2}}+\int_{0}^{t}\left\Vert \mathcal{D}\left(\left|u\right|^{2}u\left(s\right)\right)\right\Vert _{L_{x}^{2}}.
\]
Then we perform similar construction as in Theorem \ref{thm:L2existence} but now we need to estimate
for fixed $T$
\[
\int_{0}^{T}\left\Vert \mathcal{D}\left(\left|u\right|^{2}u\left(s\right)\right)\right\Vert _{L_{x}^{2}}.
\]
Notice that by the boundedness of wave operators,
\[
\int_{0}^{T}\left\Vert \mathcal{D}\left(\left|u\right|^{2}u\left(s\right)\right)\right\Vert _{L_{x}^{2}}\sim\int_{0}^{T}\left\Vert \partial_{x}\left(\left|u\right|^{2}u\left(s\right)\right)\right\Vert _{L_{x}^{2}}+\int_{0}^{T}\left\Vert \left(\left|u\right|^{2}u\left(s\right)\right)\right\Vert _{L_{x}^{2}}.
\]
The last term on the RHS, i.e.,
\[
\int_{0}^{T}\left\Vert \left(\left|u\right|^{2}u\left(s\right)\right)\right\Vert _{L_{x}^{2}}
\]
can be handled by the same way in Theorem \ref{thm:L2existence}. For the first term, we note
that
\begin{align*}
\int_{0}^{T}\left\Vert \partial_{x}\left(\left|u\right|^{2}u\left(s\right)\right)\right\Vert _{L_{x}^{2}} & \lesssim T^{\frac{1}{2}}\left\Vert u\right\Vert _{L_{t}^{4}\left(\left[0,T\right]:L_{x}^{\infty}\right)}^{2}\sup_{0\leq t\leq T}\left\Vert \partial_{x}u\right\Vert _{L_{x}^{2}}\\
 & \lesssim T^{\frac{1}{2}}\left\Vert u\right\Vert _{L_{t}^{4}\left(\left[0,T\right]:L_{x}^{\infty}\right)}^{2}\sup_{0\leq t\leq T}\left\Vert u\right\Vert _{L_{x}^{2}}\\
 & +T^{\frac{1}{2}}\left\Vert u\right\Vert _{L_{t}^{4}\left(\left[0,T\right]:L_{x}^{\infty}\right)}^{2}\sup_{0\leq t\leq T}\left\Vert Ju\right\Vert _{L_{x}^{2}}.
\end{align*}
Actually, by the Sobolev embedding in $1$D,
\[
\int_{0}^{T}\left\Vert \partial_{x}\left(\left|u\right|^{2}u\left(s\right)\right)\right\Vert _{L_{x}^{2}}\lesssim T\left(\sup_{0\leq t\leq T}\left\Vert \partial_{x}u\right\Vert _{L_{x}^{2}}\right)^{3}.
\]
 Then we can do contraction as in Theorem \ref{thm:L2existence} and the procedure can be
extended provided $\left\Vert \partial_{x}u\right\Vert _{L^{2}}$
is finite.
\end{proof}
\begin{thm}\label{thm:smallglobal}
For any $u_{0}\in H^{1}$ such that $\left\Vert u_{0}\right\Vert _{H^{1}}\ll\delta$
where $\delta$ is a small enough positive number. Then there is exists
a global-in-time solution to 
\[
i\partial_{t}u-\partial_{xx}u+Vu\pm\left|u\right|^{2}u=0.
\]
\end{thm}

\begin{proof}
By  Theorem \ref{thm:lwpH1}, we just need to show the $H^{1}$ norm of the
solution remains bounded.

Now we use the conservation of Hamiltonian
\[
H\left(u\right)=\int\left|\partial_{x}u\right|^{2}+V\left|u\right|^{2}\pm\frac{1}{4}\left|u\right|^{4}.
\]
Since $\left\Vert u\right\Vert _{H^{1}}$ is small, then by the Sobolev
embedding,
\[
\int\frac{1}{4}\left|u\right|^{4}\leq\delta^{2}\left\Vert u\right\Vert _{H^{1}}^{2}.
\]
So using the $L^{2}$ conservation, one has
\[
H\left(u\right)=\int\left|\partial_{x}u\right|^{2}+V\left|u\right|^{2}\pm\frac{1}{4}\left|u\right|^{4}\sim\left\Vert u\right\Vert _{H^{1}}^{2}.
\]
Therefore we can always apply the above Theorem \ref{thm:lwpH1} to extend the solution
and obtain a global solution.
\end{proof}
\subsection{Applications of  the global well-posedness theory}
In this part, using the information from the global existence theory above.  Our goal is to illustrate that it suffices to perform \emph{a priori} estimates in this paper.

To illustrate the idea, we still consider the model problem \eqref{eq:modelu}
\begin{equation}
i\partial_{t}u-\partial_{xx}u+Vu=a_{1}\left(x\right)u+a_{2}(x)e^{2i\int_{0}^{t}E\left[z(\sigma)\right]\,d\sigma}\bar{u}+b\left(x\right)u^{2}+\left|u\right|^{2}u
\end{equation}
with $a_1(x)$, $a_2(x)$ and $b(x)$ being smooth functions which decay exponentially such that $|a_j(x)|\lesssim\epsilon ^2 $ and $|b(x)|\lesssim \epsilon $ under the assumption that
\[
\left|\frac{d}{dt}E\left[z\left(t\right)\right]\right|\lesssim\epsilon^{2}t^{-2+2\alpha}
\]
and $\|u\|_{H^1}\lesssim \epsilon.$ The last assumption is ensured by the well-posedness theory in Theorem \ref{thm:smallglobal}.

To obtain the existence and quantitative estimates for the solution to the model problem \eqref{eq:modelu5}, we again use the fixed point theorem applied to
\begin{equation}
    \mathrm{F}(h):= e^{-itH}u_{0}-i\int_{0}^{t}e^{-\left(t-s\right)iH}\left(G(h)\left(s\right)\right)\,ds
\end{equation}
where
\begin{equation}
    G(h):=a_{1}\left(x\right)h+a_{2}(x)e^{2i\int_{0}^{t}E\left[z(\sigma)\right]\,d\sigma}\bar{h}+b\left(x\right)h^{2}+\left|h\right|^{2}h
\end{equation}
But now we will perform \emph{a priori} estimates and contractions in different spaces.

The contraction is achieved in a weaker topology
\[
Y_T:=\left\{ u |\,\left\Vert u\right\Vert _{L_{t}^{\infty}\left(\left[0,T\right]:H_{x}^{1}\right)}<\infty,\left\Vert u\right\Vert _{L_{t}^{4}\left(\left[0,T\right]:L_{x}^{\infty}\right)}<\infty\right\}
\]with the $T$ depending on the size of initial data as in Theorem \ref{thm:L2existence} and Theorem \ref{thm:lwpH1}. To extend the contraction globally, we use the fact that the $H^1$ norm of the solution is globally bounded by construction.  From the contraction in the this topology, there is a unique solution to 
\begin{equation}
    u=\mathrm{F}(u).
\end{equation}
Next, we show that the unique fixed point actually enjoys better estimates.
 
Letting $h(t,x)$ be a Schwartz function in $\mathbb{R}\times \mathbb{R}$, we decompose $h=h_{L}+h_{H}$ where
\begin{equation}
\mathcal{F}_{T}\left(h_{L}\right)\left(\tau\right)=\phi_{1}\left(\tau\right)\mathcal{F}_{T}\left(h\right)\left(\tau\right)
\end{equation}
\begin{equation}
\mathcal{F}_{T}\left(h_{H}\right)\left(\tau\right)=\phi_{2}\left(\tau\right)\mathcal{F}_{T}\left(h\right)\left(\tau\right)
\end{equation}
and
\[
\mathcal{F}_{T}\left(h\right)\left(\tau\right)=\frac{1}{\sqrt{2\pi}}\int e^{-it\tau}h\left(t\right)\,dt
\]
is the  Fourier transform with respect to time.

We define
\begin{equation}
\|h\|_{L}:=\sup_{T\in \mathbb{R}^+}\left\{T^{-\alpha}\left\Vert \left\langle x\right\rangle ^{-2}t\left(-2iE\left[z(t)\right]\partial_{t}+\partial_{t}^{2}\right)\left(h_{L}\right)\right\Vert _{L_{x}^{\infty}L_{t}^{2}\left[0,T\right]}\right\}
\end{equation}
and
\begin{equation}
\|h\|_{H}:=\sum_{j=0}^1\sup_{T\in \mathbb{R}^+}\left\{T^{-\alpha}
\left\Vert \left\langle x\right\rangle ^{-2}\partial_{x}^{j}t\left(u_{H}\right)\right\Vert _{L_{x}^{\infty}L_{t}^{2}\left[0,T\right]}\right\}
\end{equation}

Fixing an $0<\alpha\ll 1$, we do  refined estimates in
\begin{equation}
X:=\left\{ u|\,\left\Vert u\right\Vert _{L_{t}^{\infty}\left(\left[0,\infty\right];H^{1}\right)}+\left\Vert \tilde{f}\right\Vert _{L_{t}^{\infty}\left(\left[0,\infty\right];L_{k}^{\infty}\right)}+\sup_{t\geq0}\left\Vert t^{-\alpha}\tilde{f}(t)\right\Vert _{H_{k}^{1}}+\|u\|_{L}+\|u\|_{H}\right\} 
\end{equation} where $f:=e^{-iHt}u$.
Following the iteration scheme, we write
\begin{equation}
    u_{n+1}=\mathrm{F}(u_n).
\end{equation}
One can show that
\begin{equation}
    \|u_{n+1}\|_X \lesssim \|u_n\|_X\lesssim \|u_0\|_{H^{1,1}}.
\end{equation}
To show the conclusion above, the analysis is equivalent to the \emph{ a priori} estimate and the bootstrap analysis in this paper. 
From \emph{a priori} estimates, the unique solution $u$ actually is in the strong topology $X$, 
\begin{equation}
 \|u\|_X \lesssim \|u_0\|_{H^{1,1}}
\end{equation}
whence,  it enjoys the desired estimates.

%

\begin{thebibliography}{Bec1}
\bibitem{Agm} S. Agmon.  Spectral properties of Schr\"odinger
operators and scattering theory.\emph{ Ann. Scuola Norm. Sup. Pisa
Cl. Sci.} (4) 2 (1975), no.~2, 151\textendash 218.

\bibitem{Bec1} M. Beceanu,  A critical center-stable manifold
for Schr\"odinger's equation in three dimensions. \emph{Comm. Pure Appl.
Math}., no. 4, 431\textendash 507.


\bibitem{BJM16}
M. Borghese, R. Jenkins and  K. T.-R. McLaughlin, 
Long-time asymptotic behavior of the focusing nonlinear Schr\"{o}dinger
equation. \emph{Ann. Inst. H. Poincar\'e Anal. Non Lin\'eaire}{ 35 (2018), no. 4, 887--920}.

\bibitem{BJ} J. Bronski and R. Jerrard, Soliton dynamics in a potential. \emph{
Math. Res. Lett.} 7 (2000), no. 2--3, 329--342.

\bibitem{BP}V. Buslaev and G. Perel\textasciiacute man, Scattering
for the nonlinear Schr\"odinger equation: states that are close to a
soliton. (Russian)\emph{ Algebra i Analiz }4 (1992), no. 6, 63\textendash 102;
\emph{translation in St. Petersburg Math. J}. 4 (1993), no. 6, 1111\textendash 1142.

\bibitem{BP2}V. Buslaev and G. Perel\textasciiacute man, On
the stability of solitary waves for nonlinear Schr\"odinger equations.
\emph{Nonlinear evolution equations}, 75\textendash 98, Amer. Math.
Soc. Transl. Ser. 2, 164, Adv. Math. Sci., 22, \emph{Amer. Math. Soc.,
Providence}, RI, 1995.

\bibitem{BS} V. Buslaev. and C. Sulem, On asymptotic stability
of solitary waves for nonlinear Schr\"odinger equations. \emph{Ann.
Inst. H. Poincar\'e Anal. Non Lin\'eaire} 20 (2003), no. 3, 419\textendash 475. 

\bibitem{Ca}
T. Cazenave, 
Semilinear Schr\"odinger equations.
Courant Lecture Notes in Mathematics, 10. \emph{New York University, Courant Institute of Mathematical Sciences, New York; American Mathematical Society, Providence, RI,} 2003. xiv+323 pp.


\bibitem {CW}
T. Cazenave and B. Weissler, 
The Cauchy problem
for the critical nonlinear Schr\"odinger equation in $H^{s}$. \emph{Nonlinear
	Anal}. 14 (1990), no. 10, 807\textendash 836.
	
\bibitem{CP} G. Chen and F. Pusateri, The 1d nonlinear Schr\"odinger
equation with a weighted $L^{1}$ potential. arXiv:1912.10949. To
appear in \emph{Anal. PDE.}

\bibitem{CLL}
 G. Chen, J. Liu, and B. Lu, Long-time asymptotics and stability for the sine-Gordon equation, Preprint
arXiv:2009.04260.
\bibitem{Cu}

S. Cuccagna, Stabilization of solutions to nonlinear Schr\"odinger equations. \emph{Comm. Pure Appl. Math.} 54 (2001), no. 9, 1110--1145.
\bibitem{Cu1}


S. Cuccagna,  A survey on asymptotic stability of ground states of nonlinear Schr\"odinger equations. \emph{Dispersive nonlinear problems in mathematical physics}, 21--57, Quad. Mat., 15, \emph{Dept. Math., Seconda Univ. Napoli, Caserta,} 2004.

\bibitem{CuMa}
S. Cuccagna and M. Maeda, A survey on asymptotic stability of ground states of nonlinear Schr\"odinger equations II. \emph{ Discrete \& Continuous Dynamical Systems -S}, 2021, 14 (5) : 1693--1716.


\bibitem{CuM} S. Cuccagna and M. Maeda, On stability of small
solitons of the 1-D NLS with a trapping delta potential. \emph{SIAM
J. Math. Anal.} 51 (2019), no. 6, 4311\textendash 4331.

\bibitem{CuP} S. Cuccagna and D. Pelinovsky, The asymptotic
stability of solitons in the cubic NLS equation on the line. \emph{Appl.
Anal}. 93 (2014), no. 4, 791--822.

\bibitem{Dei}P. Deift,  Some open problems in random matrix theory and the theory of integrable systems. II. \emph{SIGMA Symmetry Integrability Geom. Methods Appl.} 13 (2017), Paper No. 016, 23 pp.

\bibitem{DT} P. Deift, P and E. Trubowitz, Inverse scattering
on the line. \emph{Comm. Pure Appl. Math}. 32 (1979), no. 2, 121\textendash 251. 

\bibitem{DZ} P. Deift and X. Zhou, Long-time asymptotics for
solutions of the NLS equation with initial data in a weighted Sobolev
space. Dedicated to the memory of J\"urgen K. Moser. \emph{Comm. Pure
Appl. Math}. 56 (2003), no. 8, 1029\textendash 1077.

\bibitem{DZ2}P. Deift and X. Zhou, Perturbation theory for infinite-dimensional
integrable systems on the line. A case study. \emph{Acta Math.} 188
(2002), no. 2, 163\textendash 262. 


\bibitem {Del} J.M. Delort, 
Modified scattering for odd solutions
of cubic nonlinear Schr\"odinger equations with potential
in dimension one. $<\text{hal}-01396705>$ 2016.
\bibitem{DM}
J.M. Delort and N. Masmoud, Long time dispersive estimates for a kink solution of 1D cubic wave equations. 2020.
$<\text{hal}-02862414v>$.
\bibitem {DS}
N. Dunford and J. Schwartz, 
Linear operators. Part II. Spectral theory. Selfadjoint operators in Hilbert space.
Reprint of the 1963 original. 
\emph{Wiley Classics Library}. A Wiley-Interscience Publication. John Wiley \& Sons, Inc.,
New York, 1988.


\bibitem{FGJS1}

J. Fr\"ohlich, S. Gustafson, B.L.G. Jonsson and I.M. Sigal, 
Solitary wave dynamics in an external potential. \emph{
Comm. Math. Phys.} 250 (2004), no. 3, 613--642.


\bibitem{GHS} P. Germain, Z. Hani and S. Walsh. Nonlinear resonances with a potential: multilinear estimates and an
application to NLS.\emph{ Int. Math. Res. Not.} IMRN 2015, no. 18, 8484--8544.

\bibitem{GP} P. Germain and F. Pusateri, Quadratic Klein-Gordon
equations with a potential in one dimension. arXiv:2006.15688.

\bibitem{GPR}P. Germain, F. Pusateri and F. Rousset, The nonlinear
Schr\"odinger equation with a potential. \emph{Ann. Inst. H. Poincar\'e
Anal. Non Lin\'eaire} 35 (2018), no. 6, 1477\textendash 1530. 

\bibitem{GPR2} P. Germain, F. Pusateri and F. Rousset, Asymptotic
stability of solitons for mKdV. \emph{Adv. Math.} 299 (2016), 272\textendash 330.

\bibitem{GNT} S. Gustafson, K. Nakanishi and T. Tsai, Asymptotic
stability and completeness in the energy space for nonlinear Schr\"odinger
equations with small solitary waves. \emph{Int. Math. Res. Not. }2004,
no. 66, 3559\textendash 3584. 

\bibitem{GSch} M. Goldberg and W. Schlag, Dispersive estimates
for Schr\"odinger operators in dimensions one and three. \emph{Comm.
Math. Phys}. 251 (2004), no. 1, 157\textendash 178. 




\bibitem{HN} N. Hayashi and P. Naumkin, Asymptotics for large
time, of solutions to the nonlinear Schr\"odinger and Hartree equations,
\emph{Amer. J. Math.,} 120 (1998), 369\textendash 389. 


\bibitem{HZ}
J. Holmer and  M. Zworski,
Soliton interaction with slowly varying potentials.
\emph{Int. Math. Res. Not.} IMRN 2008, no. 10, Art. ID rnn026, 36 pp.



\bibitem {Hwang}
I.L. Hwang,  The $L^{2}$ boundedness of pseudodifferential operators.
\emph{Trans. Am. Math. Soc.}, 302 (1987), no.1, 55\textendash 76. 

\bibitem{IT} M. Ifrim and D. Tataru, Global bounds for the cubic
nonlinear Schr\"odinger equation (NLS) in one space dimension. \emph{Nonlinearity}
28 (2015), no. 8, 2661\textendash 2675.

\bibitem{KP} J. Kato, J. and F. Pusateri, A new proof of long-range
scattering for critical nonlinear Schr\"odinger equations. \emph{Differential
Integral Equations }24 (2011), no. 9-10, 923\textendash 940. 

\bibitem{KK1} A. I. Komech and E. Kopylova, On asymptotic stability
of kink for relativistic Ginzburg-Landau equations. \emph{Arch. Ration.
Mech. Anal.} 202 (2011), no. 1, 213\textendash 245.


\bibitem{KMM}M. Kowalczyk, Y. Martel and C. Mu\~noz, Kink dynamics
in the $\phi^{4}$ model: asymptotic stability for odd perturbations
in the energy space. \emph{J. Amer. Math. Soc.} 30 (2017), no. 3,
769\textendash 798.

\bibitem{KMM1}M. Kowalczyk, Y. Martel and C. Mu\~noz, Soliton
dynamics for the 1D NLKG equation with symmetry and in the absence
of internal modes. arXiv:1903.12460,
\bibitem{KMMV}
M. Kowalczyk, Y. Martel, C. Mu\~noz and H. Van Den Bosch,  A sufficient condition for asymptotic stability of kinks in general (1+1)-scalar field models. \emph{Ann. PDE} 7, 10 (2021).
\bibitem{KS} J. Krieger and W. Schlag, Stable manifolds for all
monic supercritical focusing nonlinear Schr\"odinger equations in one
dimension. \emph{J. Amer. Math. Soc}. 19 (2006), no. 4, 815\textendash 920.

\bibitem{KNS} J. Krieger, K. Nakanishi and W. Schlag, Global
dynamics above the ground state energy for the one-dimensional NLKG
equation. \emph{Math. Z.} 272 (2012), no. 1-2, 297\textendash 316.

\bibitem{LLSS} H. Lindblad, J. Luhrmann, W. Schlag and A. Soffer,
On modified scattering for 1D quadratic Klein-Gordon equations with
non-generic potentials. arXiv:2012.15191.

\bibitem{LS} H. Lindblad and A. Soffer, Scattering and small
data completeness for the critical nonlinear Schr\"odinger equation.
\emph{Nonlinearity }19 (2006), no. 2, 345\textendash 353.

\bibitem{MM}
M. E. Marti\'nez,
Decay of small odd solutions for long range Schr\"odinger and Hartree equations in one dimension. \emph{Nonlinearity} 33 (2020), no. 3, 1156--1182.

\bibitem{MMS} S. Masaki, J. Murphy and J. Segata, Modified scattering
for the one-dimensional cubic NLS with a repulsive delta potential.
\emph{Int. Math. Res. Not. }IMRN 2019, no. 24, 7577\textendash 7603.

\bibitem{MMS2} S. Masaki, J. Murphy and J. Segata, Stability
of small solitary waves for the one-dimensional NLS with an attractive
delta potential. \emph{Anal. PDE} 13 (2020), no. 4, 1099\textendash 1128. 




\bibitem{Mi} T. Mizumachi, Asymptotic stability of small solitary
waves to 1D nonlinear Schr\"odinger equations with potential. \emph{J.
Math. Kyoto Univ. }48 (2008), no. 3, 471\textendash 497.


\bibitem{Miz2} T. Mizumachi, 
Asymptotic stability of small solitons for 2D nonlinear Schr\"odinger equations with potential.
\emph{J. Math. Kyoto Univ.} 47 (2007), no. 3, 599--620.



 \bibitem {N2} K. Nakanishi, Global dynamics above the first excited energy for the nonlinear Schr\"odinger equation with a potential. \emph{Comm. Math. Phys.} 354 (2017), no. 1, 161--212
\bibitem{NS}
K. Nakanishi and  W. Schlag,  Global dynamics above the ground state energy for the cubic NLS equation in 3D. \emph{Calc. Var. Partial Differential Equations} 44 (2012), no. 1-2, 1--45.

\bibitem {N} 
I.P. Naumkin, 
Sharp asymptotic behavior of solutions
for cubic nonlinear Schr\"odinger equations with a potential. 
\emph{J.Math. Phys.} 57 (2016), no. 5, 051501, 31 pp.
\bibitem{O} T. Ozawa. Long range scattering for nonlinear Schr\"odinger equations in one space dimension. \emph{Comm.
Math. Phys.} 139 (1991), no. 3, 479--493.

\bibitem{PS} F. Pusateri and A. Soffer, 
Bilinear estimates in the presence of a large potential and a critical NLS in 3d. To appear in \emph{Memoirs of the Amer. Math. Soc.}


\bibitem{Sch} W. Schlag, Dispersive estimates for Schr\"odinger
operators: a survey. \emph{Mathematical aspects of nonlinear dispersive
equations,} 255\textendash 285, Ann. of Math. Stud., 163, Princeton
Univ. Press, Princeton, NJ, 2007.





\bibitem{Sch1} W. Schlag, Stable manifolds for an orbitally
unstable nonlinear Schr\"odinger equation. \emph{Ann. of Math.} (2)
169 (2009), no. 1, 139\textendash 227.

\bibitem{Sh} J. Shatah, Normal forms and quadratic nonlinear
Klein-Gordon equations. \emph{Comm. Pure Appl. Math}. 38 (1985), no.
5, 685\textendash 696.

\bibitem{SW}A. Soffer and M. I. Weinstein, Resonances, radiation
damping and instability in Hamiltonian nonlinear wave equations. \emph{Invent.
Math}. 136 (1999), no. 1, 9\textendash 74.

\bibitem{SW1} A. Soffer and M. I. Weinstein, Selection of the
ground state for nonlinear Schr\"odinger equations. \emph{Rev. Math.
Phys}. 16 (2004), no. 8, 977\textendash 1071. 

\bibitem{SW2} A. Soffer and M. I. Weinstein, Multichannel nonlinear
scattering for nonintegrable equations. \emph{Comm. Math. Phys.} 133
(1990), no. 1, 119\textendash 146.

\bibitem{SW3} A. Soffer and M. I. Weinstein, Multichannel nonlinear
scattering for nonintegrable equations. II. The case of anisotropic
potentials and data.\emph{ J. Differential Equations} 98 (1992), no.
2, 376\textendash 390. 

\bibitem{Tao} T. Tao, \emph{Nonlinear dispersive equations.
Local and global analysis. CBMS Regional Conference Series in Mathematics},
106. Published for the Conference Board of the Mathematical Sciences,
Washington, DC; by the American Mathematical Society, Providence,
RI, 2006. xvi+373 pp. 

\bibitem{Tao2}T. Tao, Why are solitons stable? \emph{Bull.
Amer. Math. Soc.} (N.S.) 46 (2009), no. 1, 1-33.

\bibitem{TY} T. P. Tsai and H.T. Yau, Asymptotic dynamics of
nonlinear Schr\"odinger equations: resonance-dominated and dispersion-dominated
solutions. \emph{Comm. Pure Appl. Math.} 55 (2002), no. 2, 153\textendash 216.



\bibitem{Ts} T.P. Tsai, Asymptotic dynamics of nonlinear Schr\"odinger
equations with many bound states. \emph{J. Differential Equations}
192 (2003), no. 1, 225\textendash 282. 

\bibitem{Wein} M.I. Weinstein, Lyapunov stability of ground
states of nonlinear dispersive evolution equations, \emph{Comm. Pure
Appl. Math.} 39 (1986), no. 1, 51\textendash 67..
\bibitem{Wein1} 

M. I. Weinstein,  Extended Hamiltonian systems. \emph{Handbook of dynamical systems. Vol. 1B}, 1135--1153, Elsevier B. V., Amsterdam, 2006. 

\bibitem{Wed} R. Weder, The $W^{k,p}$-continuity of the Schr\"odinger
wave operators on the line. \emph{Comm. Math. Phys.} 208 (1999), no.
2, 507\textendash 520. 
\bibitem {Yaf}
D. Yafaev, 
Mathematical scattering theory. Analytic theory. 
\emph{Mathematical Surveys and Monographs}, 158. American Mathematical Society, Providence, RI, 2010. xiv+444 pp.

\end{thebibliography}
\end{document}